\documentclass{amsart}
\usepackage{amsmath}
\usepackage{amsfonts}
\usepackage{amssymb}
\usepackage{epsfig}
\newtheorem{theorem}{Theorem}[section]

\newtheorem{lm}[theorem]{Lemma}
\newtheorem{tr}[theorem]{Theorem}
\newtheorem{cor}[theorem]{Corollary}

\newtheorem{rem}[theorem]{Remark}
\newtheorem{pr}[theorem]{Proposition}

\usepackage{color}

\newtheorem{question}[theorem]{Question}

\begin{document}
	
	\title{Automorphism group functors of algebraic superschemes}
	\author{}
	\address{}
	\email{}
	\author{A. N. Zubkov}
	\address{Department of Mathematical Science, UAEU, Al-Ain, United Arabic Emirates; Sobolev Institute of Mathematics, Omsk Branch, Pevtzova 13, 644043 Omsk, Russian Federation}
	\email{a.zubkov@yahoo.com}
	\begin{abstract}
	The famous theorem of Matsumura-Oort states that if $X$ is a proper scheme, then the automorphism group functor $\mathfrak{Aut}(X)$ of $X$ is a locally algebraic group scheme.
	In this paper we generalize this theorem to the category of superschemes, that is if $\mathbb{X}$ is a proper superscheme, then the automorphism group functor $\mathfrak{Aut}(\mathbb{X})$ of $\mathbb{X}$ is a locally algebraic group superscheme. Moreover, we also show that if $H^1(X, \mathcal{T}_X)=0$, where $X$ is the geometric counterpart of $\mathbb{X}$ and $\mathcal{T}_X$ is the tangent sheaf of $X$, then $\mathfrak{Aut}(\mathbb{X})$ is a smooth group superscheme.		
	\end{abstract}
	\maketitle
	
	\section*{introduction}
	
	The purpose of this paper is to generalize the fundamental result from \cite{mats-oort} that automorphism group functors of proper schemes are representable by locally algebraic group schemes to superschemes. A close result has been recently proven in \cite{brp}. More precisely, they showed that if $X$ is superscheme that is superprojective and flat over a Noetherian superscheme $Y$, then the automorphism group functor of $X$ over $Y$ is representable by a group superscheme over $Y$. This result is a super analog of the well-known theorem of Grothendieck (cf. \cite{gr}) and their approach is parallel to the Grothendieck's one, i.e. it is based on the proof of the existence of Hilbert superscheme in the given setting. Note that both Grothendieck's theorem and its superized version from \cite{brp} do not imply either the main result of \cite{mats-oort} or our theorem respectively, since a proper (super)scheme is not necessary (super)projective. In addition, the paper \cite{brp} covers a much wider range of problems than just the representability of the automorphism group functor of (superprojective) superscheme.
	
	In this paper, all superschemes are defined over a field $\Bbbk$ of zero or odd characteristic. All (associative and unital) superalgebras are assumed to be super-commutative, unless stated otherwise. The category of superalgebras with graded morphisms is denoted by $\mathsf{SAlg}_{\Bbbk}$. Its full subcategory consisting of purely even (super)algebras is denoted by $\mathsf{Alg}_{\Bbbk}$.
	
	Let $X$ be a superscheme. Let $\mathbb{X}$ be a $\Bbbk$-functor of points of $X$, 
	i.e. $\mathbb{X}(A)=\mathrm{Mor}(\mathrm{SSpec}(A), X)$ for arbitrary $A\in\mathsf{SAlg}_{\Bbbk}$. Then the automorphism group functor $\mathfrak{Aut}(\mathbb{X})$ can be defined by two equivalent ways. The first one is to determine
	\[\mathfrak{Aut}(\mathbb{X})(A)=\{\mbox{invertible natural transformations of the functor} \ \mathbb{X}_A \},\] 
	where $\mathbb{X}_A$ is the restriction of $\mathbb{X}$ on the subcategory of $\mathsf{SAlg}_{\Bbbk}$ that consists of $A$-superalgebras with graded $A$-linear morphisms (cf. \cite{jan}). The second one is to determine 
	\[\mathfrak{Aut}(\mathbb{X})(A)=\{\mbox{invertible endomorphisms of the superscheme} \ X\times\mathrm{SSpec}(A) \] \[\mbox{over the affine superscheme} \ \mathrm{SSpec}(A)\},\]
	similarly to \cite{mats-oort}. 
	
Our approach to the problem of representability of 	$\mathfrak{Aut}(\mathbb{X})$ is inspired by the description of locally algebraic group superschemes, that was recently obtained in \cite{maszub2}. We will briefly recall the main points of this description. If $\mathbb{G}$ is a locally algebraic group superscheme, then $\mathbb{G}=\mathbb{G}_{ev}\mathbb{N}_e(\mathbb{G})$, where $\mathbb{G}_{ev}$ is the largest purely even group supersubscheme of $\mathbb{G}$, and $\mathbb{N}_e(\mathbb{G})$ is a normal group subfunctor of $\mathbb{G}$, called the \emph{formal neighborhhod of the identity}, which is strictly pro-representable by a complete local Noetherian Hopf superalgebra. In the nineth section we give an equivalent description of the group subfunctor $\mathbb{N}_e(\mathbb{G})$ as :
	\[\mathbb{N}_e(\mathbb{G})(A)=\ker(\mathbb{G}(A)\to\mathbb{G}(\widetilde{A}))\] 
	for arbitrary superalgebra $A$, where $\widetilde{A}=A/\mathsf{nil}(A)$ (see Proposition \ref{final about neighborhood} and Corollary \ref{normality of N_e} below). As a $\Bbbk$-functor, $\mathbb{N}_e(\mathbb{G})$ is isomorphic to $(\mathbb{N}_e(\mathbb{G})\cap \mathbb{G}_{ev})\times \mathrm{SSp}(\Lambda(V))$, where $V=\mathfrak{g}_1^*$ is dual to the odd component of the Lie superalgebra $\mathfrak{g}$ of $\mathbb{G}$.
	
	In turn, $\mathbb{G}$ is isomorphic to $\mathbb{G}_{ev}\times \mathrm{SSp}(\Lambda(V))$ as a superscheme. The group structure of $\mathbb{G}$ is uniquely defined by the adjoint action of $\mathbb{G}_{ev}$ on $\mathfrak{g}$, whence on $\mathfrak{g}_1^*$, and by a list of
	commutator relations on certain "generators" of $\mathrm{SSpec}(\Lambda(V))$ (see \cite[Section 12]{maszub2} for more details).
	
	Returning to the group functor $\mathfrak{Aut}(\mathbb{X})$, let $\mathfrak{Aut}(\mathbb{X})_{ev}$ denote the largest purely even group subfunctor of $\mathfrak{Aut}(\mathbb{X})$, i.e. $\mathfrak{Aut}(\mathbb{X})_{ev}(A)=\mathfrak{Aut}(\mathbb{X})(A_0)$ for any superalgebra $A$. If $X$ is proper, then using the representability criteria from \cite{mats-oort} we show that
	$\mathfrak{Aut}(\mathbb{X})_{ev}$ is a locally algebraic group scheme. Note that there is a natural morphism $\mathfrak{Aut}(\mathbb{X})_{ev}\to \mathfrak{Aut}(\mathbb{X}_{ev})$ of group schemes, that is not isomorphism in general. We prove that the kernel of this morphism is a nilpotent locally algebraic group scheme and its connected component is an unipotent algebraic group.

	Further, one can define the formal neighborhood of the identity in $\mathfrak{Aut}(\mathbb{X})$ as (see Proposition \ref{nice property of T}) :
	\[\mathfrak{T}(A)=\ker(\mathfrak{Aut}(\mathbb{X})(A)\to \mathfrak{Aut}(\mathbb{X})(\widetilde{A})), A\in\mathsf{SAlg}_{\Bbbk}.\]
	We show that $\mathfrak{T}$ is strictly pro-representable by  a complete local Noetherian Hopf superalgebra, provided $X$ is proper.
	 
	We describe the structure of arbitrary strictly pro-representable group $\Bbbk$-functor. Namely, if $\mathbb{X}$ is such a functor, then we show that
	$\mathbb{X}\simeq\mathbb{X}_{ev}\times\mathrm{SSpec}(\Lambda(V))$, where the space $V$ is dual to the odd component of Lie superalgebra of $\mathbb{X}$. The group structure
	on this direct product of functors is uniquely defined by the conjugation action of $\mathbb{X}_{ev}$ on $V$ and by the commutator relations on the certain "generators"
	of $\mathrm{SSpec}(\Lambda(V))$.
	
	Further, let $\mathbb{G}$ be a group $\Bbbk$-functor. Similarly to the above, one can determine the normal group subfunctor $\mathfrak{T}$ of $\mathbb{G}$, such that $\mathfrak{T}(A)=\ker(\mathbb{G}(\pi_A)), A\in\mathsf{SAlg}_{\Bbbk}$. It is clear that $\mathbb{G}_{ev}\mathfrak{T}$ is a group subfunctor of $\mathbb{G}$. 
	Moreover, if $\mathfrak{T}$ is strictly pro-representable, then using the decomposition $\mathfrak{T}\simeq \mathfrak{T}_{ev}\times\mathrm{SSpec}(\Lambda(V))$ one sees that $\mathbb{G}_{ev}\mathfrak{T}\simeq\mathbb{G}_{ev}\times \mathrm{SSpec}(\Lambda(V))$.  In particular, $\mathbb{G}_{ev}\mathfrak{T}$ is a locally algebraic group superscheme if and only if $\mathbb{G}_{ev}$ is a locally algebraic group scheme. Furthermore, \cite[Theorem 12.5]{maszub2} immediately implies that if $\mathbb{G}_{ev}$ is a locally algebraic group scheme, then $\mathbb{G}_{ev}\mathfrak{T}$ is the largest group subfunctor of $\mathbb{G}$ that is a locally algebaic group superscheme as well. 
	
	Summarizing all of the above, $\mathbb{G}$ is a locally algebraic group superscheme, whenever $\mathbb{G}_{ev}$ is a locally algebraic group scheme, $\mathfrak{T}$ is strictly pro-representable and $\mathbb{G}=\mathbb{G}_{ev}\mathfrak{T}$. Superizing \cite[Lemma (1.2)]{mats-oort}, we show in Proposition \ref{super criterion} that all these conditions are realized if and only if
	$\mathbb{G}_{ev}$ is a locally algebraic group scheme and $\mathbb{G}$ satisfies the super analogs of the conditions $P_i, 1\leq i\leq 5$, from \cite{mats-oort}.
	As a consequence, we immediately derive that $\mathfrak{Aut}(\mathbb{X})$ is a locally algebraic group superscheme, whenever $X$ is proper.
		
	The paper is organized as follows. In the first section we remind the definitions of superschemes as $\Bbbk$-functors and as geometric superspaces. Due \emph{Comparison Theorem} (see \cite[Theorem 5.14]{maszub1}, or \cite[I, \S 1, 4.4]{gabdem}) these categories are equivalent each to other, and we freely use them throughout. 
	
	In the second section we remind the definitions 
	and the standard properties of separated, proper and smooth morphisms of superschemes. In the third, fourth and fifth section we develope a fragment of the theory of Noetherian formal superschemes. In most cases the proofs are similar to the purely even case, but in some cases the so-called \emph{even reducibility} is used. For example, one can prove that 
	a Noetherian formal superscheme $\mathfrak{X}$ is proper over another Noetherian formal superscheme $\mathfrak{Y}$ if and only if the Noetherian formal scheme $\mathfrak{X}_0$ is proper over	the Noetherian formal scheme $\mathfrak{Y}_0$ if and only if the Noetherian formal scheme $\mathfrak{X}_{ev}$ is proper over the Noetherian formal scheme $\mathfrak{Y}_{ev}$. This reduction of the properties of superschemes to the properties of schemes allows the use of theorems from algebraic geometry to prove their super analogues.
	
	In the sixth section we prove some auxiliary properties of $\mathcal{O}_X$-supermodules, those will be needed later. In the seventh section we introduce a superized version of \v{C}ech cohomologies and prove some standrad theorems about. In the eighth section we recall Levelt's theorem on pro-representable functors and apply it to the category of superalgebras. In the ninth section we recall the definition of \emph{formal neighborhood} $\mathbb{N}_x(\mathbb{X})$ of a $\Bbbk$-point $x$ of superscheme $\mathbb{X}$, regarded as a $\Bbbk$-subfunctor of
	$\mathbb{X}$. It is nothing else but the functor of points of the formal completion of $X$ along the closed supersubscheme $Y\simeq\mathrm{SSpec}(\Bbbk)$, such that $|Y|=\{x\}$.
	
	The tenth section is devoted to a (probably) folklore result that the automorphism group functor of functor on site, is again a functor on the same site.  
	
	In the eleventh section we start to study the properties of the automorphism group functor $\mathfrak{Aut}(\mathbb{X})$ of superscheme $\mathbb{X}$. In particular, we prove that
	$\mathfrak{Aut}(\mathbb{X})_{ev}$ is a group subfunctor of $\mathfrak{Aut}(\mathbb{X})$, and we obtain a preliminary description of the formal neighborhood of the identity $\mathfrak{T}$ of $\mathfrak{Aut}(\mathbb{X})$.
	
	Our strategy is to prove that $\mathfrak{Aut}(\mathbb{X})$ satisfies the super analogs of the conditions $P_1-P_5$ from \cite{mats-oort}, regarded as super-$P_i, 1\leq i\leq 5$.
	In particular, if they take place, then $\mathfrak{Aut}(\mathbb{X})_{ev}$ satisfies $P_1-P_5$. Furthermore, using an elementary trick we show that $\mathfrak{Aut}(\mathbb{X})_{ev}$ satisfies $P_{orb}$, hence it is a locally algebraic group scheme by \cite[Theorem 3.7]{mats-oort}. This is realized in the sections 12, 13 and 14. 
	
	In the fifteenth section we describe the structure of $\mathfrak{Aut}(\mathbb{X})_{ev}$. More precisely, there is the natural morphism $\mathfrak{Aut}(\mathbb{X})_{ev}\to \mathfrak{Aut}(\mathbb{X}_{ev})$ of group schemes and we show that its kernel $\mathfrak{N}$ contains almost unipotent, locally algebraic group scheme $\mathfrak{R}$. The word "almost" means that its connected component $\mathfrak{R}^0$  is unipotent, but the etale group scheme $\mathfrak{R}/\mathfrak{R}^0$ is not necessary finite. Moreover, $\mathfrak{N}/\mathfrak{R}$ is embedded into $\mathbb{GL}_{\mathcal{O}_{X_{ev}}}(J_X/J_X^2)$ as a group functor. Here, for arbitrary scheme $M$ and a coherent $\mathcal{O}_M$-module $\mathcal{F}$, $\mathbb{GL}_{\mathcal{O}_M}(\mathcal{F})$ denotes the group functor $A\mapsto (\mathrm{End}_{\mathcal{O}_M}(\mathcal{F})(|M|)\otimes A)^{\times}$, $A\in\mathsf{Alg}_{\Bbbk}$.
	
	In sexteenth section the structure of arbitrary pro-representable group functor $\mathbb{X}$ is described. Indeed, we prove that $\mathbb{X}\simeq\mathbb{X}_{ev}\times\mathrm{SSpec}(\Lambda(V))$, where $\mathbb{X}_{ev}$ acts on $\mathrm{SSpec}(\Lambda(V))$ by conjugations, and we uniquely determine its group structure by the list of commutator relations on certain "generators" of  $\mathrm{SSpec}(\Lambda(V))$. 
	
	Using the results of sexteenth section, in the seventeenth section we show that $\mathbb{G}$ is a locally algebraic group superscheme if and only if 
	$\mathbb{G}_{ev}$ is a locally algebraic group scheme  and $\mathbb{G}$ satisfies the conditions super-$P_i, 1\leq i\leq 5$. Since in the previous sections it has been proved
	that the group functor $\mathfrak{Aut}(\mathbb{X})$ satisfies all these conditions, hence it is a locally algebraic group superscheme. 
		
	In the last section we prove that if $X$ is proper and $\mathrm{H}^1(X, \mathcal{T}_X)=0$, where $\mathcal{T}_X$ is the \emph{tangent sheaf} of $X$, then the group superscheme $\mathfrak{Aut}(\mathbb{X})$ is smooth. The proof is based on the superized version of  \cite[Proposition III.5.1]{sga 1} and the results of seventh section. 
	
	\section{Superschemes and geometric superschemes}
	
	For the content of this section we refer to \cite{brp, CCF, jan, Man, maszub1, zub1, zub2}.
	
	Let $\Bbbk$ be a field of zero or odd characteristic. Remind that $\mathsf{SAlg}_{\Bbbk}$ denotes the category of supercommutative $\Bbbk$-superalgebras with graded morphisms. Let $\mathsf{Alg}_{\Bbbk}$ denote the full subcategory of  $\mathsf{SAlg}_{\Bbbk}$ consisting of purely even superalgebras. 
	
	If $A\in\mathsf{SAlg}_{\Bbbk}$, then the superideal $AA_1$ is denoted by $J_A$. The algebra $A/J_A\simeq A_0/A_1^2$  is denoted by $\overline{A}$. Any \emph{prime} (\emph{maximal}) superideal of $A$ has a form $\mathfrak{P}=\mathfrak{p}\oplus A_1$,
	where $\mathfrak{p}$ is a prime (respectively, maximal) ideal of $A_0$.
	
	Let $\mathsf{gr}(A)$ denote the \emph{graded} superalgebra $\oplus_{n\geq 0} J_A^n/J_A^{n+1}$. Then $0$-th and $1$-th components of $\mathsf{gr}(A)$ are
	$\overline{A}$ and $J_A/J_A^2\simeq A_1/A_1^3$ respectively. 
	
	If $A$ is a local superalgebra with the maximal superideal $\mathfrak{M}$, then its \emph{residue field} $A/\mathfrak{M}\simeq A_0/\mathfrak{m}$ is denoted by $\Bbbk(A)$. 
	
	A $\Bbbk$-functor $\mathbb{X}$ is a functor from the category $\mathsf{SAlg}_{\Bbbk}$ to the category $\mathsf{Sets}$. Let $\mathcal{F}$ denote the category of $\Bbbk$-functors. The morphisms in $\mathcal{F}$ are denoted by bold letters $\bf f, g, \ldots $.
	A $\Bbbk$-functor $\mathbb{X}$, that is representable by a superalgebra $A$, is called an \emph{affine superscheme}. We denote $\mathbb{X}$ by
	$\mathrm{SSp}(A)$, so that
	\[ \mathrm{SSp}(A)(B)=\mathrm{Hom}_{\mathsf{SAlg}_{\Bbbk}}(A, B), B\in \mathsf{Alg}_{\Bbbk}. \]
	A \emph{closed supersubscheme} $\mathbb{Y}$ of $\mathrm{SSp}(A)$ is defined as
	\[\mathbb{Y}(B)=\{ \phi\in \mathrm{SSp}(A)(B)\mid \phi(I)=0\}, B\in\mathsf{SAlg}_{\Bbbk},  \]
	where $I$ is a superideal of $A$. It is clear that $\mathbb{Y}\simeq\mathrm{SSp}(A/I)$ is again affine superscheme.
	Similarly, an \emph{open supersubscheme} $\mathbb{U}$ of $\mathrm{SSp}(A)$ is defined as 
	\[\mathbb{U}(B)= \{ \phi\in \mathrm{SSp}(A)(B)\mid \phi(I)B=B\}, B\in\mathsf{SAlg}_{\Bbbk}, \] 
	where $I$ is a superideal of $A$. Note that if $I=Aa, a\in A_0$, then the corresponding open supersubscheme is isomorphic to
	$\mathrm{SSp}(A_a)$, and it is called a \emph{principal open supersubscheme}. More generally, a subfunctor $\mathbb{Y}$ of $\Bbbk$-functor $\mathbb{X}$ is called closed (open) if for any
	morphism ${\bf f} : \mathrm{SSp}(A)\to \mathbb{X}$, the subfunctor ${\bf f}^{-1}(\mathbb{Y})$ is closed (open) in $\mathrm{SSp}(A)$.
	
	Finally, a collection of open subfunctors $\mathbb{Y}_i$ of $\mathbb{X}$ form an \emph{open covering} (of $\mathbb{X}$) if for any field extension 
	$\Bbbk\subseteq L$ there is $\mathbb{X}(L)=\cup_{i\in I} \mathbb{Y}_i(L)$.
	
	Let $\mathcal{A}$ denote the full subcategory of $\mathcal{F}$ consisting of affine superschemes. Note that $\mathsf{SAlg}_{\Bbbk}^{op}$ is naturally isomorphic to $\mathcal{A}$. 
	
	The category $\mathcal{A}$
	can be regarded as a site (see Section 6 below, or \cite[Definition 2.24]{fundalg}) with respect to the topology of \emph{local coverings}
	\[ \{\mathrm{SSp}(A_{a_i})\to\mathrm{SSp}(A)\}, \sum_i A_0 a_i=A_0,  \]
	and to the topology of \emph{fpqc coverings}
	\[ \{\mathrm{SSp}(A_i)\to\mathrm{SSp}(A)\}, \ \prod_i A_i \ \mbox{is a faithfully flat} \ A-\mbox{supermodule}.  \]
	A $\Bbbk$-functor $\mathbb{X}$ is called \emph{local} if it is a sheaf on $\mathsf{SAlg}_{\Bbbk}\simeq\mathcal{A}^{op}$ with respect to the topology of local coverings. A local $\Bbbk$-functor $\mathbb{X}$ is said to be a \emph{superscheme}, provided $\mathbb{X}$ has an open covering by affine supersubschemes. The obvious examples of superschemes are affine superschemes and their open supersubschemes.
	
	Note that any superscheme is also a sheaf on  $\mathsf{SAlg}_{\Bbbk}\simeq\mathcal{A}^{op}$ with respect to the topology of fpqc coverings.
	
	The category of superschemes is equivalent to the category of \emph{geometric superschemes}. More precisely, a \emph{geometric superspace} $X$ consists of a topological space $|X|$ and a sheaf of super-commutative superalgebras $\mathcal{O}_X$ on $|X|$, such that all stalks $\mathcal{O}_{X, x}$ for $x\in |X|$ are local superalgebras. A morphism of superspaces $f : X\to Y$ is a pair $(|f|, f^{\sharp})$, where $|f|: |X|\to |Y|$ is a morphism of topological spaces and $f^{\sharp} : \mathcal{O}_Y\to |f|_*\mathcal{O}_X$ is a morphism of sheaves such that the induced morphism of stalks $f^{\sharp}_x : \mathcal{O}_{Y,|f|(x)} \to \mathcal{O}_{X, x}$ is local for any $x\in |X|$. 
	
	In what follows, the residue field $\Bbbk(\mathcal{O}_{X, x})$ is denoted by $\Bbbk(x), x\in |X|$. For a field extension $\Bbbk\subseteq L$ we say that a point $x\in |X|$ is an
	\emph{$L$-point}, if $\Bbbk(x)\simeq L$.
	
	A geometric superspace $X$ is called a \emph{geometric superscheme} if it can be covered by open supersubspaces, each of wich is isomorphic to an \emph{affine geometric superscheme} $\mathrm{SSpec}(A)$. Recall that the underlying topological space of $\mathrm{SSpec}(A)$ coincides with the prime spectrum of $A_0$ and for any open subset $U\subseteq |\mathrm{Sp}(A_0)|$, the super-ring $\mathcal{O}_{\mathrm{SSpec}(A)}(U)$ consists of all locally constant functions $h : U\to \sqcup_{\mathfrak{p}\in U} A_{\mathfrak{p}}$ such that $h(\mathfrak{p})\in A_{\mathfrak{p}}, \mathfrak{p}\in U$. 
	
	A morphism $f : X\to Y$ of geometric superschemes is called \emph{open (closed) immersion}, if $|f|$ is a homeomorphism of $|X|$ onto an open subset $U\subseteq |Y|$ (respectively, a homeomorphism onto a closed subset $Z\subseteq |Y|$), such that $f^{\sharp}$ is an isomorphism of sheaves (respectively, $f^{\sharp}_x$ is surjective for any $x\in |X|$). A composition of closed and open immersions is called just \emph{immersion}.      
	
	The equivalence of categories is given by $X\mapsto\mathbb{X}$, where for any superalgebra $A$ there is $\mathbb{X}(A)=\mathrm{Mor}(\mathrm{SSpec}(A), X)$. This equivalence takes open (closed)
	immersions in the category of geometric superschemes to the open (closed) embeddings in the category of superschemes. 
	
	In what follows we use the term \emph{superscheme} for both geometric superschemes and just superschemes, but we distinguish them by notations, i.e. geometric superschemes and their morphisms are denoted by $X, Y, \ldots $, and $f, g, \ldots$ , and just superschemes and their morphisms (natural transformations of $\Bbbk$-functors) are denoted by $\mathbb{X}, \mathbb{Y}, \ldots$, and ${\bf f}, {\bf g}, \ldots$, respectively.
	
	Let $\mathcal{P}$ be a property of geometric superscheme or a property of morphisms of geometric superschemes, such that for any $X\simeq X'$ or any commutative diagram
	\[\begin{array}{ccc}
		X & \stackrel{f}{\to} & Y \\
		\downarrow & & \downarrow \\
		X' & \stackrel{f'}{\to} & Y'
	\end{array},\]
	whose vertical arrows are isomorphisms, $X$ or $f$ satisfy $\mathcal{P}$ if and only if $X'$ or $f'$ do.
	Due the above equivalence, the property $\mathcal{P}$ can be translated to to category of just superschemes (and vice versa). In other words,  
	a superscheme $\mathbb{X}$ or a morphism ${\bf f} : \mathbb{X}\to\mathbb{Y}$ satisfy $\mathcal{P}$, provided the corresponding $X$ or $f : X\to Y$ do. 
	
	Any superscheme $X$ contains a largest purely even, closed supersubscheme $X_{ev}$, such that $|X_{ev}|=|X|$ and $\mathcal{O}_{X_{ev}}=\mathcal{O}_X/\mathcal{J}_X$, where $\mathcal{J}_X$ is the superideal sheaf of $\mathcal{O}_X$, generated by $(\mathcal{O}_X)_1$. The corresponding closed supersubscheme of $\mathbb{X}$ can be defined as $\mathbb{X}_{ev}(A)=\mathbb{X}(i)(A_0)$, where $A\in\mathsf{SAlg}_{\Bbbk}$ and $i : A_0\to A$ is the natural algebra embedding.  Similarly, one can define a purely even superscheme $X_0$, such that $|X_0|=|X|$ and $\mathcal{O}_{X_0}=(\mathcal{O}_X)_0$. It is clear that there is a natural morphism $X\to X_0$ of geometric superschemes, such that any morphism
	from $X$ to a purely even superscheme $Y$ factors through $X\to X_0$. Translating this property to the category of superschemes, we define the superscheme 
	$\mathbb{X}_0$.
	
	Recall that a morphism $f : X\to Y$ is said to be \emph{locally of finite type}, if there is a covering of $Y$ by open supersubschemes $V_i\simeq \mathrm{SSpec}(B_i)$ such that for every index $i$, the open supersubscheme $f^{-1}(V_i)$ is covered by open supersubschemes $U_{ij}\simeq \mathrm{SSpec}(A_{ij})$, where each $A_{ij}$ is a finitely generated $B_i$-superalgebra.
	If each $f^{-1}(V_i)$ can be covered by a finite number of $U_{ij}$, then $f$ is said to be a morphism of finite type (cf. \cite[II.3]{hart}). 
	For example, if $Y=\mathrm{SSpec}(\Bbbk)$, then $X$ is called a \emph{locally algebraic} or \emph{algebraic} superscheme respectively. Observe that any algebraic superscheme is \emph{Noetherian} (cf. \cite{hart, zub2}).	
	
	Finally, if $S$ is a superscheme and $X\to S$ is a morphism of superschemes, then $X$ is said to be an \emph{$S$-superscheme}, or \emph{superscheme over} $S$.
	Let $Y$ be another $S$-superscheme. Then a morphism $X\to Y$ is called an \emph{$S$-morphism}, if the diagram
	\[\begin{array}{ccc}
		X &  \to  & Y \\
		\searrow & & \swarrow \\
		& S &
	\end{array}\]  
	is commutative.	
	
	\section{Proper and smooth morphisms of superschemes}
	
	For the content of this section we refer to \cite[A.5, A.6]{brp},  \cite {maszub2}, \cite[Appendix A]{maszub3} and \cite{maszub4, zub2}. 
	
	Recall that a superscheme morphism $f : X\to Y$ is said to be \emph{separated}, if the diagonal morphism $\delta_f : X\to X\times_Y X$ is a closed immersion. Respectively, a superscheme is called \emph{separated}, if the canonical morphism $X\to\mathrm{SSpec}(\Bbbk)$ is separated. Both properties are \emph{even reducible}, i.e. $f : X\to Y$ or $X$ are separated
	if and only if the induced scheme morphism $f_{ev} : X_{ev}\to Y_{ev}$ or the scheme $X_{ev}$ are (see \cite[Proposition A.13]{brp}, \cite[Proposition 2.4]{maszub2} or Lemma \ref{reducibility} below). 
	
	The following property of superschemes, which are separated over an affine superscheme, will be in use throughout.
	\begin{lm}\label{Ex.II.4.3}(see \cite[Exercise II.4.3]{hart})
		Let $X$ be a superscheme, separated over an affine superscheme $S$. Then for any open affine supersubschemes $U$ and $V$ in $X$, the superscheme $U\cap V$ is affine as well.	
	\end{lm}
	\begin{proof}
		Let $\mathrm{pr}_1$ and $\mathrm{pr}_2$ denote the canonical projections of $X\times_S X$ to the first and second factors respectively. 
		Since $X$ is identified with a closed supersubscheme in $X\times_S X$, we have \[U\cap V=X\cap  \mathrm{pr}_1^{-1}(U)\cap\mathrm{pr}_2(V),\] and $\mathrm{pr}_1^{-1}(U)\cap\mathrm{pr}_2(V)\simeq U\times_S V$ is an open affine supersubscheme of $X\times_S X$. 	
	\end{proof}
	Further, a morphism $f : X\to Y$ is said to be \emph{proper}, if it is of finite type, separated and universally closed. Respectively, a superscheme $X$ is caled \emph{proper}, if the canonical morphism $X\to\mathrm{SSpec}(\Bbbk)$ is proper. 

	The \emph{Valuative Criterion of Separatedness and Properness} (see \cite[Theorem II.4.3 and Theorem II.4.7]{hart}) are still valid for superschemes (cf. \cite[Coollary A.14]{brp}). 
	\begin{lm}\label{reducibility}
Assume that $Z$ and $T$ are closed supersubschemes of $X$ and $Y$ respectively, their defining superideal sheaves are locally nilpotent, and $f$ takes $Z$ to $T$. Then $X$ is separated/proper over $Y$ if and only if $Z$ is separated/proper over $T$. In particular, $X$ is separated/proper over $Y$ if and only $X_{ev}$ is separated/proper over $Y_{ev}$ (as a scheme) if and only if $X_0$ is separated/proper over $Y_0$ (as a scheme as well).
	\end{lm}
\begin{proof}
Just observe that any morphism $\mathrm{SSpec}(D)\to X$ (or $\mathrm{SSpec}(D)\to Y$), where $D$ is a purely even integral domain, factors through the closed immersion $Z\to X$ (respectively, through $T\to Y$).	
\end{proof}	
	We say that a morphism $f : X\to Y$, locally of finite type,  is \emph{smooth}, whenever it is \emph{formally smooth}. The latter means that for any affine superscheme $\mathrm{SSpec}(A)$ over $Y$ and for any nilpotent superideal $I$ in $A$, the natural map $\mathbb{X}_{\mathbb{Y}}(A)\to \mathbb{X}_{\mathbb{Y}}(A/I)$ is surjective, where
	$\mathbb{X}_{\mathbb{Y}}(R)$ denote the set of all $Y$-morphism $\mathrm{SSpec}(R)\to X$. 
	
	Respectively, a locally algebraic superscheme $X$ is smooth, if the  canonical morphism $X\to\mathrm{SSpec}(\Bbbk)$ is smooth (for more details, see \cite[Appendix A]{maszub3} or \cite[A.6]{brp}).  
	
	With any superscheme $X$ one can associate a superscheme $\mathsf{gr}(X)$ such that $|\mathsf{gr}(X)|=|X|$ and $\mathcal{O}_{\mathsf{gr}(X)}$ is isomorphic to the sheafification 
	of the presheaf
	\[U\mapsto \oplus_{n\geq 0}\mathcal{J}_X(U)^n/\mathcal{J}_X(U)^{n+1}, \ U\subseteq |X|. \]
	Obviously, $X\mapsto\mathsf{gr}(X)$ is an endofunctor of the category of locally algebraic superschemes.	
	Note also that if $X\simeq\mathrm{SSpec}(A)$, then $\mathsf{gr}(X)\simeq\mathrm{SSpec}(\mathsf{gr}(A))$. Furthermore, for a superalgebra morphism $\phi : A\to B$ the induced superscheme morphism $\mathsf{gr}(\mathrm{SSpec}(\phi))$ is naturally identified with $\mathrm{SSpec}(\mathsf{gr}(\phi))$.

	As in \cite{zubbov} we call a superscheme $X$ \emph{graded}, provided $X\simeq\mathsf{gr}(Y)$ for some superscheme $Y$. Note that $X$ is Noetherian if and only if $\mathsf{gr}(X)$ is. Moreover, in the latter case $\mathcal{J}_X^t=0$ for sufficiently large $t$, hence $\mathcal{O}_{\mathsf{gr}(X)}=\oplus_{n\geq 0}\mathcal{J}_X^n/\mathcal{J}_X^{n+1}$. 
	In general, $\mathcal{O}_{\mathsf{gr}(X)}$ is a certain subsheaf of the sheaf $\prod_{n\geq 0}\mathcal{J}_X^n/\mathcal{J}_X^{n+1}$ that is not equal to $\oplus_{n\geq 0}\mathcal{J}_X^n/\mathcal{J}_X^{n+1}$. 
	
	\section{Noetherian formal superschemes}
	
	For the content of this section we refer to \cite[Section 0.7 and Section I.10]{EGA I}, \cite[Section III.5]{EGA  III} and \cite[Section II.9]{hart}.
	
	Let $X$ be a Noetherian superscheme and $Y$ be a closed supersubscheme of $X$, defined by a superideal sheaf $\mathcal{I}_Y$. Recall that the \emph{formal completion of $X$ along $Y$}  
	is defined as the geometric superspace $\widehat{X}=(|Y|, \varprojlim_{n\geq 0} \mathcal{O}_X/\mathcal{I}_Y^{n+1})$. Note that $\widehat{X}$ is covered by open supersubspaces $\widehat{U}$ (formal completion of $U$ along $Y\cap U$), where $U$ runs over an open covering of $X$. If $X=\mathrm{SSpec}(A)$ and $Y=\mathrm{SSpec}(A/I)$, then $\widehat{X}$ is denoted by $\mathrm{SSpecf}(\widehat{A})$, where $\widehat{A}=\varprojlim_n A/I^n$.  The topological space  $|\mathrm{SSpecf}(\widehat{A})|$ can be identified with the set of open prime superideals of $\widehat{A}$. Indeed, an open
	prime superideal of $\widehat{A}$ contains $\widehat{I}=\widehat{A} I$. Since $\widehat{A}/\widehat{I}\simeq A/I$, it has a form $\widehat{\mathfrak{P}}=\widehat{A}\mathfrak{P}$, where $\mathfrak{P}\in |Y|$ (cf. \cite[Lemma 1.7]{maszub4}). Since $\widehat{\mathfrak{P}}\cap A=\mathfrak{P}$, the map $\mathfrak{P}\to\widehat{\mathfrak{P}}$ is bijective. 
	Note also that $\mathcal{O}(\mathrm{SSpecf}(\widehat{A}))\simeq\widehat{A}$. 
	
	Let $B$ be a Noetherian superalgebra and $J$ be a superideal of $B$. If $B$ is Hausdorff and complete in the $J$-adic topology, then we say that $B$ is \emph{$J$-adic}, or just \emph{adic}, superalgebra (cf. \cite[Definition 0.7.1.9]{EGA I}). The superideal $J$ is said to be a \emph{superideal of definition} of $B$. For any multiplicative set $S\subseteq B_0$,
	let $B\{S^{-1}\}$ denote the completion of $S^{-1}B$ in the $S^{-1}J$-adic topology. If $S=\{g^n\mid n\geq 0\}$, then $B\{S^{-1}\}$ is denoted by $B_{(g)}$. 
	\begin{lm}\label{locality of some stalks}
		Let $\widehat{X}$ be a completion of $X=\mathrm{SSpec}(A)$ along $Y=\mathrm{SSpec}(A/I)$, where $A$ is a Noetherian superalgebra. 
		For any open superideal $\widehat{\mathfrak{P}}\in\mathrm{SSpecf}(\widehat{A})$ the stalk $\mathcal{O}_{\widehat{X}, \widehat{\mathfrak{P}}}$ is a local superalgebra, such that its residue field is isomorphic to $\Bbbk(A_{\mathfrak{P}})$.	
	\end{lm}	
	\begin{proof}	
		Set $S=A_0\setminus\mathfrak{p}$. By \cite[Lemma 1.5]{maszub4} the completion of $A$ in the $I$-adic topology coincides with the completion of $A$ in the $I_0$-adic topology.
		Thus the even component of $\mathcal{O}_{\widehat{X}, \widehat{\mathfrak{P}}}=\varinjlim_{f\in S} \widehat{A_f}$ is isomorphic to
		$\varinjlim_{f\in S}\widehat{(A_0)_f}$.  For any $g\in \widehat{A_0}\setminus\widehat{\mathfrak{p}}$ there is $f\in S$ such that $g\equiv f\pmod{\widehat{I_0}}$. Since $\frac{g}{f}=1+x$, where $x$ is topologically nilpotent in $\widehat{(A_0)_f}$, we have $\widehat{(A_0)_f}\simeq \widehat{A_0}_{(g)}$ and \cite[Corollary 0.7.6.3 and Proposition 0.7.6.17]{EGA I} imply the statement. 	
	\end{proof}
	This lemma shows that the formal completion of arbitrary Noetherian superscheme along its closed supersubscheme is a geometric superspace indeed.  
	
	A geometric superspace $\mathfrak{X}$ is said to be a \emph{Noetherian formal superscheme} (briefly, N.f. superscheme), if there is a finite open covering $\{\mathfrak{X}_i\}$ of $\mathfrak{X}$, such that each $\mathfrak{X}_i$ is isomorphic to the completion
	of some Noetherian superscheme $X_i$ along its closed supersubscheme $Y_i$. By the above remark, any N.f. superscheme has a finite open covering by the affine N.f. supersubschemes $\mathfrak{X}_i\simeq\mathrm{SSpecf}(A_i)$, where each $A_i$ is a Noetherian superalgebra, that is complete in the $J_i$-adic topology for some superideal $J_i$ of $A_i$. N.f. superschemes form a full subcategory of the category of geometric superspaces, that contains all Noetherian superschemes. 
	
Recall that a sheaf $\mathfrak{F}$ of $\mathcal{O}_{\mathfrak{X}}$-supermodules is said to be \emph{coherent}, if for some open covering $\{\mathfrak{X}_i\}$ as above, each sheaf $\mathfrak{F}|_{\mathfrak{X}_i}$ is isomorphic to $\widehat{\mathcal{F}_i}=\varprojlim_{n\geq 0}\mathcal{F}_i/\mathcal{F}_i \mathcal{I}^{n+1}_{Y_i}$, the completion of some coherent sheaf of $\mathcal{O}_{X_i}$-supermodules $\mathcal{F}_i$ along $Y_i$. 

Assume that $\mathfrak{X}=\widehat{X}$ is the formal completion of $X=\mathrm{SSpec}(A)$ along $Y=\mathrm{SSpec}(A/I)$, i.e. $\mathfrak{X}\simeq\mathrm{SSpecf}(\widehat{A})$. Let $M$ be a finitely generated $A$-supermodule. Following \cite[Section II.9]{hart}, we denote by $M^{\triangle}$ the completion of the coherent $\mathcal{O}_X$-supermodule $\widetilde{M}$ along $Y$.	
The following proposition superizes \cite[Proposition II.9.4]{hart}.   
\begin{pr}\label{some standard prop-s of coherent sheaves over affine}
	We have :
	\begin{enumerate}
		\item $\mathfrak{J}=I^{\triangle}$ is a superideal sheaf in $\mathcal{O}_{\mathfrak{X}}$, such that $\mathcal{O}_{\mathfrak{X}}/\mathfrak{I}^n$ is isomorphic to $\widetilde{A/I^n}$ as a sheaf on $|Y|$, for any $n\geq 0$;
		\item  if $M$ is a fintely generated $A$-supermodule, then $M^{\triangle}\simeq \widetilde{M}\otimes_{\mathcal{O}_X}\mathcal{O}_{\mathfrak{X}}$;
		\item the functor $M\mapsto M^{\triangle}$ is an exact functor from the category of finitely generated $A$-supermodules to the category of coherent $\mathcal{O}_{\mathfrak{X}}$-supermodules;
		\item for any integer $n> 0$ there is $(I^{\triangle})^n=(I^n)^{\triangle}$.
	\end{enumerate}	
\end{pr}	
\begin{proof}
	If $U$ is an affine supersubscheme of $X$, then \cite[Proposition 2.1(1), Proposition 3.1 and Lemma 3.2]{zub2} imply that the functor $M\mapsto \widetilde{M}(|U|)$ is exact, where 
	$N=\widetilde{M}(|U|)$ is a finitely generated supermodule over the Noetherian superalgebra $\mathcal{O}_X(|U|)$. Let $J$ denote the superideal $\widetilde{I}(|U|)$. Then
	\[M^{\triangle}(|U|)\simeq\widehat{N}=\varprojlim_{n\geq 0} N/J^{n+1}N\] and \cite[Lemma 6.1]{kolzub} infers $(3)$. Now, the same arguments as in \cite[Proposition II.9.4]{hart}, combined with \cite[Lemma 6.1(1)]{kolzub} prove $(1)$ and $(2)$.  Finally, $(I^{\triangle})^n$ is the sheafifcation of the presheaf
	\[U\mapsto \widetilde{I}(U)^n\otimes_{\mathcal{O}_X(U)}\mathcal{O}_{\mathfrak{X}}(U),\] 
and $\widetilde{I}^n=\widetilde{I^n}$ is the sheafifcation of the presheaf $U\mapsto \widetilde{I}(U)^n, U\subseteq |Y|$.	Then \cite[Lemma 17.16.1]{stack} concludes the proof.
\end{proof}	
\begin{pr}\label{characterization of formal via even}
	A geometric superspace $\mathfrak{X}$ is a N.f. superscheme if and only if $\mathfrak{X}_0$ is a N.f. scheme and $(\mathcal{O}_{\mathfrak{X}})_1$ is a coherent sheaf of $\mathcal{O}_{\mathfrak{X}_0}$-modules. 
\end{pr}
\begin{proof}
	Note that if $X$ is a Noetherian superscheme and $\mathcal{I}$ is a coherent superideal sheaf of $\mathcal{O}_X$, then there is a nonnegative integer $n$ such that $\mathcal{I}^n\subseteq \mathcal{O}_X\mathcal{I}_0$ (use \cite[Lemma 1.5]{maszub4}). 
	Let $\widehat{X}$ be a completion of $X$ along a closed supersubscheme $Y$. Let $Y'$ denote $(|Y|, \mathcal{O}_X/\mathcal{O}_X (\mathcal{I}_Y)_0)$. The above remark immediately infers that  $\widehat{X}$ and $(\widehat{X})_0$ are the completions of $X$ and $X_0$ along $Y'$ and $Y_0$ respectively, hence the part "if".
	
	Assume that $\mathfrak{X}$ is a geometric superspace, that satisfies the conditions of proposition. Without loss of a generality, one can assume that $\mathfrak{X}_0\simeq\mathrm{Specf}(A)$, where $A$ is an adic Noetherian algebra. By \cite[Proposition II.9.4]{hart} the $\mathcal{O}_{\mathfrak{X}_0}$-module
	$(\mathcal{O}_{\mathfrak{X}})_1$ is isomorphic to $M^{\triangle}\simeq \widetilde{M}\otimes_{\mathcal{O}_{X_0}} \mathcal{O}_{\mathfrak{X}_0}$, where $X_0=\mathrm{Spec}(A)$ and $M$ is a finitely generated $A$-module. The superalgebra sheaf structure of $\mathcal{O}_{\mathfrak{X}}$ is uniquely defined by an $\mathcal{O}_{\mathfrak{X}_0}$-linear morphism
	\[M^{\triangle}\otimes_{\mathcal{O}_{\mathfrak{X}_0}} M^{\triangle}\to \mathcal{O}_{\mathfrak{X}_0}\simeq A^{\triangle}.\]
	We have 
	\[M^{\triangle}\otimes_{\mathcal{O}_{\mathfrak{X}_0}} M^{\triangle}\simeq (\widetilde{M}\otimes_{\mathcal{O}_{X_0}}\widetilde{M})\otimes_{\mathcal{O}_{X_0}}\mathcal{O}_{\mathfrak{X}_0}   \simeq \widetilde{M\otimes_A M}\otimes_{\mathcal{O}_{X_0}} \mathcal{O}_{\mathfrak{X}_0}\simeq (M\otimes_A M)^{\triangle},\]
	hence by \cite[Theorem II.9.7]{hart} the above  $\mathcal{O}_{\mathfrak{X}_0}$-linear morphism is uniquely defined by a morphism $M\otimes_A M\to A$ of $A$-modules. Thus $B=A\oplus M$ has the natural superalgebra structure, and $\mathfrak{X}\simeq\mathrm{SSpecf}(B)$. 
\end{proof}
	
Let $\mathfrak{X}$ be a N.f. superscheme. A superideal sheaf $\mathfrak{I}$ of $\mathcal{O}_{\mathfrak{X}}$ is said to be a \emph{superideal of definition}, if
\[\mathrm{Supp}(\mathcal{O}_{\mathfrak{X} }/\mathfrak{I})=\{x\in |\mathfrak{X}| \mid (\mathcal{O}_{\mathfrak{X} }/\mathfrak{I})_x\neq 0  \}=|\mathfrak{X}|  \]
and $(|\mathfrak{X}|, \mathcal{O}_{\mathfrak{X} }/\mathfrak{I})$ is a Noetherian superscheme.
\begin{pr}\label{superideal of definition}(compare with \cite[Proposition II.9.5]{hart})
Let $\mathfrak{X}$ be a N.f. superscheme. Then :
\begin{enumerate}
	\item[(a)] if $\mathfrak{I}_1$ and  $\mathfrak{I}_2$ are two superideals of definition, then there are positive integers $m, n$, such that $\mathfrak{I}^m_1\subseteq \mathfrak{I}_2$
	and  $\mathfrak{I}^n_2\subseteq \mathfrak{I}_1$;
	\item[(b)] if $\mathfrak{X}\simeq\mathrm{SSpecf}(A)$, where $A$ is an $I$-adic Noetherian superalgebra, and $\mathfrak{J}$ is a superideal of definition of $\mathfrak{X}$, then there is a superideal of definition $J$ of $A$, such that $\mathfrak{J}=J^{\triangle}$;
	\item[(c)] if $\mathfrak{I}$ is a superideal of definition, then for any positive integer $n$, $\mathfrak{I}^n$ is as well;
	\item[(d)] if $\mathfrak{J}$ is an ideal of definition of $\mathfrak{X}_0$, then $\mathcal{O}_{\mathfrak{X}}\mathfrak{J}$ is a superideal of definition of $\mathfrak{X}$.
\end{enumerate}	
\end{pr}
\begin{proof}
The proof of  (a) can be copied from \cite[Proposition II.9.5(a)]{hart} verbatim. Thus, under the conditions of (b) there is $(I^n)^{\triangle}= (I^{\triangle})^n\subseteq\mathfrak{J}$ for some integer $n>0$. Without loss of generality, one can assume that $I^{\triangle}\subseteq\mathfrak{J}$. 
Since \[\mathcal{O}_{\mathrm{SSpec}(A/I)}\simeq\mathcal{O}_{\mathfrak{X}}/I^{\triangle}\to\mathcal{O}_{\mathfrak{X}}/\mathfrak{J}\] is a surjective morphism of superalgebra sheaves, 	
$Z=(|\mathfrak{X}|, \mathcal{O}_{\mathfrak{X} }/\mathfrak{J})$ is a closed supersubscheme of $\mathrm{SSpec}(A/I)$, defined by a superideal $J\supseteq I$. Moreover, $|Z|=|\mathrm{SSpec}(A/I)|$ implies $J^m\subseteq I$ for some integer $m>0$, and 
\[J^{\triangle}/I^{\triangle}\simeq (J/I)^{\triangle}\simeq\widetilde{J/I}\simeq\widetilde{J}/\widetilde{I}\]
infers $\mathfrak{J}=J^{\triangle}$.
 	
The statement (d) is obviously local, i.e. one can assume that $\mathfrak{X}\simeq\mathrm{SSpecf}(A)$ and $\mathfrak{J}=J^{\triangle}$ for some ideal of definition of $A_0$. 
Note that \[(\mathcal{O}_{\mathfrak{X}})_1\mathfrak{J}\subseteq (\mathcal{O}_{\mathfrak{X}})_1\simeq \widetilde{A_1}\otimes_{\mathcal{O}_{X_0}}\mathcal{O}_{\mathfrak{X}_0}\] is the natural image of the sheaf $\widetilde{A_1}\otimes_{\mathcal{O}_{X_0}}\mathfrak{J}$. By \cite[Lemma 17.16.3]{stack} we have
the exact sequence 
\[0\to (\mathcal{O}_{\mathfrak{X}})_1\mathfrak{J}\to \widetilde{A_1}\otimes_{\mathcal{O}_{X_0}}\mathcal{O}_{\mathfrak{X}_0}\to \widetilde{A_1}\otimes_{\mathcal{O}_{X_0}}(\mathcal{O}_{\mathfrak{X}_0}/\mathfrak{J})\to 0.\]
Since $\mathcal{O}_{\mathfrak{X}_0}/\mathfrak{J}\simeq\widetilde{A/J}$, $\widetilde{A_1}\otimes_{\mathcal{O}_{X_0}}(\mathcal{O}_{\mathfrak{X}_0}/\mathfrak{J})$
is isomorphic to $\widetilde{A_1/JA_1}$ as $\mathcal{O}_{\mathrm{Spec}(A/J)}$-module, thus (d) follows. 
Finally, the statement (c) is also local, hence it immediately follows by Proposition \ref{some standard prop-s of coherent sheaves over affine}(1, 4).
\end{proof}
\begin{lm}\label{theorem II.9.3}
Let $A$ be a Noetherian superalgebra and $I$ be a superideal of $A$. If $\{ M_n, \phi_{n, m} : M_n\to M_m,  0\leq m\leq n\}$ is an inverse system, where each $M_n$ is a finitely generated $A/I^n$-supermodule and each map $\phi_{n, m}$ is surjective with $\ker\phi_{n, m}=I^m M_n$, then $M=\varprojlim_n M_n$ is a finitely generated $\widehat{A}$-supermodule, such that $M/I^n M\simeq M_n$ for any $n\geq 0$.  	
\end{lm}
\begin{proof}
Assume that the elements $m^{(1)}_1, \ldots, m^{(1)}_s$ generate $M_1$. Since $M_1\simeq M_2/IM_2$, the elements $m^{(2)}_1, \ldots, m^{(2)}_s$, such that $\phi_{2, 1}(m^{(2)}_i)=m^{(1)}_i$ for any $i$, generate $M_2$. Similarly, $M_1\simeq M_3/IM_3$ infers that $m^{(3)}_1, \ldots, m^{(3)}_s$, such that  $\phi_{3, 2}(m^{(3)}_i)=m^{(2)}_i$ for any $i$, generate $M_3$, and so on. Set $m_i=\varprojlim_n m_i^{(n)}, 1\leq i\leq s$. Then $M=\sum_{1\leq i\leq s}\widehat{A}m_i$ by \cite[Lemma 1.7]{maszub4}.    	
\end{proof}
The following proposition superizes \cite[Proposition II.9.6]{hart}.
\begin{pr}\label{supermodule over formal as a proj limit}
Let $\mathfrak{X}$ be a N.f. superscheme. Assume that $\mathfrak{I}$ is a superideal of definition of $\mathfrak{X}$. Let $X_n$ denote the Noetherian superscheme
$(|\mathfrak{X}|, \mathcal{O}_{\mathfrak{X}}/\mathfrak{I}^n), n\geq 0$. Then :
\begin{enumerate}
	\item if $\mathfrak{F}$ is a coherent sheaf of $\mathcal{O}_{\mathfrak{X}}$-supermodules, then for any $n\geq 0$, $\mathcal{F}_n=\mathfrak{F}/\mathfrak{I}^n\mathfrak{F}$ is a coherent sheaf of $X_n$-supermodules, so that $\mathfrak{F}\simeq\varprojlim_n \mathcal{F}_n$;
	\item if $\{\mathcal{F}_n, \phi_{n, m} : \mathcal{F}_n\to\mathcal{F}_m, n>m\geq 0\}$ is an inverse system, where each $\mathcal{F}_n$ is a coherent sheaf of $\mathcal{O}_{X_n}$-supermodules and each morphism $\phi_{n, m}$ is surjective with $\ker\phi_{n, m}=\mathfrak{I}^m\mathcal{F}_n$, then $\mathfrak{F}=\varprojlim_n\mathcal{F}_n$
	is a coherent sheaf of $\mathcal{O}_{\mathfrak{X}}$-supermodules, such that $\mathcal{F}_n\simeq\mathfrak{F}/\mathfrak{I}^n\mathfrak{F}$ for any $n\geq 0$;
	\item a sheaf $\mathfrak{F}$ of $\mathcal{O}_{\mathfrak{X}}$-supermodules is coherent if and only if the sheaf $\mathfrak{F}|_{\mathfrak{X}_0}$ of $\mathcal{O}_{\mathfrak{X}_0}$-modules is.
\end{enumerate}	
\begin{proof}
For the first two statements use Lemma \ref{theorem II.9.3} and copy the proof of  \cite[Proposition II.9.6]{hart}. To prove the third one, let $\mathfrak{J}=\mathcal{O}_{\mathfrak{X} }\mathfrak{J}_0$. Then by \cite[Proposition 3.1]{zub2} each $\mathcal{F}_n\simeq \mathfrak{F}/\mathfrak{J}_0^n\mathfrak{F}$ is a coherent	$\mathcal{O}_{X_n}$-supermodule if and only if 
each $\mathcal{F}_n|_{\mathcal{O}_{(X_n)_0}}$ is a coherent $\mathcal{O}_{(X_n)_0}$-module if and only if $\mathfrak{F}|_{\mathfrak{X}_0}$ is a coherent $\mathcal{O}_{\mathfrak{X}_0}$-module.
\end{proof}
\end{pr}
\begin{pr}\label{equivalence over affine formal}
Let $\mathfrak{X}\simeq\mathrm{SSpecf}(A)$, where $A$ is an adic Noetherian superalgebra. Then $M\mapsto M^{\triangle}$ is an equivalence of the category of finitely generated $A$-supermodules to the category of coherent $\mathcal{O}_{\mathfrak{X}}$-supermodules, and its quasi-inverse is the exact functor $\mathfrak{F}\mapsto \mathfrak{F}(|\mathfrak{X}|)$. 	
\end{pr}
\begin{proof}
Proposition \ref{some standard prop-s of coherent sheaves over affine}(2) implies $M^{\triangle}(|\mathfrak{X}|)=M$. Conversely, assume that $\mathfrak{F}$ is a coherent sheaf of $\mathcal{O}_{\mathfrak{X}}$-supermodules. Set $\mathfrak{I}=I^{\triangle}$, where $I$ is a superideal of definition of $A$. By Proposition \ref{supermodule over formal as a proj limit}, for each $n\geq 1$ we have  $\mathfrak{F}/\mathfrak{I}^n\mathfrak{F}\simeq \widetilde{M_n}$, where $M_n$ is a finitely generated $A/I^n$-supermodule. Moreover, $M=\varprojlim_n M_n$ is a finitely generated $A$-supermodule and $\mathfrak{F}\simeq M^{\triangle}$. The proof of the exactness of the functor $\mathfrak{F}\mapsto \mathfrak{F}(|\mathfrak{X}|)$ can be copied from \cite[Theorem II.9.7]{hart}.
\end{proof}
Let $\mathsf{Top-SAlg}_{\Bbbk}$ denote the subcategory of topological superalgebras with graded continuous morphisms. 
\begin{lm}\label{morphisms to affine}(see \cite[Proposition I.10.4.6]{EGA I})
The natural map \[\mathrm{Mor}(\mathfrak{X}, \mathrm{SSpecf}(A))\to\mathrm{Hom}_{\mathsf{Top-SAlg}_{\Bbbk}}(A, \mathcal{O}_{\mathfrak{X}}(|\mathfrak{X}|))\] is bijective. 	
\end{lm}
\begin{proof}
Without loss of a generality, one can assume that $\mathfrak{X}\simeq\mathrm{SSpecf}(B)$. Suppose that $I$ and $J$ are superideals of definition in $A$ and $B$ respectively.
Let $f : \mathrm{SSpecf}(B)\to \mathrm{SSpecf}(A)$ be a morphism of formal superschemes.  The dual superalgebra morphism $A\to B$ is denoted by $\phi$. Since $|\mathrm{SSpecf}(B)|=|\mathrm{Specf}(B_0)|$ as well as $|\mathrm{SSpecf}(A)|=|\mathrm{Specf}(A_0)|$, and the morphism $f_0$ is uniquely defined by $\phi_0 : A_0\to B_0$, there is $|f|(\mathfrak{Q})=\phi^{-1}(\mathfrak{Q})$ for any $\mathfrak{Q}\in |\mathrm{SSpecf}(B)|$. Furthermore, the superideal $\cap_{\mathfrak{Q}\in |\mathrm{SSpecf}(B)|}\mathfrak{Q}$ is nilpotent over $J$, that implies $\phi(I)^n\subseteq J$ for a positive integer $n$, whence $\phi$ is continuous.

Any open supersubspace $\mathfrak{U}$ in $\mathrm{SSpecf}(A)$ is covered by some $\mathrm{SSpecf}(\widehat{A_x}), x\in A_0$.  Since $f^{-1}(\mathrm{SSpecf}(\widehat{A_x}))=\mathrm{SSpecf}(\widehat{B_{\phi(x)}})$, and the superalgebra morphism $f^{\sharp}(\mathfrak{U}) : \mathcal{O}(\mathfrak{U})\to \mathcal{O}(f^{-1}(\mathfrak{U}))$ is uniquely defined by its restrictions
\[f^{\sharp}(\mathrm{SSpecf}(\widehat{A_x})) : \widehat{A_x}\simeq \mathcal{O}(\mathrm{SSpecf}(\widehat{A_x}))\to \mathcal{O}(\mathrm{SSpecf}(\widehat{B_{\phi(x)}}))\simeq \widehat{B_{\phi(x)}},\]
which are, in turn, uniquely defined by $\phi$, thus our lemma follows.
\end{proof}
\begin{cor}\label{fiber product for affine}
Assume that $\mathrm{SSpecf}(A)$ and $\mathrm{SSpecf}(B)$ are affine N.f. superschemes over an affine N.f. superscheme $\mathrm{SSpecf}(C)$. 	
Then $\mathrm{SSpecf}(\widehat{A\otimes_C B})$ is isomorphic to the fiber product $\mathrm{SSpecf}(A)\times_{\mathrm{SSpecf}(C)} \mathrm{SSpecf}(B)$ in the category of N.f. superschemes.
\end{cor}
\begin{pr}\label{fibered product for fromal}
The fiber products exist in the category of N.f. superschemes.	
\end{pr}
\begin{proof}
We briefly outline the main steps of the proof. For a couple of morphisms $f : \mathfrak{X}\to \mathfrak{Z}$ and $g : \mathfrak{Y}\to \mathfrak{Z}$ of N.f. superschemes,
we choose open affine coverings  $\{\mathfrak{U}_{ij}\}, \{\mathfrak{V}_{ik}\}$ and $\{\mathfrak{W}_i\}$ of $\mathfrak{X}, \mathfrak{Y}$ and $\mathfrak{Z}$ respectively, such that 
$f(\mathfrak{U}_{ij})\subseteq \mathfrak{W}_i$ and $f(\mathfrak{V}_{ik})\subseteq \mathfrak{W}_i$ for each triple of indices $i, j, k$. Then by Corollary \ref{fiber product for affine},
for each triple of indices $i, j, k$ there exists the fiber product  $\mathfrak{U}_{ij}\times_{\mathfrak{W}_{i}} \mathfrak{V}_{ik}$, so that $\mathfrak{X}\times_{\mathfrak{Z}}\mathfrak{Y}$ can be glued from these affine N.f. superschemes  (for more details, see  \cite[Theorem I.3.2.6 and Proposition I.10.7.3]{EGA I}, or \cite[Theorem II.3.3]{hart}).	
\end{proof}

\section{Separated and proper morphisms of formal superschemes}

\begin{lm}
Let $\mathfrak{X}$ be a N.f. superscheme and $\mathfrak{J}$ be a coherent superideal sheaf in $\mathcal{O}_{\mathfrak{X}}$. Set $Y=\mathrm{Supp}(\mathcal{O}_{\mathfrak{X}}/\mathfrak{J})$. 
Then $Y$ is a closed subset of $|\mathfrak{X}|$ and $\mathfrak{Y}=(Y, \mathcal{O}_{\mathfrak{X}}/\mathfrak{J})$ is a N.f. superscheme. 
\end{lm}
\begin{proof}
Since
$\mathrm{Supp}(\mathcal{O}_{\mathfrak{X} }/\mathfrak{J})=\mathrm{Supp}(\mathcal{O}_{\mathfrak{X}_0}/\mathfrak{J}_0)$, $Y$ is a closed subset of $|\mathfrak{X}|$ and $\mathfrak{X}_0/\mathfrak{J}_0$ is a Noetherian formal scheme (see \cite[I.10.14.1]{EGA I}). Moreover, $(\mathcal{O}_{\mathfrak{X}})_1/\mathfrak{J}_1$ is a coherent
$\mathcal{O}_{\mathfrak{X}_0}$-module by \cite[Corollary II.9.9]{hart}, and Proposition \ref{characterization of formal via even} concludes the proof.	
\end{proof}
The formal superscheme $\mathfrak{Y}$ is said to be a \emph{closed formal supersubscheme} of $\mathfrak{X}$ . A morphism $f : \mathfrak{Z}\to\mathfrak{X}$ of (Noetherian) formal superschemes is called a \emph{closed immersion}, if $f$ is an isomorphism onto a closed formal supersubscheme $\mathfrak{Y}$ of $\mathfrak{X}$. 

Let $f : \mathfrak{X}\to\mathfrak{Y}$ be a morphism of N.f. superschemes. Let $\delta_f$ denote  the corresponding diagonal morphism 
$\mathfrak{X}\to \mathfrak{X}\times_{\mathfrak{Y}}\mathfrak{X}$. Then $f$ is said to be \emph{separated}, provided $|\delta_f|$ is a homeomorphism of $|X|$ onto a closed subset of 
$|\mathfrak{X}\times_{\mathfrak{Y}}\mathfrak{X}|$. Later, we will show that $f$ is separated if and only if $\delta_f$ is a closed immersion. 

A morphism of N.f. superschemes $f : \mathfrak{X}\to\mathfrak{Y}$ is called \emph{adic}, if there is a superideal of definition $\mathfrak{J}$ of $\mathfrak{Y}$ such that
$\mathcal{O}_{\mathfrak{X}} f^{\sharp}(\mathfrak{J})$ is a superideal of definition of $\mathfrak{X}$. We also say that $\mathfrak{X}$ is adic over $\mathfrak{Y}$. For any integer $n>0$ the morphism $f$ induces the morphism $f_n : X_n\to Y_n$ of Noetherian superschemes. 

Note that by Proposition \ref{superideal of definition}, if $f$ is adic, then $f$ is adic with respect to arbitrary superideal of definition of $\mathfrak{Y}$. 
\begin{lm}\label{closed immersion for affine}
Let $f : \mathfrak{X}\to\mathfrak{Y}$ be a morphism of affine N.f. superschemes. Then $f$ is a closed immersion if and only if the superalgebra morphism $f^{\sharp} : \mathcal{O}(\mathfrak{Y})\to \mathcal{O}(\mathfrak{X})$ is surjective. 
\end{lm} 
\begin{proof}
The part "if" is obvious by Proposition \ref{equivalence over affine formal}. Conversely, if $f^{\sharp}$ is surjective and $I$ is a superidel of definition of $\mathcal{O}(\mathfrak{Y})$, then $IB$ is the superideal of definition of $\mathcal{O}(\mathfrak{X})$. Again, by Proposition \ref{equivalence over affine formal}, $\mathcal{O}_{\mathfrak{X}}$ is naturally isomorphic to $(\mathcal{O}(\mathfrak{Y})/J)^{\triangle}\simeq\mathcal{O}_{\mathfrak{Y}}/J^{\triangle}$ as an $\mathcal{O}_{\mathfrak{Y}}$-supermodule.
\end{proof}
\begin{lm}\label{a criteria for a morphism to be separated}
Let $f : \mathfrak{X}\to\mathfrak{Y}$ be an adic morphism of N.f. superschemes. Then $f$ is separated if and only if $f_1$ is.  	
\end{lm}
\begin{proof}
Let $\mathfrak{J}$ be a superideal of definition of $\mathfrak{Y}$, such that $\mathfrak{I}=\mathcal{O}_{\mathfrak{X}}f^{\sharp}(\mathfrak{J})$ is a superideal of definition of $\mathfrak{X}$.	
Note that both projections and the structural morphism $\mathfrak{X}\times_{\mathfrak{Y}}\mathfrak{X}\to\mathfrak{Y}$ are adic as well. More precisely,  $\mathcal{O}_{\mathfrak{X}\times_{\mathfrak{Y}}\mathfrak{X}}\delta_f^{\sharp}(\mathfrak{J})$ is a superideal of definition of $\mathfrak{X}\times_{\mathfrak{Y}}\mathfrak{X}$.

The induced morphism $(\delta_f)_1$ is naturally identified with $\delta_{f_1}$. In fact, there is a canonical bijection between morphisms $Z\to X_1$ (of Noetherian superschemes) over $Y_1$, and morphisms  $Z\to\mathfrak{X}$ of (N.f. superschemes) over $\mathfrak{Y}$, such that $Z\to \mathfrak{Y}$ factors through $Z\to Y_1$. Thus for any couple of such morphisms from $Z$ to $X_1$, there is the unique morphism $Z\to \mathfrak{X}\times_{\mathfrak{Y}}\mathfrak{X}$, that obviously factors through $(\mathfrak{X}\times_{\mathfrak{Y}}\mathfrak{X})_1\to \mathfrak{X}\times_{\mathfrak{Y}}\mathfrak{X}$, hence $(\mathfrak{X}\times_{\mathfrak{Y}}\mathfrak{X})_1\simeq X_1\times_{Y_1} X_1$. It remains to note that $|\mathfrak{S}|=|S_1|$   
for any N.f. superscheme $\mathfrak{S}$. 
\end{proof}
\begin{cor}\label{properties of separated morphisms}
All standard properties of separated morphisms of N.f. schemes, formulated in \cite[Proposition I.10.15.13]{EGA I}, are valid for adic separated morphisms of N.f. superschemes.
\end{cor}
\begin{proof}
Use the above isomorphism $(\mathfrak{X}\times_{\mathfrak{Y}}\mathfrak{X})_1\simeq X_1\times_{Y_1} X_1$ and refer to \cite[Corollary 2.5]{maszub2}.	
\end{proof}
Let $A$ be an adic  Noetherian superalgebra, and let $I$ be a superideal of definition of $A$. Let $A[t_1, \ldots, t_n| z_1, \ldots , z_m]$ be a \emph{polynomial superalgebra over $A$}, freely generated by the even indeterminants $t_1, \ldots, t_n$ and by the odd indeterminants $z_1, \ldots, z_m$. Then the superalgebra 
\[A\{t_1, \ldots, t_n| z_1, \ldots , z_m\}=\varprojlim_{k} A/I^k[t_1, \ldots, t_n| z_1, \ldots , z_m]\] 
is said to be a \emph{superalgebra of restricted formal series} over $A$. 
It is clear that $A\{t_1, \ldots, t_n| z_1, \ldots , z_m\}$ is supersubalgebra of 
\[A[[t_1, \ldots, t_n| z_1, \ldots , z_m]]\simeq A[[t_1, \ldots, t_n]][z_1, \ldots , z_m]\] consisting of all formal series 
\[\sum_{\alpha, M}c_{\alpha, M}t^{\alpha} z^M, \alpha\in\mathbb{N}^n, M\subseteq \{1, \ldots, m \}, \]  
such that for any $k\geq 0$ almost all coefficients $c_{\alpha, M}$ belongs to $I^k$. We also use the shorter notation $A\{\underline{t}\mid\underline{z} \}$, instead of $A\{t_1, \ldots, t_n| z_1, \ldots , z_m\}$.

Finally, a topological superalgebra $B$ is said to be \emph{formally finite} over an $I$-adic superalgebra $A$, if $B$ is isomorphic to the quotient of some $A\{\underline{t} \mid \underline{z} \}$ over a closed superideal. Equivalently, $B$ is formally finite over $A$ if and only if $B$ is $IB$-adic, and $B/IB$ is finite over $A/I$. In fact, the part "if" is trivial. Conversely, we choose a finite set of $n$ even and $m$ odd generators of $B/IB$ over $A/IA$. Then for any positive integers $k\geq l$ we have a commutative diagram (of $A$-superalgebras)
\[\begin{array}{ccccccccc}
0 & \to & R_k & \to & A/I^k A[\underline{t}\mid \underline{z}] & \to & B/I^k B & \to & 0 	\\
  &     & \downarrow & & \downarrow & & \downarrow & & \\
0 & \to & R_l & \to & A/I^l A[\underline{t}\mid \underline{z}] & \to & B/I^l B & \to & 0  
\end{array},\]
where $\underline{t}=\{t_1, \ldots, t_n\}, \underline{z}=\{z_1, \ldots, z_m\}$. Since $\ker(B/I^k B\to B/I^l B)=I^l(B/I^k B)$ and
the preimage of $I^l(B/I^k B)$ is equal to $I^l(A/I^k A[\underline{t}\mid \underline{z}])+R_k$, each superalgebra morphism $R_k\to R_l$ is surjective, hence the inverse system $\{ R_k\to R_l, k\geq l\}$ satisfies the \emph{Mittag-Leffler condition}. Thus $B\simeq A\{\underline{t}\mid \underline{z}\}/R, R=\varprojlim_k R_k$. 
\begin{lm}\label{transitive}
The following statements hold :
\begin{enumerate}
\item the property to be formally finite is transitive;
\item if $B$ is an $A'$-superalgebra, $A'$ is an $A$-superalgebra, $A$ is $I$-adic and $A'$ is $IA'$-adic, and $B$ is formally finite over $A$, then $B$ is formally finite over $A'$.
\end{enumerate}
\end{lm}
\begin{proof}
Assume that $C$ is formally finite over $B$ and $B$ is formally finite over $A$. Then there are superideals of definition $I$ and $J$ of $A$ and $B$ respectively, such that $B/IB$ is finite over $A/I$ and $C/JC$ is finite over $B/J$. There is an integer $n>0$, such that $J^n\subseteq IB$. Thus, if the (homogeneous) elements $c_1, \ldots, c_l$ generate $B/J$-superalgebra $C/JC$, then they obviously generate $C/IC$ over $B/IB$, and $(1)$ follows. The second statement is obvious.  
\end{proof}	
Superizing \cite[Definition I.10.13.3]{EGA I}, we say that a morphism $f : \mathfrak{X}\to\mathfrak{Y}$ of N.f. superschemes is of \emph{finite type}, if there are finite open coverings $\{\mathfrak{V}_i\}$ and $\{\mathfrak{U}_{ij}\}$ of $\mathfrak{Y}$ and $\mathfrak{X}$ respectively, such that for each couple of indices $i, j$ there is $f(\mathfrak{U}_{ij})\subseteq \mathfrak{V}_i$, and $\mathcal{O}(\mathfrak{U}_{ij})$ is formally finite over $\mathcal{O}(\mathfrak{V}_i)$.
In particular, any morphism of finite type is obviously adic.  
\begin{lm}\label{finite type for formal}
In the above notations, $\mathfrak{X}$ is of finite type over $\mathfrak{Y}$ if and only if $\mathfrak{X}_0$ is of finite type over $\mathfrak{Y}_0$ if and only if 	
$\mathfrak{X}_{ev}$ is of finite type over $\mathfrak{Y}_{ev}$ if and only if the morphism $f$ is adic and $X_1$ is of finite type over $Y_1$.
\end{lm}
\begin{proof}
Note that a Noetherian superalgebra $C$ is a finitely generated $C_0$-supermodule (cf. \cite[Lemma 1.4]{maszub4}). Therefore, if $C$ is $I$-adic, then $C$ is formally finite over $C_0$ (regarded as an $I_0$-adic algebra). Furthermore, if $B$ is formally finite over $A$, then $B$ is formally finite over $A_0$ by Lemma \ref{transitive}(1), hence $B_0$ is formally finite over $A_0$. Conversely, if $B_0$ is formally finite over $A_0$, then $B$ is formally finite over $A_0$, and Lemma \ref{transitive}(2), combined with \cite[Lemma 1.5]{maszub4}, imply the first equivalence. 

If $\overline{B}$ is formally finite over $\overline{A}$, then there are superideals of definition $J$ and $I$ of $B$ and $A$ respectively, such that $J^n+J_B\subseteq IB+J_B$ and $I^m B+J_B\subseteq J+J_B$ for some integers $m, n>0$. On the other hand, $J_B$ is nilpotent, say $J_B^l=0$, that infers $J^{nl}\subseteq IB$ and $I^{ml} B\subseteq J$, hence $B$ is $IB$-adic. It remains to note that if $C$ is a Noetherian superalgebra over a superalgebra $D$, and $\overline{C}$ is finite over $\overline{D}$, then each quotient $J_C^k/J_C^{k+1}$ is a finitely generated supermodule over a polynomial superalgebra $A[\underline{t}]$, hence $C$ is finite over $D$.  
The third equivalence is obvious. 
\end{proof}
\begin{pr}\label{graph morphism}
Let $f : \mathfrak{X}\to\mathfrak{Y}$ be a morphism of N.f. superschemes over a N.f. superscheme $\mathfrak{Z}$. Assume that both $\mathfrak{X}$ and $\mathfrak{Y}$ are of finite type and adic over $\mathfrak{Z}$, and $\mathfrak{Y}$ is also separated over $\mathfrak{Z}$. Then the graph morphism $\Gamma_f : \mathfrak{X}\to \mathfrak{X}\times_{\mathfrak{Z}} \mathfrak{Y}$, induced by $\mathrm{id}_{\mathfrak{X}}$ and $f$, is a closed immersion.
\end{pr}
\begin{proof}
Arguing as in Lemma \ref{a criteria for a morphism to be separated}, $|\mathfrak{X}|=|X_1|, |\mathfrak{X}\times_{\mathfrak{Z}} \mathfrak{Y}|=|X_1\times_{Z_1} Y_1|$ and $|\Gamma_f|$ can be naturally identified with
$|\Gamma_{f_1}|$. Combining \cite[Lemma 2.3 and Proposition 2.4]{maszub2} with \cite[Proposition I.10.15.4]{EGA I}, one sees that $|\Gamma_f|$ is a homeomorphism onto a closed subset of
$|\mathfrak{X}\times_{\mathfrak{Z}} \mathfrak{Y}|$. Further, $\Gamma_f$ can be glued from the morphisms
\[\Gamma_f|_{\mathfrak{U}}=\Gamma_{f|_{\mathfrak{U}}} : \mathfrak{U}\to \mathfrak{U}\times_{\mathfrak{W}} \mathfrak{V},\]
where $f(\mathfrak{U})\subseteq \mathfrak{V}$, $\mathfrak{U}, \mathfrak{V}$ and $\mathfrak{W}$ are open affine N.f. supersubschemes of $\mathfrak{X}, \mathfrak{Y}$ and $\mathfrak{Z}$
respectively, such that  $f|_{\mathfrak{U}} : \mathfrak{U}\to\mathfrak{V}$ is a morphism over $\mathfrak{W}$. The same arguments as in \cite[Lemma I.10.14.4]{EGA I} show that  $\Gamma_f$ is a closed immersion if and only if each $\Gamma_{f|_{\mathfrak{U}}}$ is, so that Lemma \ref{closed immersion for affine} concludes the proof.
\end{proof}

Similarly to \cite[III.3.4.1]{EGA  III}, a morphism $f : \mathfrak{X}\to\mathfrak{Y}$ of N.f. superschemes is called \emph{proper}, if $f$ is of finite type and for some superideal of definition $\mathfrak{J}$ of $\mathfrak{Y}$, the induced morphism $X_1=(|\mathfrak{X}|, \mathcal{O}_{\mathfrak{X}}/\mathcal{O}_{\mathfrak{X}}f^{\sharp}(\mathfrak{J}))\to Y_1=(|\mathfrak{Y}|, \mathcal{O}_{\mathfrak{Y}}/\mathfrak{J})$ is proper. Thus, by Lemma \ref{reducibility} and Proposition \ref{superideal of definition}(a), this property does not depend on the choice of $\mathfrak{J}$ and $X_n$ is proper over $Y_n$ for arbitrary $n\geq 1$. 
\begin{pr}\label{properness over formal}
Let $f : \mathfrak{X}\to \mathfrak{Y}$ be a morphism of N.f. superschemes. Then $f$ is proper if and only if $f_0$ is if and only if $f_{ev}$ is if and only if the morphism $f$ is adic and $f_1$ is proper. 	
\end{pr}
\begin{proof}
The third statement is trivial. The first and second statements follow by Lemma \ref{reducibility}, Proposition \ref{superideal of definition}(d) and Lemma \ref{finite type for formal}. 
\end{proof}

\section{Comparison of superscheme morphisms and their formal prolongations}

	Let $A$ be a Noetherian superalgebra and $I$ be its superideal. Let $S$ denote the affine superscheme $\mathrm{SSpec}(A)$ and $S'\simeq\mathrm{SSpec}(A/I)$ is its closed supersubscheme, defined by $I$.	
	
	Assume that $X$ is a superscheme of finite type over $S$. In partiular, $X$ is Noetherian. The preimage of $S'$ in $X$ is a closed supersubscheme $X'$. Let $\widehat{X}$ and $\widehat{S}$
	denote the formal completions of $X$ and $S$ along $X'$ and $S'$ respectively. Observe that for any open afine supersubscheme $U\simeq\mathrm{SSpec}(B)$ of $X$ there is
	$U'=U\cap X'\simeq\mathrm{SSpec}(B/B I)$. Thus $\mathcal{O}_{\widehat{X}}\simeq \varprojlim_{n\geq 0} \mathcal{O}_X/\mathcal{O}_X I^{n+1}$ and $\widehat{X}$ is a N.f. superscheme  of finite type over $\widehat{S}=\mathrm{SSpecf}(\widehat{A})$. Besides, $X\mapsto \widehat{X}$ is a functor from the category of superschemes of finite type over $S$ to the category of N.f. superschemes of finite type over $\widehat{S}$.
	
	Assume that $Y$ is another superscheme of finite type over $S$. Let $Z$ denote the superscheme $X\times_S Y$. Then $Z$ is of finite type over $S$ as well, hence Noetherian.
	The proof of the following lemma is similar to the proof of Proposition \ref{fibered product for fromal}.
	\begin{lm}\label{completion of fibered product}
		The superspace $\widehat{Z}$ is isomorphic to the fiber product $\widehat{X}\times_{\widehat{S}}\widehat{Y}$ in the category of N.f. superschemes, so that, the canonical projections are identified with $\widehat{\mathrm{pr}_1}$ and  $\widehat{\mathrm{pr}_2}$
		respectively.	
	\end{lm}	
\begin{pr}\label{closed supersubschemes in completion}
Assume that the above morphism $X\to S$ is separated and the superalgebra $A$ is $I$-adic. Then the map $T\mapsto\widehat{T}$ is a bijection from the set of closed supersubschemes of $X$, proper over $S$, to the set of closed N.f. supersubschemes of $\widehat{X}$, proper over $\widehat{S}$.  	
\end{pr}
\begin{proof}
Without loss of generality, one can assume that $I=I_0A$. If $T$ satisfies the conditions of lemma, then $T_0$ is closed subscheme of $X_0$, proper over $S_0$, and by \cite[Corollary III.5.1.8]{EGA  III}, $\widehat{T_0}\simeq \widehat{T}_0$ is proper over $\widehat{S_0}\simeq\widehat{S}_0$. By Proposition \ref{properness over formal}, $\widehat{T}$ is proper over $\widehat{S}$. Moreover, if $\mathcal{O}_T=\mathcal{O}_X/\mathcal{J}$, where $\mathcal{J}$ is a coherent sheaf of $\mathcal{O}_X$-supermodules,
then $\mathcal{O}_{\widehat{T}}=\widehat{\mathcal{O}_X}/\widehat{\mathcal{J}}$. Since $\mathcal{J}|_{\mathcal{O}_{X_0}}$ is a coherent $\mathcal{O}_{X_0}$-module, \cite[Corollary III.5.1.6]{EGA  III} implies that $\widehat{\mathcal{J}}|_{\mathcal{O}_{\widehat{X}_0}}$ is a coherent $\mathcal{O}_{\widehat{X}_0}$-module, hence by 	Proposition \ref{supermodule over formal as a proj limit}(3),
$\widehat{\mathcal{J}}$ is a coherent $\mathcal{O}_{\widehat{X}}$-supermodule and $\widehat{T}$ is closed in $\widehat{X}$.

Conversely, if $\mathfrak{T}$ is a closed supersubscheme of $\widehat{X}$, proper over $\widehat{S}$, then $\mathcal{O}_{\mathfrak{T}}=\mathcal{O}_{\widehat{X}}/\mathfrak{J}$, where
$\mathfrak{J}$ is a coherent $\mathcal{O}_{\widehat{X}}$-supermodule, hence a coherent $\mathcal{O}_{\widehat{X}_0}$-module as well. Again, by \cite[Corollary III.5.1.6]{EGA  III}, 
$\mathfrak{J}=\widehat{\mathcal{J}}$ for some coherent $\mathcal{O}_{X_0}$-supermodule of $\mathcal{O}_X$. Note that $\widehat{\mathcal{O}_X\mathcal{J}}=\mathfrak{J}$, that is $\mathcal{J}$ is a coherent $\mathcal{O}_X$-superideal. Set $T=(\mathrm{Supp}(\mathcal{O}_X/\mathcal{J}), \mathcal{O}_X/\mathcal{J})$. It is clear that $\widehat{T}=\mathfrak{T}$, and since $\mathfrak{T}_0$ is proper over $\widehat{S_0}$, $T_0$ is proper over $S_0$, and in its turn, $T$ is proper over $S$. 
\end{proof}
	Let $\mathrm{Mor}_S(X, Y)$ (respectively, $\mathrm{Mor}_{\widehat{S}}(\widehat{X}, \widehat{Y})$) denote the set of morphisms $X\to Y$ of $S$-superschemes (respectively, the set of morphisms $\widehat{X}\to\widehat{Y}$ of N.f. $\widehat{S}$-superschemes). It is clear that all morphisms from $\mathrm{Mor}_{\widehat{S}}(\widehat{X}, \widehat{Y})$ are adic. From now on $A$ is supposed to be an $I$-adic superalgebra. Let also $I=I_0 A$.
	
	The following proposition is a superization of \cite[Theorem III.5.4.1]{EGA  III}. 
	\begin{pr}\label{endomorphisms of completions}
		The natural map $\mathrm{Mor}_S(X, Y)\to \mathrm{Mor}_{\widehat{S}}(\widehat{X}, \widehat{Y})$ is bijective, provided $X$ is proper over $S$ and $Y$ is separated over $S$.
	\end{pr}
	The proof will be given in two lemmas.
	\begin{lm}\label{injectivity}
		The map $\mathrm{Mor}_S(X, Y)\to \mathrm{Mor}_{\widehat{S}}(\widehat{X}, \widehat{Y})$ is injective.	
	\end{lm}
	\begin{proof}
		Assume that $\widehat{\phi}=\widehat{\psi}$ for some $\phi, \psi\in \mathrm{Mor}_S(X, Y)$. Then $|\phi||_{|X'|}=|\psi||_{|X'|}$. We show that there is an open supersubscheme $U$ of $X$ such that $X'\subseteq U$ and $\phi|_U=\psi|_U$. Since $X_{ev}\to S_{ev}$ is a proper morphism of schemes and $(X')_{ev}\subseteq U_{ev}$ coincides with the preimage of $(S')_{ev}=S'\cap S_{ev}=(S_{ev})'$ in $X_{ev}$, \cite[Lemma III.5.1.3.1]{EGA  III} imlplies $X_{ev}=U_{ev}$. From $|X|=|X_{ev}|=|U_{ev}|=|U|$ we immediately derive that $U=X$. 
		
		Choose finite open affine coverings $\{U_i\}, \{V_{ij}\}$ of $Y$ and $X$ respectively, such that 
		\[\phi^{-1}(Y'\cap U_i)=\cup_{j} (V_{ij}\cap X') =\psi^{-1}(Y'\cap U_i)\]
		for each index $i$. Then for arbitrary couple of indices $i, j$, we have $\widehat{\phi|_{V_{ij}}}=\widehat{\psi|_{V_{ij}}}$. In other words, one can assume that $X$ and $Y$ are affine, say
		$X\simeq\mathrm{SSpec}(B), Y\simeq\mathrm{SSpec}(C)$. Let $u :  C\to B$ and $v : C\to B$ denote the superalgebra morphisms, those are dual to $\phi$ and $\psi$ respectively.
		Then the induced superalgebra morphisms $\widehat{u}$ and $\widehat{v}$ are equal each to other. Thus it follows that $(u-v)(c)\in \cap_{n\geq 0} C I^{n+1}$ for any $c\in C$. Since $C$ is a finitely generated $A$-superalgebra, by \emph{Krull intersection theorem} (cf. \cite[Proposition 1.10]{maszub4}) there is $x\in C_0 I_0+C_1 I_1$, such that for any $c\in C$ there is $n\geq 0$ with $(1-x)^n(u-v)(c)=0$. In other words, $\phi|_{X_{1-x_0}}=
		\psi|_{X_{1-x_0}}$, where $X_{1-x_0}\simeq\mathrm{SSpec}(C_{1-x_0})$ and $X'\subseteq X_{1-x_0}$. 
	\end{proof}
	\begin{lm}\label{surjectivity}
		The map $\mathrm{Mor}_S(X, Y)\to \mathrm{Mor}_{\widehat{S}}(\widehat{X}, \widehat{Y})$ is surjective.	
	\end{lm}
	\begin{proof}
Let $h\in \mathrm{Mor}_{\widehat{S}}(\widehat{X}, \widehat{Y})$. By Lemma \ref{a criteria for a morphism to be separated} the N.f. superscheme $\widehat{Y}$ is separated over $\widehat{S}$, and by Proposition \ref{graph morphism} the graph morphism $\Gamma_h : \widehat{X}\to \widehat{Z}$ is a closed immersion, such that $\widehat{\mathrm{pr}_1}\Gamma_h=\mathrm{id}_{\widehat{X}}$.
		Let $\mathfrak{T}$ denote the closed N.f. supersubscheme of $\widehat{Z}$, such that $\Gamma_h$ induces the isomorphism $\widehat{X}\simeq\mathfrak{T}$ (over $\widehat{S}$). 
		In particular, $\mathfrak{T}$ is proper over $\widehat{S}$. By Proposition \ref{closed supersubschemes in completion} there is a closed supersubscheme $T$ in $Z$, such that $\widehat{T}=\mathfrak{T}$. Let $i$ denote the canonical embedding $T\to Z$. Then $\widehat{\mathrm{pr}_1}\widehat{i}$ is an isomorphism $\widehat{T}\to\widehat{X}$, induced by
		$u=\mathrm{pr}_1 i : T\to X$. By \cite[Proposition III.4.6.8]{EGA  III}, $u_0 : T_0\to X_0$ is an isomorphism.  Identifying $\mathcal{O}_{T_0}$ with $\mathcal{O}_{X_0}$, and using again  \cite[Corollary III.5.1.6]{EGA  III}, one sees that the morphism $u^{\sharp} : \mathcal{O}_X\to \mathcal{O}_T$ of coherent $\mathcal{O}_{X_0}$-(super)modules is an isomorphism, whence $u$ is. The rest of the proof can be copied from \cite[Theorem III.5.4.1]{EGA  III}.
	\end{proof}

	\section{Some auxiliary properties of $\mathcal{O}_X$-supermodules}
	
	Recall that if $\mathcal{F}$ and $\mathcal{G}$ are sheaves of vector $\Bbbk$-superspaces on a topological space $X$, then $\mathrm{Hom}_{\Bbbk}(\mathcal{F}, \mathcal{G})$ is a sheaf 
	\[U\mapsto \{\mbox{morphisms of} \ \Bbbk-\mbox{sheaves} \ \mathcal{F}|_U\to\mathcal{G}|_U \}, \]
	where $U$ runs over open subsets of $X$. Furthermore, $\mathrm{Hom}_{\Bbbk}(\mathcal{F}, \mathcal{G})$ has a natural $\mathbb{Z}_2$-grading by
	\[\mathrm{Hom}_{\Bbbk}(\mathcal{F}, \mathcal{G})_{i}=\oplus_{j\in\mathbb{Z}_2}\mathrm{Hom}_{\Bbbk}(\mathcal{F}_j, \mathcal{G}_{i+j}), i\in\mathbb{Z}_2, \]
	i.e. it is again a sheaf of $\Bbbk$-superspaces.
	
	If $\mathcal{F}=\mathcal{G}$, then $\mathrm{Hom}_{\Bbbk}(\mathcal{F}, \mathcal{F})$ is denoted by $\mathrm{End}_{\Bbbk}(\mathcal{F})$. 
	
	Let $X$ be a superscheme. In what follows a sheaf of $\mathcal{O}_X$-supermodules is a sheaf of left  $\mathcal{O}_X$-supermodules, unless stated otherwise.
	If  $\mathcal{F}$ and $\mathcal{G}$ are sheaves of $\mathcal{O}_X$-supermodules, then $\mathrm{Hom}_{\mathcal{O}_X}(\mathcal{F}, \mathcal{G})$ is a subsheaf of
	$\mathrm{Hom}_{\Bbbk}(\mathcal{F}, \mathcal{G})$ such that for any open subsets $V\subseteq U\subseteq |X|$ a section $s\in \mathrm{Hom}_{\Bbbk}(\mathcal{F}, \mathcal{G})(U)$ belongs to
	$\mathrm{Hom}_{\mathcal{O}_X}(\mathcal{F}, \mathcal{G})(U)$ if and only if
	\[fs(V)(g)=(-1)^{|s||f|} s(V)(fg), f\in\mathcal{O}(V),  g\in\mathcal{F}(V).\]
	\begin{rem}\label{left is right}
		The category of left $\mathcal{O}_X$-supermodules is equivalent to the category of right $\mathcal{O}_X$-supermodules. More pecisely, if $\mathcal{F}$ is a left $\mathcal{O}_X$-supermodule, then it is the right $\mathcal{O}_X$-supermodule via 
		\[ fa=(-1)^{|f||a|} af, a\in\mathcal{O}_X(U), f\in\mathcal{F}(U), U\subseteq |X|. \]	
		Moreover, if $s\in\mathrm{Hom}_{\mathcal{O}_X}(\mathcal{F}, \mathcal{G} )(U)$, then $s$ is the morphism of the right $\mathcal{O}_X|_U$-supermodules via
		\[ (f)s(V)=(-1)^{|s||f|}s(V)(f), \mathcal{F}(V), V\subseteq U\subseteq |X|.\]	
	\end{rem}		
	\begin{lm}\label{dual of coherent}
		Assume that $X$ is a Noetherian superscheme. If $\mathcal{F}$ and $\mathcal{G}$ are coherent sheaves of $\mathcal{O}_X$-supermodule, then $\mathrm{Hom}_{\mathcal{O}_X}(\mathcal{F}, \mathcal{G})$ is coherent as well. 	
	\end{lm}
	\begin{proof}
		Choose a finite covering of $X$ by affine open supersubschemes $U$. Then $\mathcal{F}(U)\simeq\widetilde{M}$ and $\mathcal{G}(U)\simeq\widetilde{N}$ for each $U$ (see the proof of  \cite[Proposition 3.1]{zub2}), where $M$ and $N$ are finitely generated $\mathcal{O}(U)$-supermodule. Then \cite[Proposition 2.1(1)]{zub2} implies
		\[\mathrm{Hom}_{\mathcal{O}_X}(\mathcal{F}, \mathcal{G})(U)\simeq \widetilde{\mathrm{Hom}_{\mathcal{O}(U)}(M, N) }\]
		for each $U$. 
		Since $\mathcal{O}(U)$ is a Noetherian superalgebra, it remains to note that $\mathrm{Hom}_{\mathcal{O}(U)}(M, N)$ can be naturally identified with an $\mathcal{O}(U)$-supersubmodule of the supermodule
		$N^{m|n}\oplus \Pi N^{m|n}\simeq N^{m+n|m+n}$, where $m$ and $n$ are the numbers of even and odd generators of $M$.
	\end{proof}
	\begin{pr}\label{global sections are finite} (compare with \cite[Corollary 2.36]{brp})
		Let $X$ be a proper superscheme and let $\mathcal{F}$ be a coherent sheaf of $\mathcal{O}_X$-supermodules. 
		Then $\mathrm{H}^i(X, \mathcal{F})$ is a finite dimensional superspace for any $i\geq 0$. 	
	\end{pr}
	\begin{proof}
		The superideal sheaf $\mathcal{J}_X$ is nilpotent. Thus there is an nonnegative integer $k$ such that $\mathcal{J}_X^k\mathcal{F}\neq 0$ but $\mathcal{J}_X^{k+1}\mathcal{F}=0$. We call $k$ a \emph{nilpotency depth} of $\mathcal{F}$. The sheaves $\mathcal{J}_X\mathcal{F}$ and $\mathcal{F}/\mathcal{J}_X\mathcal{F}$ are coherent (cf. \cite[Corollary 3.2]{zub2}) and both are of nilpotency depth strictly less than $k$. The long exact sequence of cohomologies 
		\[ \ldots \to \mathrm{H}^i(X, \mathcal{J}_X\mathcal{F}) \to \mathrm{H}^i(X, \mathcal{F}) \to \mathrm{H}^i(X, \mathcal{F}/\mathcal{J}_X\mathcal{F}) \to \ldots\]
		and the induction on $k$ allows us to consider the case $k=0$ only. In this case $\mathcal{F}$ is a direct sum of two coherent $\mathcal{O}_{X_0}$-modules $\mathcal{F}_0$ and $\mathcal{F}_1$, hence  \cite[Theorem III.3.2.1]{EGA  III} concludes the proof. 
	\end{proof}
	\begin{lm}\label{represented by a space}
		Let $\mathcal{F}$ and $\mathcal{G}$ be coherent sheaves of $\mathcal{O}_N$-supermodules, where $N$ is a proper superscheme. Then the functor \[A\mapsto \mathrm{Hom}_{\mathcal{O}_{N\times\mathrm{Spec}(A)}}(\mathrm{pr}_1^*\mathcal{F}, \mathrm{pr}_1^*\mathcal{G})(|N\times\mathrm{Spec}(A)|), A\in\mathsf{Alg}_{\Bbbk},\] 
		where $\mathrm{pr}_1$ denote the canonical projection $N\times\mathrm{Spec}(A)\to N$, is isomorphic to a functor $A\mapsto V\otimes A$ for
		some finite dimensional vector superspace $V$.				
	\end{lm}
	\begin{proof}
		Let $\{U_i\}$ be a finite covering of $N$ by open affine supersubschemes. For each couple of indices $i, j$ we also choose a finite covering $\{V_{ijk}\}$ of $U_i\cap U_j$ by open affine supersubschemes. Then any global section from $\mathrm{Hom}_{\mathcal{O}_{N\times\mathrm{Spec}(A)}}(\mathrm{pr}^*_1\mathcal{F}, \mathrm{pr}^*_1\mathcal{G}_A)$ is uniquely defined by its restrictions on 	
		\[\mathrm{Hom}_{\mathcal{O}_{N\times\mathrm{Spec}(A)}}(\mathrm{pr}_1^*\mathcal{F}, \mathrm{pr}_1^*\mathcal{G})(|U_i\times\mathrm{Spec}(A)|)\]\[\simeq\mathrm{Hom}_{\mathcal{O}_{U_i\times\mathrm{Spec}(A)}}(\mathrm{pr}_1^*\mathcal{F}|_{|U_i\times\mathrm{Spec}(A)|}, \mathrm{pr}_1^*\mathcal{G}|_{|U_i\times\mathrm{Spec}(A)|})\simeq \]
		\[\mathrm{Hom}_{\mathcal{O}(U_i)\otimes A}(\mathcal{F}(U_i)\otimes A, \mathcal{G}(U_i)\otimes A)\simeq \mathrm{Hom}_{\mathcal{O}(U_i)}(\mathcal{F}(U_i), \mathcal{G}(U_i)\otimes A) ,  \]
		those are compatible being restricted on the corresponding $V_{ijk}$.
		
		Since any $\mathcal{F}(U_i)$ is a finitely generated $\mathcal{O}(U_i)$-supermodule, the latter $A$-supermodules are isomorphic to
		$\mathrm{Hom}_{\mathcal{O}(U_i)}(\mathcal{F}(U_i), \mathcal{G}(U_i))\otimes A$. Moreover, these isomorphisms are obviously compatible with the restrictions to $V_{ijk}$. Thus it follows easily that
		\[\mathrm{Hom}_{\mathcal{O}_{N\times\mathrm{Spec}(A)}}(\mathrm{pr}_1^*\mathcal{F}, \mathrm{pr}_1^*\mathcal{G})(|N\times\mathrm{Spec}(A)|)\simeq \mathrm{Hom}_{\mathcal{O}_N}(\mathcal{F}, \mathcal{G})(|N|)\otimes A. \]
		Proposition \ref{global sections are finite} concludes the proof.
	\end{proof}	
	
	\section{\v{C}ech cohomology}
	
	Let $X$ be a topological space and let $\mathcal{F}$ be a sheaf of $\mathbb{Z}_2$-graded abelian groups. Following \cite[III.4]{hart}, with any open covering $\mathfrak{U}=\{U_i \}_{i\in I}$ of $X$ one can associate a complex of $\mathbb{Z}_2$-graded abelian groups $C^{\bullet}(\mathfrak{U}, \mathcal{F})=\{ C^p(\mathfrak{U}, \mathcal{F})\}_{p\geq 0}$. Without loss of generality, one can asume that $I$ is totally ordered. Then 
	\[C^p(\mathfrak{U}, \mathcal{F})=\prod_{i_0< \ldots < i_p} \mathcal{F}(U_{i_0}\cap\ldots\cap U_{i_p}) \]
	and the coboundary map $d : C^p(\mathfrak{U}, \mathcal{F})\to C^{p+1}(\mathfrak{U}, \mathcal{F})$ is defined as
	\[d(\alpha)_{i_0 , \ldots ,i_{p+1}}=\sum_{0\leq k\leq p+1}(-1)^k\alpha_{i_0, \ldots \widehat{i_k}, \ldots , i_{p+1} }|_{U_{i_0}\cap\ldots\cap U_{i_{p+1}} }, \]
	where $\alpha=\prod_{i_0< \ldots < i_p}\alpha_{i_0, \ldots, i_p}$. Since $d$ preserves $\mathbb{Z}_2$-grading, the \emph{\v{C}ech cohomology groups} $\mathrm{\check{H}}^p(\mathfrak{U}, \mathcal{F})$ are $\mathbb{Z}_2$-graded as well. 
	
	We have a natural resolution $0\to\mathcal{F}\stackrel{\epsilon}{\to}  \mathcal{C}^{\bullet}(\mathfrak{U}, \mathcal{F})$, i.e. the exact sequence of sheaves
	\[0\to\mathcal{F}\stackrel{\epsilon}{\to}\mathcal{C}^0(\mathfrak{U}, \mathcal{F})\to\mathcal{C}^1(\mathfrak{U}, \mathcal{F})\to\ldots , \]
	where  $\mathcal{C}^p(\mathfrak{U}, \mathcal{F})(V)=\prod_{i_0<\ldots < i_p}\mathcal{F}(U_{i_0}\cap \ldots \cap U_{i_p}\cap V)$ 
	for any open subset $V\subseteq X$. The proof of exactness is the same as in the non-graded case (cf. \cite[Lemma III.4.2]{hart}).
	Note also that $\mathcal{C}^{\bullet}(\mathfrak{U}, \mathcal{F})(X)=C^{\bullet}(\mathfrak{U}, \mathcal{F})$.
	
	Recall that a sheaf $\mathcal{F}$ is said to be \emph{flat} (or \emph{flabby}), if for any open subsets $V\subseteq U\subseteq X$ the "restriction" map $\mathcal{F}(U)\to\mathcal{F}(V)$ is surjective. Using \cite[Lemma 1.2]{zub2}, one can mimic the proof of \cite[Proposition III.4.3]{hart} to show that $\mathrm{\check{H}}^p(\mathfrak{U}, \mathcal{F})=0$ for any $p> 0$, provided $\mathcal{F}$ is flat.
	
	If $0\to\mathcal{F}\to\mathcal{J}^{\bullet}$ is an injective resolution of a sheaf $\mathcal{F}$, then \cite[Lemma XX.5.2]{slang} implies that there is a morphism of
	of $\mathbb{Z}_2$-graded complexes $\mathcal{C}^{\bullet}(\mathfrak{U}, \mathcal{F})\to \mathcal{J}^{\bullet}$, that extends the identity morphism $\mathrm{id}_{\mathcal{F}}$, and it is unique up to a homotopy preserving $\mathbb{Z}_2$-grading.  
	Thus this morphism induces the natural maps $\mathrm{\check{H}}^p(\mathfrak{U}, \mathcal{F})\to \mathrm{H}^p(X, \mathcal{F})$ of $\mathbb{Z}_2$-graded abelian groups for each $p\geq 0$, functorial in $\mathcal{F}$ (cf. \cite[Lemma III.4.4]{hart}). 
	\begin{pr}\label{isomorphism of cohomologies}
		If $X$ is a Noetherian separated superscheme, $\mathfrak{U}$ is an open covering of $X$ by affine supersubschemes,  and $\mathcal{F}$ is a quasi-coherent sheaf of $\mathcal{O}_X$-supermodules, then each map $\mathrm{\check{H}}^p(\mathfrak{U}, \mathcal{F})\to \mathrm{H}^p(X, \mathcal{F})$ is an isomorphism. 	
	\end{pr}
	\begin{proof}
		The proof can be copied from \cite[Theorem III.4.5]{hart} verbatim. Indeed, all we need is \cite[Lemma 3.2]{zub2} and Lemma \ref{Ex.II.4.3}.
	\end{proof}
	\begin{cor}\label{affine trick}(see \cite[Exercise III.4.1]{hart})
		Let $f : X\to Y$ be an affine morphism of Noetherian separated superschemes. Let $\mathcal{F}$ be a quasi-coherent sheaf of $\mathcal{O}_X$-supermodules. Then
		\[\mathrm{H}^i(X, \mathcal{F})\simeq \mathrm{H}^i(Y, f_*\mathcal{F}).\]
		for any $i\geq 0$.	
	\end{cor}	
	\begin{proof}
		Let $\mathfrak{V}=\{V_i\}_{i\in I}$ be an open covering of $Y$ by affine supersubschemes. Then $\mathfrak{U}=\{f^{-1}(V_i)\}_{i\in I}$ is an open covering of $X$ by affine supersubschemes also. Since $\mathcal{C}^{\bullet}(\mathfrak{U}, \mathcal{F})=\mathcal{C}^{\bullet}(\mathfrak{V}, f_*\mathcal{F})$, Proposition \ref{isomorphism of cohomologies} concludes the proof. 	
	\end{proof}
	
	\section{Pro-representable functors}	
	
	Let $\mathcal{C}$ be a category with finite products and fibered products as well. Let $F : \mathcal{C}\to\mathsf{Sets}$ be a covariant functor. 
	The functor $F$ is said to be \emph{left exact} in the sense of \cite{lev}, provided the following conditions hold :
	\begin{enumerate}
		\item if $A$ is a final object in $\mathcal{C}$, then $|F(A)|=1$;
		\item $F(A\times B)\simeq F(A)\times F(B)$ for any objects $A, B$ from $\mathcal{C}$; 
		\item $F$ takes any equalizer $A\to B\rightrightarrows C$ in $\mathcal{C}$
		to the equalizer $F(A)\to F(B)\rightrightarrows F(C)$ in $\mathsf{Sets}$. 
	\end{enumerate}
	The equivalent definitions are listed in \cite[Proposition 1]{lev}.
	
	Assume that $\mathcal{C}$ is also an \emph{Artinian} category, that is for any object $A\in\mathcal{C}$ any collection of subobjects of $A$ has a minimal element.
	As it has been proved in \cite[Proposition 2]{lev}, a functor $F : \mathcal{C}\to\mathsf{Sets}$ is left exact if and only if there is a \emph{projective system}
	$\{ A_i, \phi_{ij}\}_{i, j\in I}$ in $\mathcal{C}$, such that $F(A)\simeq \varinjlim \mathsf{Mor}_{\mathcal{C}}(A_i, A)$.  
	
	Let $\mathsf{SL}_{\Bbbk}$ denote the category of finite dimensional augmented local superalgebras over a field $\Bbbk$, with local morphisms between them. 
	It is clear that $\mathsf{SL}_{\Bbbk}$ is Artinian, and it has finite and fibered products. 
	
	If $A\in \mathsf{SL}_{\Bbbk}$, then its maxmal superideal is denoted by $\mathfrak{M}_A$ and the even component of $\mathfrak{M}_A$ is denoted by
	$\mathfrak{m}_A$, i.e. $\mathfrak{M}_A=\mathfrak{m}_A\oplus A_1$. The augmentation map $A\to\Bbbk$ is denoted by $\epsilon_A$, respectively.
	The full subcategory of $\mathsf{SL}_{\Bbbk}$ consisting of all $A$ with $\mathfrak{M}_A^{n+1}=0$ is denoted by $\mathsf{SL}_{\Bbbk}(n)$.
	The full subcategories of $\mathsf{SL}_{\Bbbk}$ and  $\mathsf{SL}_{\Bbbk}(n)$, consisting of purely even superalgebras, are denoted by
	$\mathsf{L}_{\Bbbk}$ and  $\mathsf{L}_{\Bbbk}(n)$ respectively.
	
	A topological superalgebra $A$ is called \emph{profinite augmented local} superalgebra (p.a.l. superalgebra), if $A\simeq\varprojlim A_i$, where $\{A_i, \phi_{ij}\}$ is a projective system of superalgebras
	in $\mathsf{SL}_{\Bbbk}$. The maximal superideal $\mathfrak{M}_A$ is isomorphic to $\varprojlim\mathfrak{M}_{A_i}$ and $A$ is Noetherian if and only if $\dim\mathfrak{M}_A/\mathfrak{M}^2_A< \infty$.
	
	By the above, a functor $F : \mathsf{SL}_{\Bbbk}\to \mathsf{Sets}$ is left exact if and only if it is pro-representable, i.e. there is a p.a.l. superalgebra $A\simeq\varprojlim A_i$ such that
	$F(B)$ is the set of all continuous morphisms from $A$ to the discrete superalgebra $B$. 
	
	Let $\Bbbk[\epsilon_0, \epsilon_1]$ be a superalgebra of \emph{dual supernumbers} (cf. \cite{maszub2}), where $|\epsilon_0|=0, |\epsilon_1|=1,  \epsilon_i\epsilon_j=0, i, j=0, 1$. 
	Then $F(\Bbbk[\epsilon_0, \epsilon_1])\simeq (\dim\mathfrak{M}_A/\mathfrak{M}^2_A )^*$. Thus the functor $F$ is pro-representable by a \emph{complete local augmented Noetherian} superalgebra (c.l.a.N. superalgebra),
	that is $F$ is \emph{strictly pro-representable} in the terminology of \cite{mats-oort}, if and only if $F(\Bbbk[\epsilon_0, \epsilon_1])$ is a finite dimensional vector (super)space.
	
	One can extend the above notion of strict pro-representability for arbitrary $\Bbbk$-functor $\mathbb{X}$. In other words, $\mathbb{X}$ is strictly pro-representable by a c.l.a.N. superalgebra $A$, if for any
	$B\in\mathsf{SAlg}_{\Bbbk}$ the set $\mathbb{X}(B)$ consists of all continuous morphisms from $A$ to the discrete superalgebra $B$.   
	
	\section{Formal neighborhoods of $\Bbbk$-points}
	
	Let $\mathbb{X}$ be a $\Bbbk$-functor. If $x\in\mathbb{X}(\Bbbk)$, then we have the functors $\mathbb{X}_{n, x} : \mathsf{SL}_{\Bbbk}(n)\to\mathsf{Sets}$, defined as
	\[\mathbb{X}_{n, x}(A)=\{y\in\mathbb{X}(A)\mid \mathbb{X}(\epsilon_A)(y)=x \}, A\in\mathsf{SL}_{\Bbbk}(n), \]
	and $\mathbb{X}_x : \mathsf{SL}_{\Bbbk}\to \mathsf{Sets}$, defined as
	\[\mathbb{X}_{x}(A)=\{y\in\mathbb{X}(A)\mid \mathbb{X}(\epsilon_A)(y)=x \}, A\in\mathsf{SL}_{\Bbbk}, \]
	respectively. It is clear that $\mathbb{X}_{x}(A)=\mathbb{X}_{n, x}(A)$ provided $A\in\mathsf{SL}_{\Bbbk}(n)$. 
	
	For exmple, if $\mathbb{G}$ is a group $\Bbbk$-functor and $e$ is the unit element of the group $\mathbb{G}(\Bbbk)$, then $\mathbb{G}_e$ is a group functor as well, analogous to the group functor $P_e$ from \cite{mats-oort}.
	
	For any $A\in\mathsf{SAlg}_{\Bbbk}$ set 
	\[<\mathbb{X}_{n, x}>(A)=\cup_{B\in\mathsf{SL}_{\Bbbk}(n), \phi\in\mathrm{SSp}(B)(A)}\mathbb{X}(\phi)(\mathbb{X}_{n, x}(B))\]
	and
	\[<\mathbb{X}_x>(A)=\cup_{B\in\mathsf{SL}_{\Bbbk}, \phi\in\mathrm{SSp}(B)(A)}\mathbb{X}(\phi)(\mathbb{X}_x(B)),\]
	respectively. It is clear that all $<\mathbb{X}_{n, x}>$ and $<\mathbb{X}_x>$ are subfunctors of the functor $\mathbb{X}$. 
	
	
	Let $\mathbb{X}$ be a superscheme. For a point $x\in \mathbb{X}(\Bbbk)$ and an open affine supersubscheme
	$\mathbb{U}$ of $\mathbb{X}$, such that $x\in\mathbb{U}(\Bbbk)$, let $\mathbb{N}_{n, x}(\mathbb{X})$ denote the closed supersubscheme of $\mathbb{U}$ that isomorphic to $\mathrm{SSp}(\mathcal{O}(\mathbb{U})/\mathfrak{M}^{n+1}_x)$, where $\mathfrak{M}_x$ is the kernel of
	$x : \mathcal{O}(\mathbb{U})\to\Bbbk$. Note that $\mathcal{O}(\mathbb{U})/\mathfrak{M}^{n+1}_x\in\mathsf{SL}_{\Bbbk}(k)$.
	\begin{lm}\label{it does not depend on U}
		The definition of $\mathbb{N}_{n, x}(\mathbb{X})$ does not depend on the choice of $\mathbb{U}$. In particular, $\mathbb{N}_{n, x}(\mathbb{X})$
		is a closed supersubscheme of $\mathbb{X}$.
	\end{lm}
	\begin{proof}
		Let $\mathbb{V}$ be another affine supersubscheme of $\mathbb{X}$, such that $x\in\mathbb{V}(\Bbbk)$. Since there is an open affine supersubscheme $\mathbb{W}\subseteq\mathbb{U}\cap\mathbb{V}$ with $x\in\mathbb{W}(\Bbbk)$, one can assume that $\mathbb{V}\subseteq\mathbb{U}$. The latter embedding is dual to a superalgebra morphism $\phi : A\to B$, where $A=\mathcal{O}(\mathbb{U})$ and
		$B=\mathcal{O}(\mathbb{V})$. Note that the maximal superideal $\mathfrak{M}_x=\ker(x : A\to\Bbbk)$ equals $\phi^{-1}(\mathfrak{N}_x)$, where
		$\mathfrak{N}_x$ is the maximal superideal $\ker(x : B\to\Bbbk)$.
		
		By \cite[Lemma 3.5]{maszub1} there are elements $a_1, \ldots, a_k\in A_0$ such that $\sum_{1\leq i\leq k} B_0\phi(a_i)=B_0$ and $\phi$ induces an isomorphism $A_{a_i}\simeq B_{\phi(a_i)}$ for each $i$. Then the natural embedding $\mathbb{N}_{n, x}(\mathbb{V})\to \mathbb{N}_{n, x}(\mathbb{U})$ is dual to the induced local morphism $\overline{\phi} : A/\mathfrak{M}^{n+1}_x\to B/\mathfrak{N}^{n+1}_x$ of local superalgebras. There is $a_i$ that does not belong to $\mathfrak{M}_x$, hence for its residue $\overline{a_i}$ modulo $\mathfrak{M}^{n+1}_x$ we have the isomorphism
		\[A/\mathfrak{M}^{n+1}_x=(A/\mathfrak{M}^{n+1}_x)_{\overline{a_i}}\to (B/\mathfrak{N}^{n+1}_x)_{\overline{\phi}(\overline{a_i})}=B/\mathfrak{N}^{n+1}_x .\]
		Finally, $\mathbb{N}_{n, x}(\mathbb{X})$ is an affine superscheme, hence a local functor. The obvious superization of \cite[Lemma I.1.13]{jan}
		implies the second statement. Lemma is proven.
	\end{proof}
	\begin{rem}\label{k-th neighborhood}
		If $\mathbb{X}$ is represented by a geometric superscheme $X$, then the closed supersubscheme $\mathbb{N}_{n, x}(\mathbb{X})$ is represented by the {\bf $n$-th neighborhood} of the superscheme morphism $\mathrm{SSpec}(\Bbbk)\to X$ that sends the unique point of the underlying topological space of $\mathrm{SSpec}(\Bbbk)$ to $x$ (see \cite[Section 1.4]{maszub2}).  	
	\end{rem}
	We have an ascending series of closed supersubschemes
	\[\mathbb{N}_{1, x}(\mathbb{X})\subseteq \mathbb{N}_{2, x}(\mathbb{X})\subseteq\ldots   \]
	and hence the subfunctor $\mathbb{N}_{x}(\mathbb{X})=\cup_{n\geq 1}\mathbb{N}_{n, x}(\mathbb{X})$.
	\begin{lm}\label{characterization of neigborhoods}
		If $\mathbb{X}$ is a superscheme, then for any $x\in\mathbb{X}(\Bbbk)$ and $n\geq 1$ we have $\mathbb{N}_x(\mathbb{X})=<\mathbb{X}_x>$	
		and $\mathbb{N}_{n, x}(\mathbb{X})=<\mathbb{X}_{n, x}>$.
	\end{lm}
	\begin{proof}
		Let $B\in\mathsf{SL}_{\Bbbk}(n)$. An element $y\in\mathbb{X}(B)$ belongs to $\mathbb{X}_{n, x}(B)$ if and only if
		$y$ sends the inique point of the underlying topological space $|\mathrm{SSpec}(B)|$ to $x$. If $U$ is an open supersubscheme of $X$ such that $x\in |U|$, then $y\in\mathbb{U}(B)$, hence $\mathbb{X}_{n, x}(B)\subseteq\mathbb{U}(B)$. Thus $<\mathbb{X}_{n, x}>\subseteq \mathbb{U}$ for any $n$, hence $<\mathbb{X}_x>\subseteq\mathbb{U}$ as well. It remains to show that our statement holds for $\mathbb{X}$ to be affine.
		As above, $x\in\mathbb{X}(\Bbbk)$ is a superalgebra morphism $\mathcal{O}(\mathbb{X})\to\Bbbk$ and the superalgebra morphism 
		$y : \mathcal{O}(\mathbb{X})\to A$ belongs to $\mathbb{N}_{n, x}(\mathbb{X})(A)$ if and only if $(\ker x)^{n+1}\subseteq \ker y\subseteq\ker x$. Lemma is proven.	
	\end{proof}
	For any open affine supersubscheme $\mathbb{U}$ in $\mathbb{X}$ such that $x\in\mathbb{U}(\Bbbk)$, the stalk $\mathcal{O}_{X, x}=\mathcal{O}_x$ is isomorphic to $\mathcal{O}(\mathbb{U})_{\mathfrak{M}_x}$. Moreover, if $\mathbb{X}$ is locally algebraic, then
	$\widehat{\mathcal{O}_x}\simeq \varprojlim_k \mathcal{O}(\mathbb{U})/\mathfrak{M}^{n+1}_x$ is a c.l.a.N. superalgebra.
	
	As in \cite[Section 9]{maszub2} one can note that $\mathbb{N}_x(\mathbb{X})$ is strictly pro-representable, provided $\mathbb{X}$ is a locally algebraic superscheme.  Indeed, for any $A\in\mathsf{SAlg}_{\Bbbk}$ the set $\mathbb{N}_x(\mathbb{X})(A)$ can be naturally identified with the set of all continuous superalgebra morphisms 
	$\widehat{\mathcal{O}_x}\to A$, where $A$ is regarded as a discrete superalgebra. 
	
	Note that an element $a$ of a superalgebra $A$ is nilpotent if and only if its even component $a_0$ is. Thus one can easily derive that
	the set of all nilpotent elements of $A$ is the largest locally nilpotent superideal. We denote this superideal by $\mathrm{nil}(A)$.
	Then $\mathrm{nil}(A)=\mathrm{nil}(A_0)\oplus A_1$.
	
	The following lemma is a folklore.
	\begin{lm}\label{folklore}
		Let $\phi : A\to B$ be an injective superalgebra morphism, such that the induced superscheme morphism $\mathrm{SSpec}(\phi)$
		is surjective on the underlying topological spaces. Then for any (not necessary of locally finite type) superscheme $\mathbb{X}$ the map
		$\mathbb{X}(A)\to\mathbb{X}(B)$ is injective.  
	\end{lm} 
	\begin{proof}
		Let $\pi$ denote $\mathrm{SSpec}(\phi)$. The statement of lemma is equivalent to the following : if $x, y\in\mathbb{X}(A)$ such that $x\pi=y\pi$, then
		$x=y$. It is clear that for any open affine supersubscheme $U$ of $X$ there holds $x^{-1}(U)=y^{-1}(U)$. Thus all we need is to show that the restrictions of $x$ and $y$ on any open affine supersubscheme $\mathrm{SSpec}(A_a)\subseteq x^{-1}(U)=y^{-1}(U)$, where $a\in A_0$, coincide to each other. Since $A_a\to B_a$ satisfies the conditions of lemma, the proof reduces to the case when $X$ is affine and the statement of lemma is obvious.   	
	\end{proof}
	The following lemma is a partial superization of \cite[I.3]{mur}. 
	\begin{lm}\label{inductive limit}
		Let $\mathbb{X}$ be a locally algebraic superscheme. Then for any inductive system of superalgebras $\{A_{\alpha}\}_{\alpha\in I}$ the canonical map $p :  \varinjlim\mathbb{X}(A_{\alpha})\to \mathbb{X}(\varinjlim A_{\alpha})$ is bijective. 	
	\end{lm}
	\begin{proof}
		Let $A$ denote $\varinjlim A_{\alpha}$. Assume that we have two morphisms $x : \mathrm{SSpec}(A_{\alpha})\to X$ and $y : \mathrm{SSpec}(A_{\beta})\to X$ such that the diagram
		\[\begin{array}{ccc}
			\mathrm{SSpec}(A) & \to & \mathrm{SSpec}(A_{\alpha}) \\ 
			\downarrow &  & \downarrow x\\
			\mathrm{SSpec}(A_{\beta}) & \stackrel{y}{\to} & X
		\end{array}  \] 
		is commutative. There are a finite collection of affine open supersubschemes $U_i$ in $X$ and two finite collections of elements
		$x_{ij}\in A_{\alpha}, y_{is}\in A_{\beta}$, where $1\leq i\leq k, 1\leq j\leq m, 1\leq s\leq n$, such that 
		\begin{enumerate}
			\item[(a)] $\mathrm{SSpec}((A_{\alpha})_{x_{ij}})\subseteq x^{-1}(U_i)$ and $\mathrm{SSpec}((A_{\beta})_{y_{is}})\subseteq y^{-1}(U_i)$ for each  triple $i, j, s$;
			\item[(b)] $\cup_{1\leq i\leq k, 1\leq j\leq m}\mathrm{SSpec}((A_{\alpha})_{x_{ij}})=\mathrm{SSpec}(A_{\alpha})$ and 
			$\cup_{1\leq i\leq k, 1\leq s\leq n}\mathrm{SSpec}((A_{\beta})_{y_{is}})=\mathrm{SSpec}(A_{\beta})$.
		\end{enumerate}
		The commutativity of the above diagram is equivalent to the commutativity of  the diagram 
		\[ \begin{array}{ccc}
			A_{x_{ij}y_{is}} & \leftarrow & (A_{\alpha})_{x_{ij}} \\
			\uparrow & & \uparrow \\
			(A_{\beta})_{y_{is}} & \leftarrow & \mathcal{O}(U_i)  
		\end{array} \]
		for any triple $i, j, s$.  Since each superalgebra $\mathcal{O}(U_i)$ is finitely generated, there is $\gamma\geq\alpha, \beta$ such that
		these diagrams remain commutative even if we replace $A$ by $A_{\gamma}$. In other words, the images of $x$ and $y$ in $\mathbb{X}(A_{\gamma})$ coincide to each other, hence $p$ is injective.
		
		Further, if $x\in\mathbb{X}(A)$, then $x$ is uniquely defined by its restrictions $\mathrm{SSpec}(A_{x_{ij}})\to U_i$, where $\mathrm{SSpec}(A_{x_{ij}})$ form a finite covering of $\mathrm{SSpec}(A)$ by open affine supersubschemes. 
		Again, since each $\mathcal{O}(U_i)$ is finitely generated, all the corresponding superalgebra morphism $\mathcal{O}(U_i)\to A_{x_{ij}}$ factor
		through some $(A_{\alpha})_{x_{ij}}$, such that any $x_{ij}$ belongs to $A_{\alpha}$ as well. Thus $p$ is surjective.     
	\end{proof}
	For any superalgebra $A$ we define a local supersubalgebra $\mathsf{n}(A)=\Bbbk\oplus\mathrm{nil}(A)$. The natural augmentation map $\mathsf{n}(A)\to\Bbbk$ is denoted by the same symbol $\epsilon_A$. The latter does not lead to confusion with the previous notation.
	Let $\pi_A$ and $\iota_A$ also denote the quotient superalgebra morphism
	$A\to A/\mathrm{nil}(A)=\widetilde{A}$ and the natural injection $\mathsf{n}(A)\to A$ respectively.
	\begin{pr}\label{another characterization of neighborhood}
		Let $\mathbb{X}$ be a locally algebraic superscheme. For any $A\in\mathsf{SAlg}_{\Bbbk}$ we define 
		$\mathbb{X}_x(\mathsf{n}(A))=\{y\in\mathbb{X}(\mathsf{n}(A))\mid  \mathbb{X}(\epsilon_A)(y)=x\}$. Then the injection $\iota_A$ induces the bijection $\mathbb{X}_x(\mathsf{n}(A))\to \mathsf{N}_x(\mathbb{X})(A)$. 
	\end{pr}
	\begin{proof}
		It is clear that the collection of all (augmented) finitely generated supersubalgebras $C_{\alpha}$ of $\mathsf{n}(A)$ form an inductive system, such that $\mathsf{n}(A)=\varinjlim C_{\alpha}=\cup_{\alpha\in I}C_{\alpha}$. Besides, each $C_{\alpha}$ belongs to $\mathsf{SL}_{\Bbbk}$. Lemma \ref{folklore} implies that $\mathbb{X}_x(\mathsf{n}(A))\to\mathbb{X}(A)$ and each $\mathbb{X}_x(C_{\alpha})\to \mathbb{X}(A)$ are injective. Since $\mathbb{X}_x(\mathsf{n}(A))\simeq\varinjlim\mathbb{X}_x(C_{\alpha})$, the map $\mathbb{X}_x(\mathsf{n}(A))\to\mathbb{X}(A)$ is the direct limit of the maps
		$\mathbb{X}_x(C_{\alpha})\to \mathbb{X}(A)$, hence the image $\mathbb{X}_x(\mathsf{n}(A))$ is contained in $\mathsf{N}_x(\mathbb{X})(A)$.
		Conversely, for any $B\in\mathsf{SL}_{\Bbbk}(n)$ and any superalgebra morphism $B\to A$ the map $\mathbb{X}_{n, x}(B)\to\mathbb{X}(A)$
		obviously factors through $\mathbb{X}_x(\mathsf{n}(A))\to\mathbb{X}(A)$. Proposition is proven. 
	\end{proof}
	Let $x_A$ denote the image of $x$ in $\mathbb{X}(A)$ with respect to the natural map $\mathbb{X}(\Bbbk)\to \mathbb{X}(A)$. In particular,
	$x_{\Bbbk}=x$. 
	\begin{pr}\label{final about neighborhood}
		Let $\mathbb{X}$ be a locally algebraic superscheme. Then $\mathbb{N}_x(\mathbb{X})(A)=\mathbb{X}(\pi_A)^{-1}(x_{\widetilde{A}})$.	
	\end{pr}
	\begin{proof}
		Recall that an element $y\in\mathbb{X}(A)$ is interpreted as a superscheme morphism $\mathrm{SSpec}(A)\to X$. Then its image in $\mathbb{X}(\widetilde{A})$ coincides with $z=y \mathrm{SSpec}(\pi_A)$. Moreover,
		$z=x_{\widetilde{A}}$ if and only if the diagram 
		\[\begin{array}{ccc}
			X & \stackrel{x}{\leftarrow} & \mathrm{SSpec}(\Bbbk) \\ 
			y \uparrow &  & \uparrow \\
			\mathrm{SSpec}(A) & \stackrel{\mathrm{SSpec}(\pi_A)}{\leftarrow} & \mathrm{SSpec}(\widetilde{A})
		\end{array}  \] 
		is commutative. Similarly, $y$ belongs to $\mathbb{N}_x(\mathbb{X})(A)=\mathbb{X}_x(\mathsf{n}(A))$ if and only if the diagram
		\[ \begin{array}{ccc}
			X & \stackrel{x}{\leftarrow} & \mathrm{SSpec}(\Bbbk) \\
			y \uparrow & \nwarrow  & \downarrow \\
			\mathrm{SSpec}(A) & \stackrel{\mathrm{SSpec}(\iota_A)}{\rightarrow}& \mathrm{SSpec}(\mathsf{n}(A))
		\end{array}\]
		is commutative for some morphism $\mathrm{SSpec}(\mathsf{n}(A))\to X$. Therefore, all we need is to prove that $y$ makes the first diagram commutative if and only if $y$ does the second one.
		
		Since $\mathsf{n}(A)\stackrel{\iota_A}{\to} A\stackrel{\pi_A}{\to}\widetilde{A}$ factors through the natural injection $\Bbbk\to\widetilde{A}$, Proposition \ref{another characterization of neighborhood} implies $\mathbb{N}_x(\mathbb{X})(A)\subseteq \mathbb{X}(\pi_A)^{-1}(x_{\widetilde{A}})$. It remains to prove
		"only if" part. 
		
		Since $|\mathrm{SSpec}(\pi_A)|$ is a homeomorphism of underlying topological spaces, the commutativity of the first diagram 
		implies that $y$ sends all points of $|\mathrm{SSpec}(A)|$ to $x$. Without loss of a generality, one can replace
		$X$ by any open affine supersubscheme $U$ such that $x\in\mathbb{U}(\Bbbk)$. Then we have a commutative diagram
		\[ \begin{array}{ccc}
			\mathcal{O}(U) & \stackrel{x}{\rightarrow} & \Bbbk \\ 
			y \downarrow &  & \downarrow \\
			A & \stackrel{\pi_{A}}{\rightarrow} & \widetilde{A}
		\end{array}.\]
		This immediately implies that $y$ factors through $\iota_A$, so that the diagram
		\[ \begin{array}{ccc}
			\mathcal{O}(U) & \stackrel{x}{\rightarrow} & \Bbbk \\ 
			y \downarrow & \searrow & \uparrow \\
			A & \stackrel{\iota_{A}}{\leftarrow} & \mathsf{n}(A)
		\end{array},\]
		is commutative. Proposition is proven.
	\end{proof}
\begin{rem}\label{N_x as N.f. superscheme}
In terms of Section $3$, $\mathbb{N}_x(\mathbb{X})$ is just the functor of points of the formal completion of $X$ along the closed supersubscheme $Y\simeq \mathrm{SSpec}(\Bbbk)$, such that
$|Y|=\{ x\}$. In fact, it has been already explained in the penultimate paragraph before Lemma \ref{folklore}. Proposition \ref{final about neighborhood} gives a constructive description
of $\mathbb{N}_x(\mathbb{X})$ in terms of $\Bbbk$-functors only.
\end{rem}
	\begin{cor}\label{normality of N_e}(see \cite[Lemma 9.5]{maszub2})
		Let $\mathbb{G}$ be a locally algebraic group superscheme. Then {\bf the formal neighborhood of the identity} $\mathbb{N}_e(\mathbb{G})$ is a normal group subfunctor of $\mathbb{G}$,
		strictly pro-representable by the c.l.a.N. Hopf superalgebra $\widehat{\mathcal{O}_{G, e}}$ (cf. \cite[Definition 3.3]{hmt}).
	\end{cor}

	\section{Automorphism group functor of a sheaf on site}
	
	Let $\mathcal{C}$ be a \emph{site}, that is a category equipped with a Grothendieck topology. More precisely, for any object $U$ of $\mathcal{C}$ there is a collection of sets of arrows $\{U_i\to U \}$, called \emph{coverings} of $U$, such that the following conditions hold :
	\begin{enumerate}
		\item if $V\to U$ is an isomorphism, then $\{V\to U \}$ is a covering;
		\item if $\{U_i\to U \}$ is a covering and $V\to U$ is any arrow, then all fibered products $\{ U_i\times_U V\}$ exist and 
		$\{U_i\times_U V\to V \}$ is a covering;
		\item if $\{U_i\to U  \}$ is a covering, and for any $i$ we have a covering $\{V_{ij}\to U_i\}$, then $\{V_{ij}\to U_i\to U \}$ is a covering.
	\end{enumerate}  
	
	A \emph{sheaf} on site $\mathcal{C}$ is a (contravariant) functor such that for any covering $\{ U_i\to U \}$ the natural sequence
	\[ X(U) \to \prod_{i} X(U_i)\rightrightarrows\prod_{i, j} X(U_i\times_U U_j) \] 
	is exact (cf. \cite{fundalg}, 2.3.3).
	
	Let $\mathcal{C}$ be a category and $X : \mathcal{C}^{op}\to\mathsf{Sets}$ be a covariant functor.
	
	For an object $A\in\mathcal{C}$ let $\mathcal{C}_A$ denote a category of pairs $(A', \phi)$, where $A'\in\mathcal{C}, \phi\in\mathsf{Mor}_{\mathcal{C}}(A', A)$, with morphisms $(A', \phi)\to (A'', \gamma)$ are just morphisms $\xi : A'\to A''$ in $\mathcal{C}$ such that $\gamma\xi=\phi$. Any object from $\mathcal{C}_A$ is called an \emph{$A$-object} as well as any morphism in $\mathcal{C}_A$ is called an \emph{$A$-morphism}. The functor $X$ induces a functor $X_A : \mathcal{C}_A^{op}\to\mathsf{Sets}$, such that $X((A', \phi))=X(A')$. 
	
	We define the \emph{automorphism group functor}
	\[\mathfrak{Aut}(X) : \mathcal{C}^{op}\to \mathsf{Gr}\] of the functor $X$ as 
	\[\mathfrak{Aut}(X)(A)=\{f\in\mathfrak{Mor}_{\mathcal{C}_A}(X_A, X_A)\mid f \ \mbox{is \ invertible}\}, A\in\mathcal{C}.\]
	Observe that for any morphism $\alpha : A\to B$ in $\mathcal{C}$, there is a functor $\alpha_* : \mathcal{C}_A\to\mathcal{C}_{B}$ that sends a pair $(A', \phi)$ to $(A', \alpha\phi)$.  
	In particular, any functor morphism $f : X_B\to X_B$ can be "restricted"  to a functor morphism $f_A : X_A\to X_A$, and the map $f\mapsto f_A$ defines a group morphism $\mathfrak{Aut}(X)(B)\to\mathfrak{Aut}(X)(A)$.
	\begin{theorem}\label{category_locality}
		Assume that $\mathcal{C}$ is a site and $X$ is a sheaf on $\mathcal{C}$. Then the group functor $\mathfrak{Aut}(X)$ is a sheaf on $\mathcal{C}$ with respect to the same Grothendieck topology.
	\end{theorem}
	\begin{proof}
		Consider an open covering $\{\alpha_i : A_i\to A\}$ of an object $A$. Assume that $\prod_{i\in I} f_i\in \prod_{i\in I}\mathfrak{Aut}(X)(A_i)$ belongs to
		the kernel of the map
		\[\prod_{i\in I}\mathfrak{Aut}(X)(A_i)\rightrightarrows\prod_{i, j\in I}\mathfrak{Aut}(X)(A_i\times_A A_j),\]
		that is for any couple of indexes $i, j$ there holds $(f_i)_{A_i\times_A A_j}=(f_j)_{A_i\times_A A_j}$. 
		More precisely, for any $A_i\times_A A_j$-object $B$, regarded as $A_i$-object and $A_j$-object via the canonical "projections" $A_i\times_A A_j\to A_i$ and $A_i\times_A A_j\to A_j$ respectively, 
		$f_i(B)=f_j(B)$. 
		
		Let $B$ be an $A$-object. We have an open covering $\{\beta_i : B_i\to B\}$, where $B_i=B\times_A A_i, i\in I$. For any $x\in X(B)$ set $X(\beta_i)(x)=x_i\in X(B_i)$. Then the element $\prod_{i\in I} f_i(x_i)$ belongs to the kernel of the map
		\[\prod_{i\in I}X(B_i)\rightrightarrows\prod_{i, j\in I, i\neq j}X(B_i\times_B B_j).\]
		Indeed, the universal property of fibred products implies that there is a canonical morphism $B_i\times_B B_j\to A_i\times_A A_j$ that makes the diagrams 
		\[\begin{array}{ccc}
			B_i\times_B B_j & \stackrel{\pi_i}{\to} & B_i \\
			\downarrow & & \downarrow  \\
			A_i\times_A A_j & \to & A_i
		\end{array}\]
		and 
		\[\begin{array}{ccc}
			B_i\times_B B_j & \stackrel{\pi_j}{\to} & B_j \\
			\downarrow & & \downarrow  \\
			A_i\times_A A_j & \to & A_j
		\end{array}\]
		commutative, that is $\pi_i$ and $\pi_j$ are $A_i$-morphism and $A_j$-morphism correspondingly. Thus
		\[X(\pi_i)(f_i(x_i))=f_i(X(\pi_i)(x_i))=(f_i)_{A_i\times_A A_j}(X(\pi_i)(x_i))=\]
		\[(f_j)_{A_i\times_A A_j}(X(\pi_j)(x_j))=f_j(X(\pi_j)(x_j))=X(\pi_j)(f_j(x_j)).\]
		Since $X$ is a sheaf, there is the unique element $z\in X(B)$ such that for each $i$ we have 
		$z_i=f_i(x_i)$. Set $f(x)=z$.
		
		First, we claim that $f_{A_i}=f_i$ for each index $i$. Let $B$ be an $A_i$-object, hence an $A$-object as well, via morphism $\alpha_i$. For any indices $i, j$ the composition of $\beta_j$ and the morphism $B\to A_i$ defines a morphism $B_j\to A_i$, which makes the diagram
		\[\begin{array}{ccc}
			B_j & \to & A_j \\
			\downarrow & & \downarrow \\
			A_i & \to & A
		\end{array}\]
		commutative. Therefore, both morphisms $B_j\to A_j$ and $B_j\to A_i$ factor through the unique morphism
		$B_j\to A_i\times_A A_j$, and thus we have
		\[f_j(x_j)=f_i(x_j)=X(\beta_j)(f_i(x)),\]
		whence $f(x)=f_i(x)$. 
		
		Second, for any $A$-morphism $\gamma : B\to B'$ and any index $i$, there is the unique $A_i$-morphism $\gamma_i : B_i\to B'_i$ such that the diagram
		\[\begin{array}{ccc}
			B_i & \stackrel{\gamma_i}{\to} & B'_i \\
			\downarrow & & \downarrow \\
			B & \to & B'
		\end{array}\]
		commutative. 
		
		Continuing as above, one can find the unique morphism 
		\[\gamma_{ij} : B_i\times_B B_j\to B'_i\times_{B'} B'_j\]
		that makes the diagrams 
		\[\begin{array}{ccc}
			B_i\times_B B_j & \to & B'_i\times_{B'} B'_j \\
			\downarrow & & \downarrow \\
			B_i & \to & B'_i \end{array}
		\]
		and
		\[\begin{array}{ccc}
			B_i\times_B B_j & \to & B'_i\times_{B'} B'_j \\
			\downarrow & & \downarrow \\
			B_j & \to & B'_j \end{array}
		\]
		commutative. Combining all, we obtain the commutative diagram 
		\[\begin{array}{ccccccc}
			X(B) & \to & \prod_{i\in I} X(B_i) & \stackrel{\prod_{i\in I} f_i}{\to} & \prod_{i\in I} X(B_i) & \rightrightarrows & \prod_{i, j\in I, i\neq j} X(B_i\times_B B_j) \\
			\downarrow & & \downarrow & & \downarrow & & \downarrow \\
			X(B') & \to & \prod_{i\in I} X(B'_i) & \stackrel{\prod_{i\in I} f_i}{\to} & \prod_{i\in I} X(B'_i) & \rightrightarrows & \prod_{i, j\in I, i\neq j} X(B'_i\times_{B'} B'_j),
		\end{array}\]
		where the vertical arrows are the maps $X(\gamma), \prod_{i\in I} X(\gamma_i)$ and $\prod_{i, j\in I, i\neq j}\gamma_{ij}$ respectively. It obviously follows that $X(\gamma)(f(x))=f(X(\gamma)(x))$.
		
		We left for the reader to check that $f$ is invertible. Theorem is proved.
	\end{proof}
	
	\section{First properties of $\mathfrak{Aut}(\mathbb{X})$}
	
	\begin{lm}\label{specific intersection}
		Let $X$ be a superscheme and $A$ be a superalgebra. Let $V$ be an open affine supersubscheme of $X\times\mathrm{SSpec}(A)$. 
		Then for any superideal $I$ of $A$ the closed, hence affine, supersubscheme $V\cap (X\times \mathrm{SSpec}(A/I))$ of $V$ is defined by the superideal $\mathcal{O}(V)I$.  	
	\end{lm}
	\begin{proof}
		Assume that $\{ U\}$ is a covering of $X$ by open affine supersubschemes. Then $X\times\mathrm{SSpec}(A)$ is covered by open affine supersubschemes $U\times\mathrm{SSpec}(A)$ as well as
		$X\times\mathrm{SSpec}(A/I)$ is covered by open affine supersubschemes \[U\times\mathrm{SSpec}(A/I)=(U\times\mathrm{SSpec}(A))\cap (X\times\mathrm{SSpec}(A/I)).\]
		
		Each open supersubscheme $V\cap (U\times\mathrm{SSpec}(A))$ is covered by open affine supersubschemes $\mathrm{SSpec}(\mathcal{O}(V)_g), g\in\mathcal{O}(V)_0$. Collecting all such $\mathrm{SSpec}(\mathcal{O}(V)_g)$ we obtain a covering of $V$, hence $\sum_g \mathcal{O}(V)_0 g=\mathcal{O}(V)_0$. 
		Furthermore, there is  \[W=V\cap (X\times\mathrm{SSpec}(A/I))=\mathrm{SSpec}(\mathcal{O}(V)/J)\] for some superideal $J$ of
		$\mathcal{O}(V)$, and $W$ is covered by (finitely many!) open affine supersubschemes $(\mathrm{SSpec}(\mathcal{O}(V)_g))\cap (U\times\mathrm{SSpec}(A/I))$. All we need is to show that \[(\mathrm{SSpec}(\mathcal{O}(V)_g))\cap(U\times\mathrm{SSpec}(A/I))=
		\mathrm{SSpec}((\mathcal{O}(V)/I)_g)\] 
		for each couple $g, U$ and refer to \cite[Corollary 1.1]{zub1}. 
		In other words, it remains to prove that if $V$ is an open affine supersubscheme of some affine superscheme $\mathrm{SSpec}(B)$ and $I$ is a superideal of $B$, then we have \[V\cap\mathrm{SSpec}(B/I)=\mathrm{SSpec}(\mathcal{O}(V)/\mathcal{O}(V)I).\]   
		This statement is an obvious consequence of \cite[Lemma 3.5]{maszub1} and \cite[Corollary 1.1]{zub1} (see also the proof of Lemma \ref{it does not depend on U}).  	
	\end{proof}
	\begin{lm}\label{cancellation}
		Let $X$ be a superscheme over $\mathrm{SSpec}(A)$. Let $A\to B$ be a superalgebra morphism such that $\mathrm{SSpec}(B)\to\mathrm{SSpec}(A)$
		is surjective on the underlying topological spaces. If $V$ and $V'$ are open supersubschemes in $X$ such that \[V\times_{\mathrm{SSpec}(A)}\mathrm{SSpec}(B)=V'\times_{\mathrm{SSpec}(A)}\mathrm{SSpec}(B),\] then $V=V'$.
	\end{lm}
	\begin{proof}
		Translating to the category of $\Bbbk$-functors we have
		\[\mathbb{V}\times_{\mathrm{SSp}(A)}\mathrm{SSp}(B)=\mathbb{V}'\times_{\mathrm{SSp}(A)}\mathrm{SSp}(B).\] 
		In particular, for any field extension $\Bbbk\subseteq L$ there holds
		\[\mathbb{V}(L)\times_{\mathrm{SSp}(A)(L)}\mathrm{SSp}(B)(L)=\mathbb{V}'(L)\times_{\mathrm{SSp}(A)(L)}\mathrm{SSp}(B)(L).\]
		The condition of lemma implies that $\mathrm{SSp}(B)(L)\to \mathrm{SSp}(A)(L)$ is surjective, thus  $\mathbb{V}(L)=\mathbb{V}'(L)$.
		Superizing \cite[I.1.7(4)]{jan} we obtain $\mathbb{V}=\mathbb{V}'$, or $V=V'$.  	
	\end{proof}
	
	For any $A\in\mathsf{SAlg}_{\Bbbk}$ let $\mathfrak{T}(A)$ denote $\ker(\mathfrak{Aut}(\mathbb{X})(\pi_A))$, where $\pi_A : A\to \widetilde{A}$ is the canonical epimorphism.  
	\begin{pr}\label{nice property of T}
		The following statements hold :
		\begin{enumerate}
			\item[(i)]
			for any superalgebra monomorphism $A\to B$ such that $\mathrm{SSpec}(B)\to\mathrm{SSpec}(A)$ is surjective on the underlying topological spaces, the map $\mathfrak{Aut}(\mathbb{X})(A)\to \mathfrak{Aut}(\mathbb{X})(B)$ is injective;
			\item[(ii)] there is $\mathfrak{T}(A)=\mathfrak{Aut}(\mathbb{X})_e(\mathsf{n}(A))=\ker(\mathfrak{Aut}(\mathbb{X})(\epsilon_A))$.	  
		\end{enumerate}	
	\end{pr}
	\begin{proof}
		We start by proving the second assertion, the proof of the first one is similar and will be briefly outlined. 
		
		Recall that the group $\mathfrak{Aut}(\mathbb{X})(A)$ can be identified with the group of $\mathrm{SSpec}(A)$-automorphisms of $X\times \mathrm{SSpec}(A)$, regarded as an  $\mathrm{SSpec}(A)$-superscheme via the projection $\mathrm{pr}_2$. 
		
		Choose a covering of $X$ by open affine supersubschemes $U_i$. The automorphism $\phi\in \mathfrak{Aut}(\mathbb{X})(A)$ is uniquely defined by its restrictions on affine supersubschemes $V_i=\phi^{-1}(U_i\times \mathrm{SSpec}(A))$, i.e. by isomorphisms $\mathcal{O}(U_i)\otimes A\to \mathcal{O}(V_i)$ of $A$-superalgebras. Respectively, Lemma \ref{specific intersection} implies that the image of $\widetilde{\phi}$ in $\mathfrak{Aut}(\mathbb{X})(\widetilde{A})$ is defined by
		the induced morphisms $\mathcal{O}(V_i)/\mathcal{O}(V_i)\mathrm{nil}(A)\to \mathcal{O}(U_i)\otimes\widetilde{A}$.  
		
		If $\widetilde{\phi}$ is the identity automorphism, then \[\mathrm{SSpec}(\mathcal{O}(V_i)/\mathcal{O}(V_i)\mathrm{nil}(A))=\mathrm{SSpec}(\mathcal{O}(U_i))\otimes\widetilde{A}). \]
		On the other hand, since for any $A$-superalgebra $B$ there is $\mathrm{nil}(A)B\subseteq\mathrm{nil}(B)$, it obviously implies that 
		$|V_i|=|U_i\times\mathrm{SSpec}(A)|$, hence 
		$V_i= U_i\times\mathrm{SSpec}(A)$. Therefore, $\phi$ is defined by isomorphisms $\mathcal{O}(U_i)\otimes A\to \mathcal{O}(U_i)\otimes A$
		of $A$-superalgebras. 
		
		Assume that $\mathcal{O}(U_i)$ is generated by the elements $f_1, \ldots, f_t$. Then the isomorphism $\mathcal{O}(U_i)\otimes A\to \mathcal{O}(U_i)\otimes A$ is determined on generators as
		\[f_l\mapsto\sum_k f_{lk}\otimes a_{lk}, f_{lk}\in\mathcal{O}(U_i), a_{lk}\in A, 1\leq l\leq t.  \] 
		Let $\widetilde{a_{lk}}$ denote the images of $a_{lk}$ in $\widetilde{A}$. Without loss of a generality one can assume that the nonzero
		$\widetilde{a_{lk}}$ are linearly independent for each $l$. The condition $\widetilde{\phi}=\mathrm{id}_{X\times \mathrm{SSpec}(\widetilde{A})}$ is equivalent to
		\[f_l\otimes 1= \sum_k f_{lk}\otimes\widetilde{a_{lk}}, 1\leq l\leq t . \]
		In other words, for each $l$ there is only one $k$ such that $\widetilde{a_{lk}}\neq 0$. Moreover, for this index $k$ we have 
		$\widetilde{a_{lk}}=1$ and $f_{lk}=f_l$. Thus it obviously follows that $\phi\in\mathfrak{Aut}(\mathbb{X})_e(\mathsf{n}(A))$ and (ii) is proven. 
		
		Recall that $X\times\mathrm{SSpec}(B)$ is canonically isomorphic to the fibered product \[(X\times\mathrm{SSpec}(A))\times_{\mathrm{SSpec}(A)}\mathrm{SSpec}(B),\]
		so that the image of $\phi$ in $\mathfrak{Aut}(\mathbb{X})(B)$ coincides with $\phi\times_{ \mathrm{SSpec}(A)} \mathrm{id}_{\mathrm{SSpec}(B)}$. If this image is the identity morphism, then \[V_i\times_{\mathrm{SSpec}(A)}\mathrm{SSpec}(B)=(U_i\times\mathrm{SSpec}(A))\times_{\mathrm{SSpec}(A)}\mathrm{SSpec}(B)\]
		for each index $i$. Lemma \ref{cancellation} infers $V_i=U_i\times\mathrm{SSpec}(A)$, and repeating the above arguments with the generators of each $\mathcal{O}(U_i)$, one immediately sees that $\phi=\mathrm{id}_{X\times\mathrm{SSpec}(A)}$. 
	\end{proof}	
	Proposition \ref{nice property of T} implies that the natural group morphism $\mathfrak{Aut}(\mathbb{X})(A_0)\to \mathfrak{Aut}(\mathbb{X})(A)$ is injective for any superalgebra $A$. Therefore, the group functor $A\mapsto \mathfrak{Aut}(\mathbb{X})_{ev}(A)=\mathfrak{Aut}(\mathbb{X})(A_0)$  
	can be regarded as a group subfunctor of $\mathfrak{Aut}(\mathbb{X})$.

	\section{$\mathfrak{Aut}(\mathbb{X})$ commutes with direct limits of superalgebras}
	
	The following lemma is a folklore (see \cite[II, Exercise 3.3]{hart}).
	\begin{lm}\label{finite type property}
		If $Y$ is of finite type over $\mathrm{SSpec}(A)$ and $V$ is an open affine supersubscheme of $Y$, then
		$\mathcal{O}(V)$ is a finitely generated $A$-superalgebra.	
	\end{lm}
	\begin{proof}
		By \cite[Lemma 1.10]{maszub2} there is a finite covering of $Y$ by open affine supersubschemes $U$ such that each $\mathcal{O}(U)$ is a finitely generated $A$-superalgebra. 
		
		Then each $V\cap U$ can be covered by open affine supersubschemes $\mathrm{SSpec}(\mathcal{O}(U)_h), h\in\mathcal{O}(U)_0$. Since any affine superscheme is quasi-compact, $V$ can be covered by finitely many such supersubschemes (in general, for various $U$ and $h$).	
		Applying \cite[Lemma 1.10]{maszub2} again, we obtain that the induced morphism
		$V\to\mathrm{SSpec}(A)$ is of finite type, hence the $A$-superalgebra $\mathcal{O}(V)$ is finitely generated. 
	\end{proof}	
	\begin{lm}\label{variation on Lemma 5.3}
		Let $Y$ be a superscheme of finite type over $\mathrm{SSpec}(A)$. For any superalgebra morphism $A\to B$ and for arbitrary open quasi-compact supersubschemes
		$V$ and $V'$ of $Y$, such that
		\[V\times_{\mathrm{SSpec}(A)} \mathrm{SSpec}(B)=V'\times_{\mathrm{SSpec}(A)} \mathrm{SSpec}(B),  \]
		there is a finitely generated $A$-supersubalgebra $B'$ of $B$, that satisfies
		\[V\times_{\mathrm{SSpec}(A)} \mathrm{SSpec}(B')=V'\times_{\mathrm{SSpec}(A)} \mathrm{SSpec}(B').  \]  	
	\end{lm}
	\begin{proof}
		Choose a covering of $Y$ by open affine supersubschemes $U$ as in Lemma \ref{finite type property}. Since
		\[V\times_{\mathrm{SSpec}(A)} \mathrm{SSpec}(B)=V'\times_{\mathrm{SSpec}(A)} \mathrm{SSpec}(B)  \]
		if and only if
		\[(U\cap V)\times_{\mathrm{SSpec}(A)} \mathrm{SSpec}(B)=(U\cap V')\times_{\mathrm{SSpec}(A)} \mathrm{SSpec}(B)  \]
		for each $U$, one can assume that $Y$ is affine. Then both $V$ and $V'$ are covered by finitely many open supersubschemes $\mathrm{SSpec}(\mathcal{O}(Y)_g)$ and
		$\mathrm{SSpec}(\mathcal{O}(Y)_h)$ respectively, where $g, h\in\mathcal{O}(V)_0$. Again, the original equality is equivalent to the finitely many equalities
		\[\mathrm{SSpec}(\mathcal{O}(Y)_g)\times_{\mathrm{SSpec}(A)} \mathrm{SSpec}(B)=(\cup_h \mathrm{SSpec}(\mathcal{O}(Y)_{gh}))\times_{\mathrm{SSpec}(A)} \mathrm{SSpec}(B)  \]
		for each element $g$. Thus it remains to consider the case when $V=Y$ and $V'=\cup_h \mathrm{SSpec}(\mathcal{O}(Y)_{h})$. We have
		\[\mathrm{SSpec}(\mathcal{O}(Y)\otimes_A B)=\cup_h \mathrm{SSpec}((\mathcal{O}(Y)\otimes_A B)_{h\otimes 1}),  \]
		that is equivalent to
		\[ \sum_h (h\otimes 1)(\sum_i f_{i h}\otimes b_{i h})=\sum_{i, h} hf_{i h}\otimes b_{i h}=1\otimes 1 , \] where $f_{i h}\in\mathcal{O}(V), b_{i h}\in B, |f_{i h}|+|b_{i h}|=0$ for any $i, h$. By setting $B'=A[b_{i h}\mid i, h]$ our lemma follows.
	\end{proof}
	\begin{pr}\label{a typical affine in a product}
		Let $X$ be an algebraic superscheme and let $\{ V_i\}_{1\leq i\leq m}$ be a finite covering of $X\times\mathrm{SSpec}(A)$ by open affine supersubschemes. Then there are a finitely generated $\Bbbk$-supersubalgebra $A'$ of $A$ and a finite covering $\{V'_i \}_{1\leq i\leq m}$ of $X\times\mathrm{SSpec}(A')$ by open affine supersubschemes, such that each diagram
		\[\begin{array}{ccc}
			V_i & \to & X\times\mathrm{SSpec}(A) \\
			\downarrow & & \downarrow \\
			V'_i\times_{\mathrm{SSpec}(A')}\mathrm{SSpec}(A) & \to & (X\times\mathrm{SSpec}(A'))\times_{\mathrm{SSpec}(A')}\mathrm{SSpec}(A)
		\end{array}\]
		is commutative and the left vertical arrow is the isomorphism induced by the natural isomorphism on the right.
	\end{pr}
	\begin{proof}
		Let $\{U_j\}_{1\leq j\leq n}$ be a finite covering of $X$ by open affine supersubschemes. For each couple of indices $i, j$ we choose a finite covering $\{V_{ij k}\}_{1\leq k\leq l}$ of the open supersubscheme $V_i\cap (U_j\times\mathrm{SSpec}(A))$ by open affine supersubschemes. Since the statement of lemma is invariant with respect to the replacements of $A'$ by a larger (finitely generated) supersubalgebra $A''$ and $V'$ by $V'\times_{\mathrm{SSpec}(A')}\mathrm{SSpec}(A'')$ respectively, all we need is to prove our statement for the covering $\{V_{ijk} \}_{1\leq i\leq m, 1\leq k\leq l}$ of $U_j\times\mathrm{SSpec}(A), 1\leq j\leq n$. In other words, one can assume that $X$ is affine. Let $\phi_i$ denote the dual superalgebra morphism $\mathcal{O}(X)\otimes A\to \mathcal{O}(V_i), 1\leq i\leq m$. 	
		
		By \cite[Lemma 3.5]{maszub1}, there are elements $x_{i 1}, \ldots, x_{i t}\in (\mathcal{O}(X)\otimes A)_0$ and $b_{i1}, \ldots, b_{it}\in\mathcal{O}(V)_0$, such that 
		$\sum_{1\leq s\leq t} b_{i s}\phi_i(x_{i s})=1$ and $\phi_i$ induces $(\mathcal{O}(X)\otimes A)_{x_{is}}\simeq \mathcal{O}(V)_{\phi_i(x_{is})}$ for any $1\leq i\leq m, 1\leq s\leq t$. Moreover,  there are elements $d_{is}\in (\mathcal{O}(X)\otimes A)_0,  1\leq i\leq m, 1\leq s\leq t,$ such that $\sum_{1\leq i\leq m, 1\leq s\leq t}x_{is}d_{is}=1$.
		
		For arbitrary indices  $1\leq i\leq m, 1\leq s, k\leq t,$ let $\frac{y_{iks}}{x_{ik}^{n_{iks}}}$ be the unique preimage of $\frac{b_{is}}{1}\in\mathcal{O}(V)_{\phi_i(x_{ik})}$, where $y_{iks}\in (\mathcal{O}(X)\otimes A)_0$.  
		There is a finitely generated $\Bbbk$-supersubalgebra $A'$ of $A$ such that all elements $x_{is}, d_{is}$ and $y_{iks}$ belong to $(\mathcal{O}(X)\otimes A')_0$.
		Set $B_i=\phi_i(\mathcal{O}(X)\otimes A'), 1\leq i\leq m$. By the same \cite[Lemma 3.5]{maszub1}, each $\mathrm{SSpec}(B_i)$ is isomorphic to an open supersubscheme $V'_i$ of
		$X\times\mathrm{SSpec}(A')$ and the induced morphism $V_i\to V'_i\times_{\mathrm{SSpec}(A')}\mathrm{SSpec}(A)$ is an isomorphism. It is also clear that $\cup_{1\leq i\leq m} V_i'=X\times\mathrm{SSpec}(A')$.
	\end{proof}
	\begin{pr}\label{Property P_3}
		Let $\mathbb{X}$ be an algebraic superscheme and let $\{A_{\lambda}  \}_{\lambda\in I}$ be an inductive system of $\Bbbk$-superalgebras. Set $A=\varinjlim A_{\lambda}$. Then
		the natural map
		\[ q : \varinjlim \mathfrak{Aut}(\mathbb{X})(A_{\lambda})\to \mathfrak{Aut}(\mathbb{X})(A)  \]
		is a group isomorphism. 
	\end{pr}
	\begin{proof} More generally, let $\mathfrak{End}(\mathbb{X})$ denote the functor from $\mathsf{SAlg}_{\Bbbk}$ to the category of monoids,
		where for any $A\in\mathsf{SAlg}_{\Bbbk}$ the monoid $\mathfrak{End}(\mathbb{X})(A)$ consists of all endomorphisms of the superscheme $X\times\mathrm{SSpec}(A)$ over $\mathrm{SSpec}(A)$.
		As above, we have the natural map 
		\[ p : \varinjlim \mathfrak{End}(\mathbb{X})(A_{\lambda})\to \mathfrak{End}(\mathbb{X})(A). \]
		It is clear that $\mathfrak{Aut}(\mathbb{X})$ is a submonoid functor of $\mathfrak{End}(\mathbb{X})$ and $p|_{\mathfrak{Aut}(\mathbb{X})}=q$.
		
		Let $\phi\in \mathfrak{End}(\mathbb{X})(A_{\lambda})$ and its equivalence class $[\phi]\in \varinjlim \mathfrak{End}(\mathbb{X})(A_{\lambda})$ satisfies $p([\phi])=\mathrm{id}_{X\times\mathrm{SSpec}(A)}$. Arguing as in Proposition \ref{nice property of T} and applying Lemma \ref{variation on Lemma 5.3}, one sees that there is a superalgebra $A_{\beta}, \beta\geq\gamma,$
		such that the endomorphism $\phi'=\phi\times_{\mathrm{SSpec}(A_{\gamma}) }\mathrm{id}_{\mathrm{SSpec}(A_{\beta})}$ maps each open supersubscheme
		\[ U\times\mathrm{SSpec}(A_{\beta})\simeq (U\times\mathrm{SSpec}(A_{\gamma}))\times_{\mathrm{SSpec}(A_{\gamma})} \mathrm{SSpec}(A_{\beta})\] to itself, where $U$ runs over a finite covering of
		$X$ by open affine supersubschemes. Moreover, $\phi'\times_{\mathrm{SSpec}(A_{\beta}) }\mathrm{id}_{\mathrm{SSpec}(A) }=\mathrm{id}_{X\times\mathrm{SSpec}(A)  }$. Assume that $f_1, \ldots, f_l$ generate a superalgebra $\mathcal{O}(U)$. Then the superalgebra endomorphism of $\mathcal{O}(U)\otimes A_{\beta}$, dual to $\phi'|_{U\times\mathrm{SSpec}(A_{\beta})}$, acts on these generators as
		\[f_i\otimes 1\mapsto \sum_j h_{ij}\otimes a_{ij}, h_{ij}\in\mathcal{O}(U), a_{ij}\in A_{\beta}, 1\leq i\leq l .\]
		Besides, for each index $i$ the elements $h_{ij}$ can be chosen linearly independent. We have
		\[\sum_j h_{ij}\otimes [a_{ij}]=f_i\otimes 1, 1\leq i\leq l, \]
		where $[a]$ is the equivalence class of $a\in A_{\beta}$ in $A$. Since the covering $\{ U\}$ is finite, for sufficiently large $\beta'$ there is
		$\phi'\times_{\mathrm{SSpec}(A_{\beta}) }\mathrm{id}_{\mathrm{SSpec}(A_{\beta'}) }=\mathrm{id}_{X\times\mathrm{SSpec}(A_{\beta'}) }$. In particular, $q$ is an injective
		group homomorphism.
		
		It remains to show that $q$ is surjective. 
		
		Condier an element $\phi\in \mathfrak{Aut}(\mathbb{X})(A)$. Let $\{U_i \}$ be a finite covering of $X$ by open affine supersubschemes. 
		For each couple of indices we choose a covering of $U_i\cap U_j$ by open affine supersubschemes $U_{ij k}$. Then $X\times\mathrm{SSpec}(A)$ is covered by open affine supersubschemes $U_i\times\mathrm{SSpec}(A)$ and $V_i=\phi^{-1}(U_i\times\mathrm{SSpec}(A))$ respectively. Moreover, each $V_i\cap V_j=\phi^{-1}(U_i\cap U_j)$ is covered by open
		affine supersubschemes $V_{ij k}=\phi^{-1}(U_{ij k})$ and $\phi$ is uniquely defined by the superalgebra isomorphisms \[\phi^{\sharp}_i : \mathcal{O}(U_i)\otimes A\to \mathcal{O}(V_i) \ \mbox{and} \ \phi_{ijk}^{\sharp} : \mathcal{O}(U_{ijk})\otimes A\to \mathcal{O}(V_{ijk}),\]
		such that for any triple $i, j, k$ we have a commutative diagram
		\[\begin{array}{ccc}
			\mathcal{O}(U_i)\otimes A & \to &\mathcal{O}(V_i) \\
			\downarrow & & \downarrow \\
			\mathcal{O}(U_{ij k})\otimes A & \to & \mathcal{O}(V_{ij k})
		\end{array}.  \]	
		If it does not lead to confusion, 	we omit the indices $i, j, k$ in the notations $\phi_i^{\sharp}$ and $\phi_{ijk}^{\sharp}$. 
		
		Proposition \ref{a typical affine in a product} implies that there are an index $\gamma$ and open affine supersubschemes $V'_i, V'_{ijk}$ in $X\times\mathrm{SSpec}(A_{\gamma})$ such that
		\[\cup_i V_i'=X\times\mathrm{SSpec}(A_{\gamma})\]
		and
		\[V_i\simeq V'_i\times_{\mathrm{SSpec}(A_{\gamma})} \mathrm{SSpec}(A), \ \ V_{ijk}\simeq V'_{ijk}\times_{\mathrm{SSpec}(A_{\gamma})} \mathrm{SSpec}(A)\]
		for any triple $i, j, k$. Besides, $V_i'\cap V'_j=\cup_k V'_{ijk}$ for each couple of indices $i, j$.
		
		Note that if there is an $A$-superalgebra morphism $\psi : B\otimes_{A'} A\to C\otimes_{A'} A$, where $B$ and $C$ are finitely generated $A'$-superalgebras, then there is a finitely generated $A'$-supersubalgebra $A''$ of $A$, such that $\psi$ takes $B\otimes_{A'} A''$ to $C\otimes_{A'} A''$. Thus there is an index $\beta\geq\gamma$ such that $\phi^{\sharp}$ takes
		each $\mathcal{O}(U_i)\otimes A_{\beta}$ to $\mathcal{O}(V'_i)\otimes_{A_{\gamma}} A_{\beta}$ as well as each $\mathcal{O}(U_{ijk})\otimes A_{\beta}$ to $\mathcal{O}(V'_{ijk})\otimes_{A_{\gamma}} A_{\beta}$, so that all diagrams
		\[\begin{array}{ccc}
			\mathcal{O}(U_i)\otimes A_{\beta} & \to & \mathcal{O}(V'_i)\otimes_{A_{\gamma}} A_{\beta} \\
			\downarrow & & \downarrow \\
			\mathcal{O}(U_{ij k})\otimes A_{\beta} & \to & \mathcal{O}(V'_{ij k})\otimes_{A_{\gamma}} A_{\beta}
		\end{array} \]		
		are commutative. Gluing all together, one sees that there is an endomorphism $\phi'\in \mathfrak{End}(\mathbb{X})(A_{\beta})$ such that $\phi'\times_{\mathrm{SSpec}(A_{\beta})} \mathrm{id}_{\mathrm{SSpec}(A)}=\phi$. If $\psi$ is the inverse of $\phi$, then increasing $\beta$ one can construct the similar $\psi'\in \mathfrak{End}(\mathbb{X})(A_{\beta})$ such that
		$\psi'\times_{\mathrm{SSpec}(A_{\beta})} \mathrm{id}_{\mathrm{SSpec}(A)}=\psi$. We have
		\[ \phi'\psi'\times_{\mathrm{SSpec}(A_{\beta})} \mathrm{id}_{\mathrm{SSpec}(A)}=\phi\psi=\]
		\[\mathrm{id}_{X\times\mathrm{SSpec}(A)}=\psi\phi=\psi'\phi'\times_{\mathrm{SSpec}(A_{\beta})} \mathrm{id}_{\mathrm{SSpec}(A)}. \]
		By the above, for some $\alpha\geq\beta$ the morphisms $\phi'\times_{\mathrm{SSpec}(A_{\beta})} \mathrm{id}_{\mathrm{SSpec}(A_{\alpha})}$ and $\psi'\times_{\mathrm{SSpec}(A_{\beta})} \mathrm{id}_{\mathrm{SSpec}(A_{\alpha})}$ are mutually inverse to each other. Proposition is proved.   
	\end{proof}
	In terms of \cite{mats-oort} this proposition states that $\mathfrak{Aut}(\mathbb{X})$ satisfies the super analog of the condition $P_3$. In what follows we call the corresponding super analogs of all conditions $P_i$ from \cite{mats-oort} as \emph{super-$P_i$}.
	
	\section{Property super-$P_2$}
	
	Recall that $X_A$ denote $X\times\mathrm{SSpec}(A)$ for any superalgebra $A$. 
		\begin{lm}\label{open in a specific product}
		Let $X$ be a superscheme and $A$ be a finite dimensional local superalgebra. Then any open supersubscheme $V$ of $X_A\simeq X_{\Bbbk(A)}\times_{\mathrm{SSpec}(\Bbbk(A))}\mathrm{SSpec}(A)$ has a form $U\times_{\mathrm{SSpec}(\Bbbk(A))}\mathrm{SSpec}(A)$, where $U$ is an open supersubscheme of $X_{\Bbbk(A)}$.
	\end{lm}
	\begin{proof}
		Choose an open covering of $X$ by open affine supersubschemes $U_i$. Since $V=\cup_i (V\cap (U_i\times\mathrm{SSpec}(A)))$, one can assume that $X$ is affine. We have
		$V=\cup_g \mathrm{SSpec}((\mathcal{O}(X)\otimes A)_g)$ for a finite set of elements $g\in (\mathcal{O}(X)\otimes A)_0$. It remains to note that $g=\bar{g}\otimes 1 +y$, where
		$\bar{g}\in\mathcal{O}(X)\otimes\Bbbk(A), y\in\mathcal{O}(X)\otimes\mathfrak{M}_A$, and since $y$ is nilpotent, $(\mathcal{O}(X)\otimes A)_g\simeq (\mathcal{O}(X)\otimes\Bbbk(A))_{\bar{g}}\otimes A$ for each $g$. In other words, $V=(\cup_g \mathrm{SSpec}((\mathcal{O}(X)\otimes\Bbbk(A))_{\bar{g}}))\times\mathrm{SSpec}(A)$. 	
	\end{proof}
	
	\begin{pr}\label{P_2}
		If $\mathbb{X}$ is a proper superscheme, then $\mathfrak{Aut}(\mathbb{X})$ satisfies the super-$P_2$.	
	\end{pr}
	\begin{proof}
	Let $A$ be a c.l.N. superalgebra with the maximal superideal $\mathfrak{M}_A$. The superscheme $X_A$ is proper (and separated) over $S=\mathrm{SSpec}(A)$. 	Let $S'=\mathrm{SSpec}(A/\mathfrak{M}_A)$. By Proposition \ref{endomorphisms of completions}, the natural map
$\mathrm{End}_S(X_A)\to \mathrm{End}_{\widehat{S}}(\widehat{X_A})$ is an isomorphism of monoids. In particular, we have 
\[\mathfrak{Aut}(\mathbb{X})(A)=\mathrm{Aut}_S(X_A)\simeq \mathrm{Aut}_{\widehat{S}}(\widehat{X_A}),\]
and even more, we have a commutative diagram
\[\begin{array}{ccc}
\mathfrak{Aut}(\mathbb{X})(A) & \simeq & \mathrm{Aut}_{\widehat{S}}(\widehat{X_A}) \\
\searrow &  & \nearrow \\
 & \varprojlim_n \mathfrak{Aut}(\mathbb{X})(A/\mathfrak{M}_A^n) &  	
\end{array},\] 
where $q : \varprojlim_n \mathfrak{Aut}(\mathbb{X})(A/\mathfrak{M}_A^n)\to \mathrm{Aut}_{\widehat{S}}(\widehat{X_A})$ is the canonical group morphism, defined as follows.
Let $\varprojlim \phi_n\in \varprojlim \mathfrak{Aut}(\mathbb{X})(A/\mathfrak{M}^{n})$. Lemma \ref{open in a specific product} implies that  
$|X_{A/\mathfrak{M}^n}|=|X_{\Bbbk(A)}|$ and $|\phi_m|=|\phi_n|$ for any $m\geq n\geq 1$. Besides, there is a finite open covering of
$X_{\Bbbk(A)}$ by affine supersubschemes $U_{ijk}$ and $V_{ijk}$, such that :
\begin{enumerate}
	\item $U_i=\cup_{j, k} U_{ijk}$ and $V_i=\cup_{j, k} V_{ijk}$ are open affine for each index $i$;
	\item any $\phi_n^{\sharp}$ is uniquely defined by the collection of superalgebra isomorphisms \[\psi_{i, n} : \mathcal{O}(U_i)\otimes_{\Bbbk(A)} A/\mathfrak{M}^n\to \mathcal{O}(V_i)\otimes_{\Bbbk(A)} A/\mathfrak{M}^n\] and 
	\[\psi_{i, j, k, n} : \mathcal{O}(U_{ijk})\otimes_{\Bbbk(A)} A/\mathfrak{M}^n\to \mathcal{O}(V_{ijk})\otimes_{\Bbbk(A)} A/\mathfrak{M}^n,\] so that  
	all diagrams
	\[\begin{array}{ccc}
		\mathcal{O}(U_i)\otimes_{\Bbbk(A)} A/\mathfrak{M}^n & \stackrel{\psi_{i, n}}{\to} & \mathcal{O}(V_i)\otimes_{\Bbbk(A)} A/\mathfrak{M}^n\\
		\downarrow & & \downarrow  \\
		\mathcal{O}(U_{ijk})\otimes_{\Bbbk(A)} A/\mathfrak{M}^n & \stackrel{\psi_{i, j, k, n}}{\to} & \mathcal{O}(V_{ijk})\otimes_{\Bbbk(A)} A/\mathfrak{M}^n
	\end{array},\]
	\[\begin{array}{ccc}
		\mathcal{O}(U_i)\otimes_{\Bbbk(A)} A/\mathfrak{M}^m & \stackrel{\psi_{i, m}}{\to} & \mathcal{O}(V_i)\otimes_{\Bbbk(A)} A/\mathfrak{M}^m\\
		\downarrow & & \downarrow  \\ 
		\mathcal{O}(U_i)\otimes_{\Bbbk(A)} A/\mathfrak{M}^n & \stackrel{\psi_{i, n}}{\to} & \mathcal{O}(V_i)\otimes_{\Bbbk(A)} A/\mathfrak{M}^n\\
	\end{array}	,		 \]
	and  
	\[\begin{array}{ccc}
		\mathcal{O}(U_{ijk})\otimes_{\Bbbk(A)} A/\mathfrak{M}^m & \stackrel{\psi_{i, j, k, m}}{\to} & \mathcal{O}(V_{ijk})\otimes_{\Bbbk(A)} A/\mathfrak{M}^m\\
		\downarrow & & \downarrow  \\ 
		\mathcal{O}(U_{ijk})\otimes_{\Bbbk(A)} A/\mathfrak{M}^n & \stackrel{\psi_{i, j, k, n}}{\to} & \mathcal{O}(V_{ijk})\otimes_{\Bbbk(A)} A/\mathfrak{M}^n\\
	\end{array}		 \]
	are commutative. 
\end{enumerate} 
Note that each $|\phi_n|$ can be uniquely glued from the maps $|\mathrm{SSpec}(\psi_{i, n})|$, compatible being restricted on $(V_{ij k})_{A/\mathfrak{M}^n}$.
Determine $q(\varprojlim \phi_n)=\phi\in\mathrm{Aut}_{\widehat{S}}(\widehat{X_A})$ by the collection of isomorphisms $\phi^{\sharp}|_{\widehat{\mathcal{O}(U_i)\otimes A}}=\varprojlim \psi_{i, n}$ and $\phi^{\sharp}|_{\widehat{\mathcal{O}(U_{ijk})\otimes A}}=\varprojlim \psi_{i, j, k, n}$. Besides, $|\phi|=|\phi_1|$. Since $q$ is obviously injective, 
then $q$ is bijective, and our statement clearly follows.

	\end{proof}
	\begin{cor}\label{P_2 for even}
		If $\mathbb{X}$ is proper, then $\mathfrak{Aut}(\mathbb{X})_{ev}$ satisfies the condition $P_2$.	
	\end{cor}

	\section{More properties of the group functors $\mathfrak{Aut}(\mathbb{X}), \mathfrak{Aut}(\mathbb{X})_{ev}$ and $\mathfrak{T}$}	
	
	\begin{lm}\label{P_4 and P_5}
		The functor $\mathfrak{Aut}(\mathbb{X})$ satisfies the conditions	super-$P_4$ and super-$P_5$. In particular, the functor $\mathfrak{Aut}(\mathbb{X})_{ev}$ satisfies the conditions $P_4$ and $P_5$.	
	\end{lm}
	\begin{proof}
		Theorem \ref{category_locality} implies that $\mathfrak{Aut}(\mathbb{X})$ is a sheaf on 
		$\mathsf{SAlg}^{opp}_{\Bbbk}$, regarded as a site with respect to fpqc coverings, thus it satisfies super-$P_4$ and super-$P_5$ (use \cite[Corollary 1.1 and the discussion on p.721]{zub1}).
		In particular, $\mathfrak{Aut}(\mathbb{X})_{ev}$ is a sheaf on $\mathsf{Alg}_{\Bbbk}$ with respect to the same Grothendieck topology on $\mathsf{Alg}_{\Bbbk}^{opp}$, that is $\mathfrak{Aut}(\mathbb{X})_{ev}$
		satisfies the conditions $P_4$ and $P_5$. 
	\end{proof}	
	\begin{lm}\label{P_3}
		If $\mathbb{X}$ is algebraic, then the functor $\mathfrak{Aut}(\mathbb{X})_{ev}$ satisfies the condition $P_3$ from \cite{mats-oort}.	
	\end{lm}
	\begin{proof}
		Since $\mathfrak{Aut}(\mathbb{X})(A)=\mathfrak{Aut}(\mathbb{X})_{ev}(A)$ for any algebra $A$, Proposition \ref{Property P_3} concludes the proof. 
	\end{proof}
	\begin{lm}\label{pro-representability}
		If $\mathbb{X}$ is algebraic, then the functor $\mathfrak{T}$ satisfies the super-$P_3$. In particular, it is strictly pro-representable if and only if $\mathfrak{T}|_{\mathsf{SL}_{\Bbbk}}$ is.
	\end{lm}
	\begin{proof}	
		For any inductive system of $\Bbbk$-superalgebra $\{A_{\lambda} \}_{\lambda\in I}$, let $A$ denote $\varinjlim A_{\lambda}$. Then $\mathrm{nil}(A)=\varinjlim \mathrm{nil}(A_{\lambda})$
		and $\widetilde{A}\simeq \varinjlim \widetilde{A_{\lambda}}$. We have a commutative diagram 
		\[\begin{array}{ccccccccc}
			1 & \to & \varinjlim\mathfrak{T}(A_{\lambda}) & \to & \varinjlim\mathfrak{Aut}(\mathbb{X})(A_{\lambda}) & \to & \varinjlim\mathfrak{Aut}(\mathbb{X})(\widetilde{A_{\lambda}}) & \to & 1 \\  
			&  & \downarrow & & \downarrow  &  & \downarrow & \\
			1 & \to & \mathfrak{T}(A) &\to & \mathfrak{Aut}(\mathbb{X})(A) & \to &  \mathfrak{Aut}(\mathbb{X})(\widetilde{A}) & \to & 1
		\end{array}  \]
		By Proposition \ref{Property P_3} the middle and right vertical arrows are isomorphisms, hence the left one is an isomorphism as well.
		
		An alternative proof can be given by combining  Proposition \ref{nice property of T}(ii) with Proposition \ref{Property P_3}, and with the fact that $\mathsf{n}(A)=\varinjlim \mathsf{n}(A_{\lambda})$ .  The second statement also follows by Proposition \ref{nice property of T}(ii).	
	\end{proof}
	For any integer $n\geq 0$, let $\mathfrak{T}_n$ denote the subfunctor of $\mathfrak{T}$ defined as $\mathfrak{T}_n(A)=\varinjlim_{B\subseteq A, B\in\mathsf{SAlg}_{\Bbbk}(n)} \mathfrak{T}(B), A\in\mathsf{SAlg}_{\Bbbk}$. 
	\begin{lm}\label{describe T_n}
		We still assume that $\mathbb{X}$ is algebraic.	An element $\phi\in\mathfrak{T}(A)$ belongs to $\mathfrak{T}_n(A)$ if and only if the following conditions hold :
		\begin{enumerate}
			\item $|\phi|$ is the identity map of the topological space $|X\times\mathrm{SSpec}(A)|$;
			\item there is a supersubalgebra $B\subseteq A$ such that
			$B\in\mathsf{SL}_{\Bbbk}(n)$, and $\phi^{\sharp}$ induces an identity map of $\mathcal{O}(V)$ modulo $\mathcal{O}(V)\mathfrak{M}_B$ for arbitrary affine supersubscheme $V$ of $X\times\mathrm{SSpec}(A)$.
		\end{enumerate}
	\end{lm}
	\begin{proof}
		By the definition, an element $\phi\in\mathfrak{T}(A)$ belongs to $\mathfrak{T}_n(A)$ if and only if $|\phi|$ is an identity map of the topological space $|X\times\mathrm{SSpec}(A)|$. Moreover, there is a supersubalgebra $B\subseteq A$ such that
		$B\in\mathsf{SL}_{\Bbbk}(n)$, and $\phi^{\sharp}$ induces the identity map of $\mathcal{O}(U)\otimes B$ modulo $\mathcal{O}(U)\otimes\mathfrak{M}_B$ for each affine supersubscheme $U$ from a finite open covering of $X$.	In turn, for arbitrary open affine supersubscheme $V\subseteq X\times\mathrm{SSpec}(A)$ each open supersubscheme $V\cap U\times\mathrm{SSpec}(A)$ can be covered by open affine supersubschemes $\mathrm{SSpec}((\mathcal{O}(U)\otimes A)_g), g\in (\mathcal{O}(U)\otimes A)_0$. Moreover, $\phi^{\sharp}$ induces the identity map of $(\mathcal{O}(U)\otimes A)_g$ modulo
		$(\mathcal{O}(U)\otimes A)_g\mathfrak{M}_B$. 
		
		Using \cite[Lemma 3.5]{maszub1} one can cover each $\mathrm{SSpec}((\mathcal{O}(U)\otimes A)_g)$ by open supersubschemes $\mathrm{SSpec}(\mathcal{O}(V)_h), h\in \mathcal{O}(V)_0$, so that $\sum_h \mathcal{O}(V)_0 h=\mathcal{O}(V)_0$ and $\mathcal{O}(V)_h\simeq  ((\mathcal{O}(U)\otimes A)_g)_{\overline{h}}$, where $\overline{h}$ is the image of $h$ in the corresponding $(\mathcal{O}(U)\otimes A)_g$. Thus $\phi^{\sharp}$ induces the identity map of
		each $\mathcal{O}(V)_h$ modulo $\mathcal{O}(V)_h\mathfrak{M}_B$. Therefore, for any $t\in\mathcal{O}(V)$ there is a nonnegative integer $l$ such that
		$(t-\phi^{\sharp}(t))h^l\in \mathcal{O}(V)\mathfrak{M}_B$ for each $h$. Since $\sum_h a_h h^l=1$ for some $a_h\in \mathcal{O}(V)_0$, it follows that
		\[t-\phi^{\sharp}(t)=\sum_h (t-\phi^{\sharp}(t))h^l a_h\in \mathcal{O}(V)\mathfrak{M}_B. \]  
	\end{proof}
	The following corollary is now obvious (compare with \cite[Lemma 9.5]{maszub2}).  
	\begin{cor}\label{normality of T_n}
	For arbitrary integer $n\geq 1$ the subfunctor $\mathfrak{T}_n$ is invariant under the action of $\mathfrak{Aut}(\mathbb{X})$ by conjugations. 	
	\end{cor}
	If $X$ is a superscheme (not necessary algebraic), let $\mathfrak{RDer}(\mathcal{O}_X)$ denote the sheaf of \emph{right superderivations} of the superalgebra sheaf $\mathcal{O}_X$. In other words, for any open subsets $V\subseteq U\subseteq |X|$ we have
	\[\mathfrak{RDer}(\mathcal{O}_X)_i(U)=\{d\in \mathrm{End}_{\Bbbk}(\mathcal{O}_X)_i(U)\mid (fg)d=f (g)d+(-1)^{i |g|}(f)d g, f, g\in\mathcal{O}(V)\},\]
	where $i\in\mathbb{Z}_2$. Note that $\mathfrak{RDer}(\mathcal{O}_X)$ has the natural structure of a right $\mathcal{O}_X$-supermodule, hence it has the structure of left $\mathcal{O}_X$-supermodule as well. 
	
	Symmetrically, one can  define the sheaf of \emph{left superderivations} $\mathfrak{LDer}(\mathcal{O}_X)$. There is a canonical isomorphism $\mathfrak{RDer}(\mathcal{O}_X)\simeq \mathfrak{LDer}(\mathcal{O}_X)$ of $\mathcal{O}_X$-supermodules, say $d\mapsto d'$, where $d'(f)=(-1)^{|d||f|} (f)d, f\in\mathcal{O}_X$ (see Remark \ref{left is right}).
	
	Assume that $X$ is proper. Let $\delta : X\to X\times X$ be a diagonal closed immersion. Let $\mathcal{I}_X$ denote the superideal sheaf that defines the closed supersubscheme $\delta(X)$.
	The universality of $\mathcal{O}_X$-supermodule of Kahler superdifferentials $\Omega^1_{X/\Bbbk}\simeq \delta^*(\mathcal{I}_X/\mathcal{I}_X^2)$ (see \cite[Section 3]{maszub4} for more details) imply that  $\mathfrak{LDer}(\mathcal{O}_X)$ is isomorphic to the $\mathcal{O}_{X}$-supermodule sheaf $\mathcal{T}_X=\mathrm{Hom}_{\mathcal{O}_X} (\mathcal{I}_X/\mathcal{I}_X^2, \mathcal{O}_X)$, that is called the \emph{tangent} sheaf of $X$.
	\begin{lm}\label{Der}
		If $X$ is a proper superscheme, then $\mathfrak{LDer}(\mathcal{O}_X)(|X|)$ is finite dimensional.
	\end{lm}
	\begin{proof}
		The sheaf $\mathcal{I}_X/\mathcal{I}_X^2$ is coherent (see \cite[Section 3]{maszub4}). By Lemma \ref{dual of coherent} the sheaf $\mathcal{T}_X$ is coherent. Proposition \ref{global sections are finite} concludes the proof.
	\end{proof}
	Note that by Cohen's theorem any c.l.N. superalgebra $A$ contains a \emph{field of representatives}, that is isomorphic to $\Bbbk(A)$. In particular, it is valid for any finite dimensional local superalgebra.
	
	Let $\mathbb{F}$ denote the functor $\mathfrak{Aut}(\mathbb{X})_e : \mathsf{SL}_{\Bbbk}\to \mathsf{Gr}$.
	
	\begin{lm}\label{even P_1} 
		If $\mathbb{X}$ is a proper superscheme, then the functor $\mathbb{F}$ is strictly pro-representable.	In other words, $\mathfrak{Aut}(\mathbb{X})$ satisfies the condition super-$P_1$.
	\end{lm}
	\begin{proof}
		If $A\in\mathsf{SL}_{\Bbbk}$, then Lemma \ref{open in a specific product} infers $|X\times\mathrm{SSpec}(A)|=|X|$. Thus for any $\phi\in\mathbb{F}(A)$ there holds $|\phi|=\mathrm{id}_{|X\times\mathrm{SSpec}(A)|}$. 	
		By \cite[Lemma 9.6]{maszub2} we have $\mathcal{O}(U\times\mathrm{SSpec}(A))\simeq\mathcal{O}(U)\otimes A$ for any open supersubscheme $U$ of $X$. Therefore, $\mathbb{F}(A)$ can be naturally identified with the group of $A$-linear automorphisms $\phi$ of the superalgebra sheaf $\mathcal{O}_X\otimes A$, such that $(\mathrm{id}_{\mathcal{O}_X}\otimes\epsilon_A)\phi |_{\mathcal{O}_X}=
		\mathrm{id}_{\mathcal{O}_X}$. In partcular, $\phi\mapsto \phi|_{\mathcal{O}_X\otimes 1}$ is an injective  map of $\mathbb{F}(A)$ into the set of all superalgebra sheaf morphisms 
		$\psi : \mathcal{O}_X\to  \mathcal{O}_X\otimes A$ such that $(\mathrm{id}_{\mathcal{O}_X}\otimes\epsilon_A)\psi=\mathrm{id}_{\mathcal{O}_X}$. Moreover, this map is also surjective.
		In fact, just set $\phi=\psi\otimes\mathrm{id}_A$ and note that $\phi$ induces the identity morphism of the graded superalgebra sheaf $\oplus_{n\geq 0}\mathcal{O}_X\otimes \mathfrak{M}_A^n /\mathcal{O}_X\otimes \mathfrak{M}_A^{n+1}$.
		
		As in \cite[Lemma 3.4]{mats-oort}, it is easy to see that $\mathbb{F}$ is left exact, hence pro-representable. It remains to show that 
		$\mathbb{F}(\Bbbk[\epsilon_0, \epsilon_1])$ is a finite dimensional $\Bbbk$-vector space. Since any $\psi\in \mathbb{F}(\Bbbk[\epsilon_0, \epsilon_1])$ obviously has a form $\mathrm{id}_{\mathcal{O}_X}+ d_0\otimes\epsilon_0+d_1\otimes\epsilon_1$, where $d\in\mathfrak{RDer}(\mathcal{O}_X)_i(|X|), i=0, 1$, Lemma \ref{Der}
		concludes the proof.
	\end{proof}	
	\begin{cor}\label{P_1}
		If 	$\mathbb{X}$ is a proper superscheme, then $\mathfrak{Aut}(\mathbb{X})_{ev}$ satisfies the condition $P_1$.
	\end{cor}
	
	Recall that a group functor $\mathbb{G}$ satisfies the condition $P_{orb}$, if for every algebra $A$ over an algebraically closed field extension $L$ of $\Bbbk$, and any
	$g\in \mathbb{G}(A)\setminus 1$, there is an algebraic scheme $\mathbb{Y}$ (over $L$) on which $\mathbb{G}$ acts on the left, and $y\in\mathbb{Y}(L)$, such that $gy_A\neq y_A$, where
	$y_A$ is the image of $y$ in $\mathbb{Y}(A)$ (cf. \cite{mats-oort}). It is clear that if $\mathbb{H}\leq \mathbb{G}$ and $\mathbb{G}$ satisfies $P_{orb}$, then $\mathbb{H}$ does.
	
	Let $X$ be an algebraic superscheme.  Let $M$ and $\mathcal{T}$ denote $X_0$ and $(\mathcal{O}_X)_1$ respectively.
	We define the \emph{Grassman inflation} $\Lambda(X)$ of $X$ as 
	$(|M|, \Lambda_{\mathcal{O}_M}(\mathcal{T}))$. More precisely, there is a nonnegative integer $t$ such that $\mathcal{T}^{t+1}=0$ but $\mathcal{T}^t\neq 0$.
	Then $\Lambda_{\mathcal{O}_M}(\mathcal{T})$ is the sheafification of the superalgebra presheaf
	\[U\mapsto \oplus_{0\leq k\leq t}\wedge^k_{\mathcal{O}_M(U)}(\mathcal{T}(U)), U\subseteq |X|. \]
	
	If $X$ is algebraic or proper, then $\Lambda(X)$ is. Since $\mathrm{id}_{\mathcal{T}}$ is naturally extended to the surjective sheaf morphism $\Lambda_{\mathcal{O}_M}(\mathcal{T})\to\mathcal{O}_X$, $X$ is a closed supersubscheme of $\Lambda(X)$.   
	
	Furthermore, for any algebra $A\in\mathsf{Alg}_{\Bbbk}$ the superscheme $X\times\mathrm{SSpec}(A)$ is isomorphic to $(|M\times\mathrm{Spec}(A)|, \mathcal{O}_{M\times\mathrm{Spec}(A)}\oplus\mathcal{T}_A )$, where $\mathcal{T}_A=\mathrm{pr}_1^*\mathcal{T}$ and $\mathrm{pr}_1 : M\times\mathrm{Spec}(A)\to M$ is the canonical projection.
	
	Thus each $\phi\in\mathfrak{Aut}(\mathbb{X})(A)$ is uniquely extended to an automorphism $\Lambda(\phi)\in \mathfrak{Aut}(\Lambda(\mathbb{X}))(A)$ and the map $\phi\mapsto\Lambda(\phi)$ is a group monomorphism, functorial in $A$. More precisely, $|\Lambda(\phi)|=|\phi|$ and $\Lambda(\phi)^{\sharp}$ acts on
	$|\phi|_*\wedge_{\mathcal{O}_{M\times\mathrm{Spec}(A)}} (\mathcal{T}_A)$ as $\wedge\phi^{\sharp}$. 
	In particular, $\mathfrak{Aut}(\mathbb{X})_{ev}$ is a subgroup functor of $\mathfrak{Aut}(\Lambda(\mathbb{X}))_{ev}$.
	
	Consider a closed (and proper) supersubscheme $Z$ of $\Lambda(X)$, defined by the superideal sheaf  $\Lambda_{\mathcal{O}_M}(\mathcal{T})\mathcal{T}^2$. 
	As above, for arbitrary $A\in\mathsf{Alg}_{\Bbbk}$ the supersubscheme $Z\times\mathrm{SSpec}(A)$ is defined by the superideal sheaf $\Lambda_{\mathcal{O}_{M\times\mathrm{Spec}(A)}}(\mathcal{T}_A)\mathcal{T}_A^2$. 
	Since each automorphism
	$\phi\in \mathfrak{Aut}(\Lambda(\mathbb{X}))_{ev}(A)$ takes $|\phi|_*\mathcal{T}_A$ to $\Lambda_{\mathcal{O}_{M\times\mathrm{Spec}(A)}}(\mathcal{T}_A)\mathcal{T}_A$, it induces the automorphism $\overline{\phi}\in \mathfrak{Aut}(\mathbb{Z})_{ev}(A)$. Besides, $\phi\mapsto\overline{\phi}$ is a group homomorphism, functorial in $A$. Moreover, the induced  morphism $\mathfrak{Aut}(\mathbb{X})_{ev}\to \mathfrak{Aut}(\mathbb{Z})_{ev}$ of group functors is obviously injective (not necessary isomorphism!).
	
	\begin{theorem}\label{even is a group scheme}
		If $\mathbb{X}$ is proper, then $\mathfrak{Aut}(\mathbb{X})_{ev}$ is a locally algebraic group scheme.	
	\end{theorem}	
	\begin{proof}
		Since $\mathfrak{Aut}(\mathbb{X})_{ev}$ satisfies the conditions $P_1-P_5$, by \cite[Corollary 2.5]{mats-oort} all we need is to show that $\mathfrak{Aut}(\mathbb{X})_{ev}$ satisfies $P_{orb}$. By the above remarks, all we need is to show that $\mathfrak{Aut}(\mathbb{Z})_{ev}$ satisfies $P_{orb}$.
		
		The superscheme $Z$ can be regarded as an ordinary (proper) scheme. In fact, the sheaf $\mathcal{O}_Z\simeq \mathcal{O}_M\oplus\mathcal{T}$ is a sheaf of commutative algebras, since $\mathcal{T}^2=0$ in $\mathcal{O}_Z$. Note also that $\mathfrak{Aut}(\mathbb{Z})_{ev}$ is a group subfunctor of the automorphism group functor of $\mathbb{Z}$, regarded as such a scheme. By \cite[Theorem 3.7]{mats-oort} the latter group functor is isomorphic to a locally algebraic group scheme, hence by \cite[Corollary 2.5]{mats-oort} it satisfies $P_{orb}$. Theorem is proven. 
	\end{proof}		
	
	\section{The structure of $\mathfrak{Aut}(\mathbb{X})_{ev}$}
	
	We still assume that $X$ is a proper algebraic superscheme. Since any automorphism  of $X\times\mathrm{SSpec}(A)$ over $\mathrm{SSpec}(A)$ induces the automorphism of $X_{ev}\times\mathrm{Spec}(A)$ over $\mathrm{Spec}(A)$, provided $A$ is purely even, thus there is the natural morphism 
	$\mathfrak{Aut}(\mathbb{X})_{ev}\to\mathfrak{Aut}(\mathbb{X}_{ev})$ of group schemes. Set $\mathfrak{N}_t=\ker(\mathfrak{Aut}(\mathbb{X})_{ev}\to\mathfrak{Aut}(\mathbb{X}_{ev}))$ (the index $t$ refers to the fact that $\mathcal{T}^{t+1}=0$). As a closed group subfunctor of $\mathfrak{Aut}(\mathbb{X})_{ev}$, $\mathfrak{N}_t$ is a locally algebraic group scheme. 		
	
	Each group $\mathfrak{N}_t(A)$ can be identified with the group of $\mathcal{O}_{\mathrm{Spec}(A)}$-linear automorphisms of superalgebra sheaf $\mathcal{O}_{X\times\mathrm{SSpec}(A)}\simeq \mathcal{O}_{M\times\mathrm{Spec}(A)}\oplus \mathcal{T}_A$, those induce the identity morphisms modulo  $\mathcal{T}_A^2\oplus\mathcal{T}_A$ . 
	
	For any $1\leq l\leq t$, let $X^{\leq l}$ denote the superscheme $(|X|, \mathcal{O}_X/\mathcal{O}_X\mathcal{T}^{l+1})$. Set
	$\mathfrak{N}_l=\ker(\mathfrak{Aut}(\mathbb{X}^{\leq l})_{ev}\to\mathfrak{Aut}(\mathbb{X}_{ev}))$. Note that this is compatible with the definition of $\mathfrak{N}_t$. Moreover,
	the natural morphism $\mathfrak{Aut}(\mathbb{X})_{ev}\to\mathfrak{Aut}(\mathbb{X}^{\leq l})_{ev}$ of group schemes makes the diagram
	\[\begin{array}{ccc}
		\mathfrak{Aut}(\mathbb{X})_{ev} & \to & \mathfrak{Aut}(\mathbb{X}^{\leq l})_{ev}	\\
		\searrow & & \swarrow \\
		& \mathfrak{Aut}(\mathbb{X}_{ev}) &
	\end{array} \]
	to be commutative, hence it induces a group scheme morphism $\mathfrak{N}_t\to\mathfrak{N}_l$. Set $\mathfrak{N}_{t, l}=\ker(\mathfrak{N}_t\to\mathfrak{N}_l)$.
	\begin{lm}\label{central series}
		The series of normal locally algebraic group subschemes
		\[\mathfrak{N}_t\geq \mathfrak{N}_{t, 1}\geq \mathfrak{N}_{t, 2} \geq \ldots \geq \mathfrak{N}_{t, t}=1\]
		is central, starting from the second term. In other words, for any $1\leq l <t$ there is $[\mathfrak{N}_t, \mathfrak{N}_{t, l}]\leq \mathfrak{N}_{t, l+1}$. 
	\end{lm}
	\begin{proof}
		Consider $g\in\mathfrak{N}_t(A)$ and $h\in\mathfrak{N}_{t, l}(A)$. 
		If $l=2m+1$, then for any $a\in \mathcal{O}_{M\times \mathrm{Spec}(A)}$ and $f\in\mathcal{T}_A$ we have 
		\[h(a)\equiv a\pmod{\mathcal{T}_A^{l+1}}, \ h(f)\equiv f\pmod{\mathcal{T}_A^{l+2} }. \]	
		Thus $h(f)\equiv f\pmod{\mathcal{T}_A^{l+3}}$ for any $f\in\mathcal{T}_A^2$, that implies
		\[gh(a)\equiv g(a) \pmod{\mathcal{T}_A^{l+1} },\]
		and \[ hg(a)=h(a+(g(a)-a))\equiv a+(g(a)-a)=g(a)\pmod{\mathcal{T}_A^{l+1}  }. \]
		Similarly, we have
		\[gh(f)\equiv g(f)\pmod{\mathcal{T}_A^{l+2} }  \]	
		and
		\[hg(f)=h(f+(g(f)-f))\equiv g(f)\pmod{\mathcal{T}_A^{l+2} }. \]	
		We leave the case $l=2m$ to the reader. 
	\end{proof}
	Assume that $t>1$ and consider the central group subscheme $\mathfrak{N}_{t, t-1}$. If $t=2m+1$, then any $g\in\mathfrak{N}_{t, t-1}$ is characterized by the the conditions
	$g(a)=a, g(f)\equiv f\pmod{\mathcal{T}_A^t}, a\in \mathcal{O}_{M\times\mathrm{Spec}(A)}, f\in\mathcal{T}_A$. Since $g$ is an automorphism of superalgebra sheaf $\mathcal{O}_{X\times\mathrm{SSpec}(A)}$, $g-1$ is just
	$\mathcal{O}_{M\times\mathrm{Spec}(A)}$-linear sheaf morphism $\mathcal{T}_A\to\mathcal{T}_A^t$. In other words, the group $\mathfrak{N}_{t, t-1}(A)$ is isomorphic to
	the abelian group of global sections of the sheaf 
	\[\mathrm{Hom}_{\mathcal{O}_{M\times\mathrm{Spec}(A)}}(\mathcal{T}_A, \mathcal{T}_A^t)\simeq \mathrm{Hom}_{\mathcal{O}_{X_{ev}\times\mathrm{Spec}(A) }}((\mathcal{J}_X/\mathcal{J}_X^2)_A,  (\mathcal{J}_X^t)_A). \]
	
	Finally, if $t=2m$, then any $g\in\mathfrak{N}_{t, t-1}$ is characterized by the the conditions $g(a)\equiv a\pmod{\mathcal{T}_A^t}, g(f)=f$. Moreover, $g-1$ is an $\mathcal{O}_{\mathrm{Spec}(A)}$-linear derivation $\mathcal{O}_{M\times\mathrm{Spec}(A)}\to \mathcal{T}_A^t$. Since any such derivation vanishes on $\mathcal{T}_A^2$, it implies that 
	$\mathfrak{N}_{t, t-1}$ can be identified with the sheaf of $\mathcal{O}_{\mathrm{Spec}(A)}$-linear derivations $\mathcal{O}_{X_{ev}\times\mathrm{Spec}(A)}\to (\mathcal{J}^t_X)_A$. 
	Using \cite[Proposition II.8.10]{hart}, one sees that  $\mathfrak{N}_{t, t-1}(A)$ is isomorphic to the abelian group
	of global sections of the sheaf $\mathrm{Hom}_{\mathcal{O}_{X_{ev}\times\mathrm{Spec}(A)}}((\mathcal{I}_{X_{ev}}/\mathcal{I}_{X_{ev}}^2)_A, (\mathcal{J}_X^t)_A)$.
	
	Recall that $\mathbb{G}_a$ denote the one dimensional unipotent group, i.e. $\Bbbk[\mathbb{G}_a]\simeq \Bbbk[z]$, where $z$ is a primitive element of Hopf algebra $\Bbbk[z]$.
	\begin{lm}\label{N_{t, t-1}}
	If $t>1$, then the group scheme $\mathfrak{N}_{t, t-1}$ is isomorphic to the finite direct product of several copies of $\mathbb{G}_a$.	
	\end{lm}	
	\begin{proof}
		Since $\mathcal{J}_X/\mathcal{J}_X^2, \mathcal{J}_X^t$ and $\mathcal{I}_{X_{ev}}/\mathcal{I}_{X_{ev}}^2$ are coherent sheaves of $\mathcal{O}_{X_{ev}}$-(super)modules, Lemma \ref{represented by a space} implies the statement. 
	\end{proof}
	\begin{pr}\label{R is almost unipotent}
	If $t>1$, then $\mathfrak{N}^0_{t, 1}$ is an (affine) algebraic unipotent group.	
	\end{pr}		
	\begin{proof}
If $t=2$, then our stetment follows by Lemma \ref{N_{t, t-1}}. Assume that $t>2$. Again, by Lemma \ref{N_{t, t-1}} the group subscheme $\mathfrak{N}_{t, t-1}$ is connected, hence $\mathfrak{N}_{t, t-1}\leq \mathfrak{N}_{t, 1}^0$. Moreover, $\mathfrak{N}_{t, 1}^0$ is a clopen affine algebraic subgroup of  $\mathfrak{N}_{t, 1}$ (cf. \cite[Theorem 2.4.1]{mbrion}). 

It is clear that the group scheme morphism $\mathfrak{N}_t\to \mathfrak{N}_{t-1}$ takes $\mathfrak{N}_{t, 1}^0$ to $\mathfrak{N}_{t-1, 1}^0$. Then \cite[Proposition 2.7.1(1)]{mbrion} implies that the faisceau quotient $\mathfrak{N}_{t, 1}^0/\mathfrak{N}_{t, t-1}$ is a closed algebraic subgroup of $\mathfrak{N}^0_{t-1, 1}$. By the induction on $t$, $\mathfrak{N}^0_{t-1, 1}$ is affine and unipotent, hence 
$\mathfrak{N}_{t, 1}^0/\mathfrak{N}_{t, t-1}$ is. The latter infers the statement.   
		\end{proof}	
	\begin{pr}\label{t=1}
If $t=1$, then 	for any algebra $A$ the group $\mathfrak{N}_1(A)$ is isomorphic to \[\mathrm{End}_{\mathcal{O}_{M\times\mathrm{Spec}(A)}}(\mathcal{T}_A)(|M\times\mathrm{Spec}(A)|)^{\times}\simeq (\mathrm{End}_{\mathcal{O}_M}(\mathcal{T})(|M|)\otimes A)^{\times},\] 
that is $\mathfrak{N}_1$ is isomorphic to $\mathbb{GL}_{\mathcal{O}_M}(\mathcal{T})$. In particular, the group functor $\mathfrak{N}_t/\mathfrak{N}_{t, 1}$ is embedded into the locally algebraic group scheme $\mathbb{GL}_{\mathcal{O}_{X_{ev}}}(\mathcal{J}_X/\mathcal{J}_X^2)$.	
	\end{pr}
\begin{proof}
Arguing as above, one easily sees that $\mathfrak{N}_1(A)$ is isomorphic to the group of invertible $\mathcal{O}_{M\times\mathrm{Spec}(A)}$-linear sheaf endomorphisms	of $\mathcal{T}_A$. Lemma \ref{represented by a space} implies the first statement. The second statement is now obvious.
\end{proof}	
\begin{question}\label{unipotent entirely?}
Proposition \ref{R is almost unipotent} rises the natural question whether $\mathfrak{N}_t$ or $\mathfrak{N}_{t, 1}$ are algebraic for any $t\geq 1$?
\end{question}	
	
	\section{Strictly pro-representable group $\Bbbk$-functors}
	
	Let $\mathbb{X}$ be a strictly pro-representable $\Bbbk$-functor, and let  $A$ be a c.l.a.N. superalgebra, that represents $\mathbb{X}$. 
	The maximal superideal of $A$ is denoted by $\mathfrak{M}_A$. We also call $A$ the \emph{coordinate superalgebra} of $\mathbb{X}$.
	\begin{lm}\label{morphisms of strictly pro-representable functors}
		The category of strictly representable $\Bbbk$-functors is anti-equivalent to the category of c.l.a.N. superalgebras (with continuous local morphisms between them).
	\end{lm} 
	\begin{proof}
		We outline the proof and leave details for the reader. 
		
		Let $\mathbb{X}$ and $\mathbb{Y}$ be $\Bbbk$-functors, represented by superalgebras $A$ and $B$ respectively.
		Let ${\bf f} : \mathbb{X}\to\mathbb{Y}$ be a morphism of $\Bbbk$-functors.
		
		For any couple of nonnegative integers $n< t$ and any superalgebra $C\in\mathsf{SL}_{\Bbbk}(n)$, there is \[\mathbb{X}(C)=\mathrm{Hom}_{\mathsf{SL}_{\Bbbk}}(A/\mathfrak{M}_A^{n+1} , C)=\mathrm{Hom}_{\mathsf{SL}_{\Bbbk}}(A/\mathfrak{M}_A^{t+1} , C),\]
		and the similar statement holds for $\mathbb{Y}$. In particular, we have a commutative diagram
		\[\begin{array}{ccc}
			\mathrm{Hom}_{\mathsf{SL}_{\Bbbk}}(A/\mathfrak{M}_A^{t+1} , A/\mathfrak{M}_A^{t+1}) & \to & \mathrm{Hom}_{\mathsf{SL}_{\Bbbk}}(B/\mathfrak{M}_B^{t+1} , A/\mathfrak{M}_A^{t+1})\\
			\downarrow & & \downarrow \\
			\mathrm{Hom}_{\mathsf{SL}_{\Bbbk}}(A/\mathfrak{M}_A^{n+1} , A/\mathfrak{M}_A^{n+1}) & \to & \mathrm{Hom}_{\mathsf{SL}_{\Bbbk}}(B/\mathfrak{M}_B^{n+1} , A/\mathfrak{M}_A^{n+1})
		\end{array}\]
		Let $\phi_n$ denote the image of $\mathrm{id}_{A/\mathfrak{M}_A^{n+1}}$ in $\mathrm{Hom}_{\mathsf{SL}_{\Bbbk}}(B/\mathfrak{M}_B^{n+1} , A/\mathfrak{M}_A^{n+1}), n\geq 1$.
		Then the collection $\{\phi_n\}_{n\geq 1}$ defines the dual morphism $\phi : B\to A$.  
	\end{proof}
	This lemma allows us to define a Zariski topology on arbitrary strictly pro-representable functor $\mathbb{X}$, so that a closed subfunctor $\mathbb{Y}\subseteq\mathbb{X}$ is again strictly pro-representable by a c.l.a.N. superalgebra $A/I$, where $I$ is a closed superideal of $A$. For example, $\mathbb{X}_{ev}$ is defined by the superideal $J_A$.
	\begin{lm}\label{Hopf}
		If $\mathbb{X}$ is strictly pro-representable, then it is a group functor if and only if its coordinate superalgebra $A$ is a c.l.a.N.
		Hopf superalgebra. Thus the category of strictly pro-representable group $\Bbbk$-functors is anti-equivalent to the category of c.l.a.N. Hopf superalgebras (with continuous local Hopf superalgebra morphisms between them).   	
	\end{lm}
	\begin{proof}
		Observe that if $\mathbb{X}$ is represented by $A$, then $\mathbb{X}\times\mathbb{X}$ is represented by $\widehat{A\otimes  A}$, the completion of $A\otimes  A$ in the $(A\otimes  \mathfrak{M}_A + \mathfrak{M}_A\otimes  A)$-adic topology. The second statement is obvious.
	\end{proof}
	From now on $\mathbb{X}$ is a strictly pro-representable group $\Bbbk$-functor. Let $\Delta=\Delta_A$ and $S=S_A$ denote the coproduct and the antipode of its coordinate superalgebra $A$. We use the \emph{Sweedler}'s notations, i.e. $\Delta(a)=\sum a_{(0)}\otimes a_{(1)}$ is a convergent series
	in the $(A\otimes\mathfrak{M}_A+\mathfrak{M}_A\otimes A)$-adic topology.
	
	Following \cite[Section 9]{maszub2}, one can construct a Hopf superalgebra $\mathrm{hyp}(\mathbb{X})=\cup_{n\geq 1}\mathrm{hyp}_n(\mathbb{X})$, where 
	$\mathrm{hyp}_n(\mathbb{X})=(A/\mathfrak{M}^{n+1}_A)^*$. In other words, $\mathrm{hyp}(\mathbb{X})$ coincides with the vector space $A^{\underline{*}}$ consisting of all
	continuous linear maps $A\to\Bbbk$, where $\Bbbk$ is regarded as a discrete vector space. The product and coproduct of $\mathrm{hyp}(\mathbb{X})$ are defined as
	\[\phi\psi(a)=\sum (-1)^{|\psi||a_{(0)}|}\phi(a_{(0)})\psi(a_{(1)}),  \]
	and $\Delta^*(\phi)=\sum \phi_{(0)}\otimes\phi_{(1)}$ if
	\[\phi(ab)=\sum (-1)^{|\phi_{(1)}||a|}\phi_{(0)}(a)\phi_{(1)}(b), a, b\in A, \phi, \psi\in\mathrm{hyp}(\mathbb{X}). \] 
	Besides, $S^*(\phi)(a)=\phi(S(a))$. The following lemma is a copy of \cite[Lemma 9.2]{maszub2}.
	\begin{lm}\label{Lemma 9.2}
		For any nonnegative integers $k$ and $t$ we have :
		\begin{enumerate}
			\item $\mathrm{hyp}_k(\mathbb{X})\mathrm{hyp}_t(\mathbb{X})\subseteq \mathrm{hyp}_{k+t}(\mathbb{X})$;
			\item $\Delta^*(\mathrm{hyp}_k(\mathbb{X}))\subseteq \sum_{0\leq s\leq k} \mathrm{hyp}_s(\mathbb{X})\otimes \mathrm{hyp}_{k-s}(\mathbb{X})$;
			\item $S^*(\mathrm{hyp}_k(\mathbb{X}))\subseteq \mathrm{hyp}_k(\mathbb{X})$.	
		\end{enumerate}		
	\end{lm} 
	Let $R$ be a superalgebra. The $R$-superalgebras $A\otimes R$ and $\mathrm{hyp}(\mathbb{X})\otimes R$ are Hopf $R$-superalgebras, i.e. their products, coproducts, counits and antipodes are just $R$-linear extensions of products, coproducts, counits and antipodes of $A$ and $\mathrm{hyp}(\mathbb{X})$ respectively. Moreover, we have the pairing 
	\[(\mathrm{hyp}(\mathbb{X})\otimes R)\times (A\otimes R)\to R, \ <\phi\otimes r, a\otimes r'>=(-1)^{|r||a|}\phi(a)rr',\]
	that is the Hopf superalgebra pairing in the sense of \cite{hmt}. The proof of the following lemma can be copied from \cite[Lemma 9.3]{maszub2}
	\begin{lm}\label{another def of X}
		The group functor $\mathbb{X}$ is isomorphic to $R\mapsto \mathsf{Gpl}(\mathrm{hyp}(\mathbb{X})\otimes R)$, where
		\[\mathsf{Gpl}(\mathrm{hyp}(\mathbb{X})\otimes R)=\{ u\in (\mathrm{hyp}(\mathbb{X})\otimes R)^{\times} \mid  (\Delta^*\otimes\mathrm{id}_R)(u)=u\otimes_R u\}.\]
	\end{lm}
	The Hopf superalgebra $\mathrm{hyp}(\mathbb{X})$ has another filtration with $\mathrm{hyp}^{(k)}(\mathbb{X})=\{\phi\in\mathrm{hyp}(\mathbb{X})\mid \phi(J_A^{k+1})=0 \}, k\geq 0$.
	Since $A_1$ is a finitely generated $A_0$-module, this filtration is finite. The proof of the following lemma is similar to the proof of Lemma \ref{Lemma 9.2}.
	\begin{lm}\label{Lemma 11.4}
		For any nonnegative interegrs $k, l$ there is :
		\begin{enumerate}
			\item  $\mathrm{hyp}^{(k)}(\mathbb{X})\mathrm{hyp}^{(l)}(\mathbb{X})\subseteq \mathrm{hyp}^{(k+l)}(\mathbb{X})$;
			\item $\Delta^*(\mathrm{hyp}^{(k)}(\mathbb{X}))\subseteq \sum_{0\leq s\leq k} \mathrm{hyp}^{(s)}(\mathbb{X})\otimes \mathrm{hyp}^{(k-s)}(\mathbb{X})$;
			\item $S^*(\mathrm{hyp}^{(k)}(\mathbb{X}))\subseteq \mathrm{hyp}^{(k)}(\mathbb{X})$.
		\end{enumerate}
	\end{lm}
	Lemma \ref{Lemma 11.4} implies that $\mathsf{gr}(\mathrm{hyp}(\mathbb{X}))=\oplus_{k\geq 0}\mathrm{hyp}^{(k)}(\mathbb{X})/\mathrm{hyp}^{(k+1)}(\mathbb{X})$ is a \emph{graded} Hopf superalgebra
	(cf. \cite{maszub2}).
	
	Next, $\mathsf{gr}(A)$ is a c.l.N. Hopf superalgebra with the maximal superideal 
	\[\mathfrak{M}_{\mathsf{gr}(A)}=\overline{\mathfrak{m}_A}\oplus (\oplus_{k\geq 1} J_A^k/J_A^{k+1}),\] 
	where $\overline{\mathfrak{m}_A}=\mathfrak{m}_A/A_1^2$. Recall that $\mathfrak{M}_{\mathsf{gr}(A)}$-adic topology of $\mathsf{gr}(A)$ coincides with the $\overline{\mathfrak{m}_A}$-adic one (cf. \cite[Section 1.6]{maszub4}). 
	
	The comultiplication of $\mathsf{gr}(A)$ is defined as follows. For any $x\in J_A^k$ there is $\Delta(x)=\sum_{0\leq s\leq k} x_{(0)}^{(s)}\otimes x_{(1)}^{(k-s)}$, where $x_{(0)}^{(s)}, x_{(1)}^{(s)}\in J_A^s, 0\leq s\leq k$.   
	Then 
	\[\mathsf{gr}(\Delta)(x+J_A^{k+1})=\sum_{0\leq s\leq k} (x_{(0)}^{(s)}+J_A^{s+1})\otimes (x_{(1)}^{(k-s)}+J_A^{k+1-s}),\]
	so that $\mathsf{gr}(A)$ is a graded Hopf superalgebra as well.  The antipode of $\mathsf{gr}(A)$ is defined as $\mathsf{gr}(S)(x+J_A^{k+1})=S(x)+J_A^{k+1}$. 
	
	The superalgebra $\mathsf{gr}(A)$ represents a strictly pro-representable group $\Bbbk$-functor, that is denoted by $\mathsf{gr}(\mathbb{X})$. Moreover, $\mathbb{X}\to\mathsf{gr}(\mathbb{X})$ is an endofunctor of the category of strictly pro-representable group $\Bbbk$-functors, such that ${\bf f} : \mathbb{X}\to\mathbb{Y}$ is an isomorphism if and only if $\mathsf{gr}({\bf f})$ is.   
	
	Recall that the c.l.a.N. Hopf algebra $\overline{A}$ represents the group subfunctor $\mathbb{X}_{ev}$. Moreover, we have (local) Hopf superalgebra morphisms
	$q : \overline{A}\to \mathsf{gr}(A)$ and $i : \mathsf{gr}(A)\to \overline{A}$ such that $iq=\mathrm{id}_{\overline{A}}$. Dualizing, the embedding $\mathbb{X}_{ev}\to\mathsf{gr}(\mathbb{X})$ is the right inverse of the morphism ${\bf q} : \mathsf{gr}(\mathbb{X})\to\mathbb{X}_{ev}$, induced by $q$. Therefore, $\mathsf{gr}(\mathbb{X})$ is isomorphic to the semi-direct product $\mathbb{X}_{ev}\ltimes\ker{\bf q}$.  
	\begin{lm}\label{Prop 11.1}
		The group $\Bbbk$-functor $\ker{\bf q}$ is isomorphic to the purely odd unipotent group superscheme, represented by a (finite dimensional) local Hopf superalgebra $\Lambda(V)$, where
		$V=(\mathfrak{M}_A/\mathfrak{M}_A^2)_1\simeq A_1/\mathfrak{m}_A A_1$ and all elements of $V$ are primitive.  		
	\end{lm}	
	\begin{proof}
		We have $\ker{\bf q}(C)=\{\phi\in \mathsf{gr}(\mathbb{X})(C)\mid \phi|_{\overline{A}}=\epsilon_A \}$ for arbitrary superalgebra $C$.	Thus $\ker{\bf q}$ is represented by
		the finite dimensional local Hopf superalgebra $B=\mathsf{gr}(A)/\mathsf{gr}(A)\overline{\mathfrak{m}_A}$. The Hopf superalgebra $B$ is generated by its (purely odd) degree one component $(A_1/A_1^3)/(\mathfrak{m}_A A_1/A_1^3)\simeq A_1/\mathfrak{m}_A A_1$. Since $B$ is also a graded Hopf superalgebra, the elements of $V$ are primitive. The final arguments from \cite[Proposition 11.1]{maszub2} conclude the proof. 
	\end{proof}
	\begin{rem}\label{odd unipotent}
		In the notations from \cite{maszub3}, $\ker{\bf q}$ is isomorphic to the direct product of $\dim V$ copies of the purely odd unipotent group $\mathbb{G}^-_a$ (of superdimension $0|1$).
		Recall that $\Bbbk[\mathbb{G}^-_a]\simeq\Bbbk[t]$, where $t$ is an odd primitive element.
	\end{rem}
	Combining Lemma \ref{Prop 11.1} with Lemma \ref{Hopf}, one sees that $\mathsf{gr}(A)\simeq\overline{A}\otimes \Lambda(V)$. Equivalently, there is a minimal generating set $v_1+J_A^2, \ldots , v_l+J_A^2$ of $\overline{A}$-module $J_A/J_A^2$, such that any element $f\in A$ has a "canonical" form
	\[ f=\sum_{0\leq s\leq l, 1\leq i_1<\ldots < i_s\leq l} f_{i_1, \ldots , i_s} v_{i_1}\ldots v_{i_s},\]
	where the even coefficients  $f_{i_1, \ldots , i_s}$ are uniquely defined by $f$, modulo $A_1^2$. With this remark in mind, the proof of the following lemma can be copied from
	\cite[Proposition 11.5]{maszub2}. We outline the principal steps only. 
	\begin{lm}\label{Prop 11.5} 
		There is a natural isomorphism $\mathsf{gr}(\mathrm{hyp}(\mathbb{X}))\simeq\mathrm{hyp}(\mathsf{gr}(\mathbb{X}))$ of graded Hopf superalgebras.	
	\end{lm}
	\begin{proof}
		As it has been already observed, one can work in the $\overline{\mathfrak{m}_A}$-adic topology. 
		In particular, $\mathrm{hyp}(\mathsf{gr}(\mathbb{X}))$ is isomorphic to $\oplus_{0\leq k\leq l} (J_A^k/J_A^{k+1})^{\underline{*}}$ as a superspace.
		
		The above Hopf superalgebra pairing  $< , >$ induces an embedding of $\Bbbk$-vector (super)spaces
		\[\mathrm{hyp}^{(k)}(\mathbb{X})/\mathrm{hyp}^{(k+1)}(\mathbb{X})\to (J_A^k/J_A^{k+1})^{\underline{*}}.\]
		For each $t\geq 0$ we choose a basis $B_t$ of $\overline{\mathfrak{m}_A}^t/\overline{\mathfrak{m}_A}^{t+1}$. Set $B=\cup_{t\geq 0} B_t$.
		Using a "canonical" basis of $\mathsf{gr}(A)$ consisting of the elements $\overline{f} v_{i_1}\ldots v_{is}$, where $\overline{f}\in B$ and $0\leq s\leq l$,
		one can show that the above emebedding is an isomorphism. In particular, we obtain a pairing
		\[\mathsf{gr}(\mathrm{hyp}(\mathbb{X}))\times\mathsf{gr}(A)\to\Bbbk,\]
		that is a Hopf superalalgebra pairing due the definition of Hopf superalgebra structure of $\mathsf{gr}(A)$.
	\end{proof}
	The definition of $\Delta^*$ implies that a (homogeneous) element $\phi\in\mathrm{hyp}(\mathbb{X})$ is primitive if and only if 
	$\phi$ is a $\Bbbk$-superderivation of $A$, i.e. 
	\[\phi(ab)=\phi(a)\epsilon_A(b) +(-1)^{|\phi||a|}\epsilon_A(a)\phi(b), a, b\in A.\]
	Thus the superspace of primitive elements of $\mathrm{hyp}(\mathbb{X})$ coincides with $\mathrm{hyp}_1(\mathbb{X})^+=\{\phi\in \mathrm{hyp}_1(\mathbb{X}) \mid \phi(1)=0\}$.
	In partcular, $\mathrm{hyp}_1(\mathbb{X})^+$ is a Lie superalgebra with respect to the standard Lie superbracket $[\phi, \psi]=\phi\psi -(-1)^{|\phi||\psi|}\psi\phi, \phi, \psi\in \mathrm{hyp}_1(\mathbb{X})^+$. We call it a \emph{Lie superalgebra} of $\mathbb{X}$ and denote by $\mathrm{Lie}(\mathbb{X})$.
	Choose a basis $\gamma_1, \ldots, \gamma_l$ of $\mathrm{Lie}(\mathbb{X})_1\simeq (A_1/\mathfrak{m}_A A_1)^*$. 
	Note also that $\mathrm{hyp}^{(0)}(\mathbb{X})\simeq \mathrm{hyp}(\mathbb{X}_{ev})$.
	\begin{lm}\label{Prop 11.6}
		Every element $\phi\in\mathrm{hyp}_k(\mathbb{X})$ has the unique form
		\[\phi=\sum_{0\leq s\leq k, 1\leq i_1<\ldots <i_s}\phi_{i_1, \ldots , i_s}\gamma_{i_1}\ldots \gamma_{i_s}, \phi_{i_1, \ldots , i_s}\in\mathrm{hyp}_{k-s}(\mathbb{X}_{ev}). \]	
	\end{lm}
	\begin{proof}
		Use Lemma \ref{Prop 11.5} and copy the proof of \cite[Proposition 11.6]{maszub2}.	
	\end{proof}
	Following \cite{maszub2}, we define the group-like elements 
	\[f(b, x)=\epsilon_A\otimes 1+x\otimes b, \ e(a, v)=\epsilon_A\otimes 1+v\otimes a,\] 
	where $x\in \mathrm{Lie}(\mathbb{X})_0, v\in \mathrm{Lie}(\mathbb{X})_1, b\in A_0, b^2=0, a\in A_1$. Mimic the proof of \cite[Lemma 4.2]{masshib} we obtain the following relations :
	\begin{enumerate}
		\item $[e(a, v), e(a', v')]=f(-aa', [v, v'])$;
		\item $[f(b, x), e(a, v)]=e(ba, [x, v])$;
		\item $[f(b, x), f(b', x')]=f(bb, [x, x'])$;
		\item $e(a, v)e(a', v)=f(-aa', \frac{1}{2}[v, v])e(a+a', v)$ .
	\end{enumerate}
	Further, let $\Sigma_{\mathbb{X}}$ denote the group subfunctor of $\mathbb{X}$ such that $\Sigma_{\mathbb{X}}(R)$ is generated by all elements $f(b, x)$ and $e(a, v), a, b\in R$, for arbitrary superalgebra $R$.
	\begin{lm}\label{Lemma 12.2}
		We have the following :
		\begin{enumerate}
			\item if $\Sigma_{\mathbb{X}, ev}$ denotes $\Sigma_{\mathbb{X}}\cap\mathbb{X}_{ev}$, then for any superalgebra $R$ the group $\Sigma_{\mathbb{X}, ev}(R)=\Sigma_{\mathbb{X}}(R_0)$ is generated by the elements $f(b, x), b\in R$;
			\item each element of $\Sigma_{\mathbb{X}}(R)$ is uniquely expressed in the form
			\[f e(a_1, \gamma_1)\ldots e(a_l, \gamma_l), f\in\Sigma_{\mathbb{X}, ev}(R), a_1, \ldots, a_l\in R_1.\]
		\end{enumerate}	
	\end{lm} 
	\begin{proof}
		Use Lemma \ref{Prop 11.6} and copy the proof of \cite[Proposition 4.3]{masshib}. 	
	\end{proof}
	\begin{lm}\label{action of X}
		The group functor $\mathbb{X}$ acts linearly on $\mathrm{Lie}(\mathbb{X})$ by conjugations. 	
	\end{lm}
	\begin{proof}
		For any $\phi\in \mathrm{Lie}(\mathbb{X})$ and arbitrary $x\in \mathsf{Gpl}(\mathrm{hyp}(\mathbb{X})\otimes R)$, there is
		\[\Delta^*(x(\phi\otimes 1)x^{-1})=\]
		\[(x\otimes_R x)((\phi\otimes 1)\otimes_R (\epsilon_A\otimes 1)+((\epsilon_A\otimes 1)\otimes_R (\phi\otimes 1))(x^{-1}\otimes_R x^{-1})=\]
		\[(x(\phi\otimes 1)x^{-1})\otimes _R(\epsilon_A\otimes 1)+(\epsilon_A\otimes 1)\otimes_R (x(\phi\otimes 1)x^{-1}).\]
		The lineariry is obvious by the definition of the product in $\mathrm{hyp}(\mathbb{X})\otimes R$. 
	\end{proof}	
	Define the subfunctor ${\bf E}_{\mathbb{X}}$ of $\mathbb{X}$ such that 
	\[{\bf E}_{\mathbb{X}}(R)=\{e(a_1, \gamma_1)\ldots e(a_l, \gamma_l)\mid a_1, \ldots , a_l\in R_1 \}, R\in\mathsf{SAlg}_{\Bbbk}.\]
	Note that any morphism ${\bf g} : \mathbb{X}\to\mathbb{Y}$ of strictly pro-representable group $\Bbbk$-functors induces a morphism of filtred Hopf superalgebras
	$\mathrm{hyp}({\bf g}) : \mathrm{hyp}(\mathbb{X})\to \mathrm{hyp}(\mathbb{X})$, hence a Lie superalgebra morphism $\mathrm{d}{\bf g} : \mathrm{Lie}(\mathbb{X})\to \mathrm{Lie}(\mathbb{Y})$. The latter is dual to the induced morphism $\mathfrak{M}_B/\mathfrak{M}_B^2\to \mathfrak{M}_A/\mathfrak{M}_A^2$ of $\Bbbk$-superspaces.
	
	Furthermore, ${\bf g}$ induces the morphism of group functors $\Sigma_{\mathbb{X}}\to \Sigma_{\mathbb{Y}}$, defined on the generators by  
	\[f(b, x)\mapsto f(b, \mathrm{d}{\bf g}(x)), \ e(a, v)\mapsto e(a, \mathrm{d}{\bf g}(v)). \]
	If we choose the basis of $\mathrm{Lie}(\mathbb{X})$ so that some of its vectors form a basis of the kernel of $\mathrm{d}{\bf g}$, then ${\bf g}$ induces morphism  ${\bf E}_{\mathbb{X}}\to {\bf E}_{\mathbb{Y}}$ of $\Bbbk$-functors.
	
	Lemma \ref{action of X} implies that $\mathbb{Y}=\mathbb{X}_{ev}{\bf E}_{\mathbb{X}}$ is a group subfunctor of $\mathbb{X}$. Moreover, Lemma \ref{Lemma 12.2} infers that each element of $\mathbb{Y}$ can be uniquely expressed as 
	$x=x' e, x'\in\mathbb{X}_{ev}, e\in {\bf E}_{\mathbb{X}}$, i.e.  $\mathbb{Y}\simeq \mathbb{X}_{ev}\times {\bf E}_{\mathbb{X}}$.
	\begin{lm}\label{equality}
		We have $\mathbb{X}=\mathbb{Y}$.	
	\end{lm}
	\begin{proof}
		The functor $\mathbb{Y}$ is strictly pro-representable as a direct product of two strictly pro-representable functors. Its coordinate superalgebra $B$ is isomorphic to
		$\overline{A}\otimes\Lambda(V)$. The composition of two embeddings $\mathbb{X}_{ev}\to \mathbb{Y}$ and $\mathbb{Y}\to\mathbb{X}$ coincides with the canonical embedding $\mathbb{X}_{ev}\to\mathbb{X}$, hence the dual morphism 
		$A\to B$ induces the isomorphism $\overline{A}\to \overline{B}$. 
		
		Similarly, since ${\bf E}_{\mathbb{Y}}={\bf E}_{\mathbb{X}}$, the above remarks and Lemma \ref{Prop 11.6} imply that $\mathrm{Lie}(\mathbb{Y})\to \mathrm{Lie}(\mathbb{X})$ is an isomorphism of Lie superalgebras. Thus the induced morphism $\mathsf{gr}(A)\to\mathsf{gr}(B)$ is an isomorphism of graded superalgebras, hence $A\to B$ is. Lemma is proved.
	\end{proof}
	
	\section{Matsumura-Oort criterion and representability of $\mathfrak{Aut}(\mathbb{X})$}	
	
	\begin{lm}\label{Lemma 1.2}
	Let ${\bf f} : \mathbb{G}\to\mathbb{H}$ be a morphism of $\Bbbk$-functors. Assume that both $\mathbb{G}$ and $\mathbb{H}$ satisfy the conditions super-$P_i, i=2, 3, 4, 5$. Then 	
	${\bf f}$ is an isomorphism if and only if for any finitely dimensional local $\Bbbk$-algebra $A$ the map ${\bf f}(A)$ is bijective.
\end{lm}
\begin{proof}
	The proof of \cite[Lemma 1.2]{mats-oort} can be copied per verbatim. However, some parts of this proof require additional comments in the wider super-context. First, by \cite[Lemma 6.1(1)]{kolzub} if $A$ is a local Noetherian superalgebra, then its completion $\widehat{A}$ in the $\mathfrak{M}_A$-adic topology is a faithfully flat $A$-supermodule. 
	Second, if ${\bf f}(B)$ is injective for any local Noetherian superalgebra $B$, such that the extension $\Bbbk\subseteq\Bbbk(B)$ is finite, then ${\bf f}(A)$ is injective for any finitely generated superalgebra $A$. In fact, let $\mathrm{Max}(A)$ denote the set of all maximal superideals of $A$, and let $B$ denote the superalgebra \[\times_{\mathfrak{M}\in\mathrm{Max}(A)} A_{\mathfrak{M}}=\varinjlim_{S\subseteq\mathrm{Max}(A), |S|<\infty}\times_{\mathfrak{M}\in S} A_{\mathfrak{M}}.\]   
	The condition super-$P_3$ implies that ${\bf f}(B)$ is injective. On the other hand, by \cite[Lemma 1.4]{zub1} the superalgebra $B$ is a faithfully flat $A$-supermodule, hence 
	the induced maps $\mathbb{G}(A)\to\mathbb{G}(B)$ and $\mathbb{H}(A)\to\mathbb{H}(B)$ are injective, and in its turn, ${\bf f}(A)$ is. These two remarks are sufficient to reproduce the proof of \cite[Lemma 1.2]{mats-oort}  in the super-context.	
\end{proof}
Let $\mathbb{G}$ be a group $\Bbbk$-functor. We determine $\mathbb{G}_{ev}$ as a group subfunctor of $\mathbb{G}$, such that $\mathbb{G}_{ev}(A)=\mathbb{G}(\iota_A)(\mathbb{G}(A_0))$, where $\iota_A : A_0\to A$ is the natural embedding. Unlike group morphism $\mathfrak{Aut}(\mathbb{X})(\iota_A)$, the group morphism $\mathbb{G}(\iota_A)$ is not necessary injective.

The following proposition gives a criterion of representability of $\mathbb{G}$ that is partially similar to \cite[Theorem 1.4]{mats-oort}. 
\begin{pr}\label{super criterion}(Matsumura-Oort criterion)
The group functor $\mathbb{G}$ is isomorphic to a locally algebraic group superscheme if and only if the following conditions hold :
\begin{enumerate}
	\item $\mathbb{G}_{ev}$ is isomorphic to a locally algebraic group scheme;
	\item $\mathbb{G}$ satisfies the conditions super-$P_i, $, $1\leq i\leq 5$.  
\end{enumerate}	
\end{pr}
\begin{proof}
Only the part "if" needs the proof. As above, we define a normal group subfunctor $\mathfrak{T}$ of $\mathbb{G}$ by $\mathfrak{T}(A)=\ker(\mathbb{G}(\pi_A)), A\in\mathsf{SAlg}_{\Bbbk}$.
The conditions super-$P_1$ and super-$P_3$ imply that $\mathfrak{T}$, or equivalently, $\mathfrak{T}|_{\mathsf{SL}_{\Bbbk} }=\mathbb{G}_e$, are strictly pro-representable. In turn, Lemma \ref{equality} implies that the group subfunctor $\mathbb{G}_{ev}\mathfrak{T}$ is isomorphic to  $\mathbb{G}_{ev}\times {\bf E}_{\mathfrak{T}}$. Since ${\bf E}_{\mathfrak{T}}$ is a purely odd algebraic superscheme, the group subfunctor $\mathbb{G}_{ev}\mathfrak{T}$ is a locally algebraic group superscheme. 

It remains to show that $\mathbb{G}=\mathbb{G}_{ev}\mathfrak{T}$. It is clear that the latter is equivalent to the following :  for any superalgbra $R$ and for any $\phi\in \mathbb{G}(R)$ there is $\psi\in \mathbb{G}_{ev}(R)$ such that 	$\mathbb{G}(\pi_R)(\phi)=\mathbb{G}_{ev}(\pi_{R_0})(\psi)$.	On the other hand, both $\mathbb{G}$ and $\mathbb{G}_{ev}\mathfrak{T}$ satisfy the conditions of Lemma \ref{Lemma 1.2}, hence one needs to consider only a superalgebra $R$ that is finitely dimensional and local. 
In this case $\mathfrak{M}_R=\mathrm{nil}(R)$, as well as $\mathfrak{m}_R=\mathrm{nil}(R_0)$, and by Cohen's theorem $\pi_R$, as well as $\pi_{R_0}$, are both split, whence $\mathbb{G}(\pi_R)$ and $\mathbb{G}_{ev}(\pi_{R_0})$ are surjective. Proposition is proved.
\end{proof}
\begin{tr}\label{main}
Let $X$ be a proper superscheme. Then the group functor $\mathfrak{Aut}(\mathbb{X})$ is a locally algebraic group superscheme.	
\end{tr}
\begin{proof}
 Combine Proposition \ref{super criterion} with Proposition \ref{Property P_3}, Proposition \ref{P_2}, Lemma \ref{P_4 and P_5}, Lemma \ref{even P_1} and Theorem \ref{even is a group scheme}.
\end{proof}	
	Let us give few comments on the structure of $\mathfrak{Aut}(\mathbb{X})$, for which we refer to \cite{maszub2}. As we have alreay noted, the formal neighborhood of the identity of $\mathfrak{Aut}(\mathbb{X})$ coincides with $\mathfrak{T}$. 
	
	By \cite[Lemma 9.7]{maszub2} the conjugation action of $\mathfrak{Aut}(\mathbb{X})$ on $\mathfrak{T}$ factors through the \emph{largest affine quotient} $\mathfrak{Aut}(\mathbb{X})^{aff}$ of $\mathfrak{Aut}(\mathbb{X})$. In the strict sense of the word, this is not a quotient, but the universal morphism
$\pi_{aff} : \mathfrak{Aut}(\mathbb{X})\to \mathfrak{Aut}(\mathbb{X})^{aff}$ of group superschemes, such that any morphism of $\mathfrak{Aut}(\mathbb{X})$ to an affine group superscheme factors through $\pi_{aff}$. 
	
For any superalgebra $R$, the group
	$\mathfrak{Aut}(\mathbb{X})^{aff}(R)$ acts on $A\otimes R$ by $R$-linear Hopf superalgebra automorphsms, preserving the $\mathfrak{M}_A\otimes R$-adic filtration. 
	In turn, the Hopf superalgebra pairing $< , >$ induces the action of
	$\mathfrak{Aut}(\mathbb{X})^{aff}(R)$ on $\mathrm{hyp}(\mathfrak{T})\otimes R$ by $R$-linear Hopf superalgebra automorphisms, preserving the filtration $\{\mathrm{hyp}_k(\mathfrak{T}) \}_{k\geq 0}$ as well (see \cite[Lemma 9.8]{maszub2}).  
	In particular, the adjoint action of $\mathfrak{Aut}(\mathbb{X})(R)$ on $\mathrm{Lie}(\mathfrak{Aut}(\mathbb{X}))(R)$
	is naturally identified with the conjugation action of $\mathfrak{Aut}(\mathbb{X})(R)$ on $\mathrm{hyp}_1(\mathfrak{T})^+\otimes R\simeq \mathrm{Lie}(\mathfrak{Aut}(\mathbb{X}))(R)$ (cf. \cite[Corollary 9.9]{maszub2}).
	
	\section{Extensions of morphisms}
	
	The aim of this section is to superize some fundamental results from \cite[III]{sga 1}. Based on these results, we formulate a sufficient condition for $\mathfrak{Aut}(\mathbb{X})$ to be a smooth group superscheme.
	
	Let $Y$ be  a superscheme. With any quasi-coherent sheaf of $\mathcal{O}_Y$-superalgebras $\mathcal{A}$ one can associate a $Y$-superscheme $\mathrm{SSpec}(\mathcal{A})$, called a \emph{spectrum} of $\mathcal{A}$ (see \cite[Exercise II.5.17(c)]{hart}). More precisely, $|Y|=|\mathrm{SSpec}(\mathcal{A})|$ and for any affine open supersubscheme $U$ in $Y$ set $\mathrm{SSpec}(\mathcal{A})|_U=\mathrm{SSpec}(\mathcal{A}(U))$. It is easy to check that the data $\{\mathrm{SSpec}(\mathcal{A}(U)), \mathrm{SSpec}(\mathcal{A}(U\cap U'))\}_{U, U'}$ satisfies the conditions of \emph{Glueing Lemma} (cf. \cite[Exercise II.2.12]{hart}), hence there is a superscheme $\mathrm{SSpec}(\mathcal{A})$, that is glued from all $\mathrm{SSpec}(\mathcal{A}(U))$. 
	Finally, gluing together the affine morphisms $\mathrm{SSpec}(\mathcal{A}(U))\to U$ we obtain the unique affine morphism $\mathrm{SSpec}(\mathcal{A})\to Y$. 
	\begin{lm}\label{extension of spectrum}
		For any morphism $f : X\to Y$ of superscheme and for arbitrary quasi-coherent sheaf of $\mathcal{O}_Y$-superalgebras $\mathcal{A}$, there is $\mathrm{SSpec}(\mathcal{A})\times_Y X\simeq \mathrm{SSpec}(f^*\mathcal{A})$.  
	\end{lm}
	\begin{proof}
		The superscheme $\mathrm{SSpec}(\mathcal{A})\times_Y X$ is covered by open affine supersubschemes
		$\mathrm{SSpec}(\mathcal{A})|_V\times_V W\simeq \mathrm{SSpec}(\mathcal{A}(V))\times_V W$, where $V$ runs over open affine supersubschemes of $Y$ and $W$ runs over open affine supersubschemes of $f^{-1}(V)$. It remains to note that 
		\[\mathrm{SSpec}(\mathcal{A}(V))\times_V W\simeq \mathrm{SSpec}(\mathcal{A}(V)\otimes_{\mathcal{O}(V)} \mathcal{O}(W))\simeq \mathrm{SSpec}(f^*\mathcal{A}(W)).\]
	\end{proof}
	
	Let $X$ and $Y$ be superschemes over a superscheme $S$. Let $Y_0$ be a closed supersubscheme of $Y$ that is determined by a superideal sheaf $\mathcal{J}$ such that $\mathcal{J}^2=0$. 
	For any $S$-morphism $g_0 : Y_0\to X$ define a sheaf $\mathcal{P}(g_0)$ on the topological space $|Y|$, such that for any $U\subseteq |Y|$ the set  $\mathcal{P}(g_0)(U)$
	consists of all $S$-morphism $g : U\to X$ such that $g|_{U\cap Y_0}=g_0|_{U\cap Y_0}$.
	
	The following proposition is a superization of \cite[Proposition III.5.1]{sga 1}. Since the proof is just outlined therein, for the reader convenience we reproduce it with few more comments in the wider super context. For the notion of a \emph{principal homogeneous sheaf} or \emph{principal homogeneous torsor} we refer to the introduction of \cite[III.5]{sga 1}.
	\begin{pr}\label{extensions1}
		$\mathcal{P}(g_0)$ is a principal homogeneous sheaf with respect to the natural action of the additive group sheaf $\mathcal{G}=\mathrm{Hom}_{\mathcal{O}_{Y_0}}(g^*_0\Omega^1_{X/S}, \mathcal{J})$. 	
	\end{pr}
	\begin{proof}
		Without loss of generality, one can suppose that $U=|Y|$. Fix some $g\in\mathcal{P}(g_0)(|Y|)$. The $S$-morphisms  $g'\in\mathcal{P}(g_0)(|Y|)$ are in one-to-one correspondence with $S$-morphism $h : Y\to X\times_S X$ such that $\mathrm{pr}_1 h=g$ and $hi=(g_0, g_0)$, where $i : Y_0\to Y$ is the correponding closed immersion and $(g_0, g_0)$
		is the unique morphism $Y_0\to X\times_S X$ with $\mathrm{pr}_j(g_0, g_0)=g_0, j=1,2$. This  data can be represented locally by the commutative diagrams
		\[\begin{array}{ccc}
			\mathcal{O}(U)\otimes_{\mathcal{O}(V)}\mathcal{O}(U) & \stackrel{(g_0, g_0)^{\sharp}}{\to} & \mathcal{O}(W)/\mathcal{J}(W) \\
			\uparrow &  \searrow & \uparrow \\
			\mathcal{O}(U) & \stackrel{g^{\sharp}}{\to} & \mathcal{O}(W) 	
		\end{array}, \]
		where $(g_0, g_0)^{\sharp}(a\otimes b)=h^{\sharp}(a\otimes b)+\mathcal{J}(W)$, the middle diagonal map is defined by $h^{\sharp}(a\otimes b)=g^{\sharp}(a)g'^{\sharp}(b), a, b\in \mathcal{O}(U)$,
		$U$ is an open affine supersubscheme of $X$, $W=|g|^{-1}(U)\cap |g'|^{-1}(U)$ and $V$ is the preimage of $U$ in $S$. 
		
		Since $g^{\sharp}=g'^{\sharp}$ modulo $\mathcal{J}(W)$, $h^{\sharp}(\mathcal{I}_X(U\times_V U))\subseteq\mathcal{J}(W)$, that implies $\mathcal{I}_X(U\times_V U)^2\subseteq\ker h^{\sharp}$, that is $h^{\sharp}$ and $(g_0, g_0)^{\sharp}$ factor through $1$-st neighborhood of the diagonal morphism $\delta : X\to X\times_SX$. Moreover, the induced morphism
		\[\mathcal{O}(U)\otimes\mathcal{O}(U)/\mathcal{I}(U\times_V U)^2\to\mathcal{O}(U)\] splits (via $\mathrm{pr}_1^{\sharp}$), hence this neighborhood can be naturally identified with
		$X'=\mathrm{SSpec}(\mathcal{O}_X\oplus\Omega^1_{X/S})$. 
		
		Set $Y'=X'\times_X Y$ and $Y'_0=X'\times_X Y_0$. Any $S$-morphism $h : Y\to X'$, such that $\mathrm{pr}_1h=g$ and $hi=(g_0, g_0)$, is uniquely extended to the section $h' : Y\to Y'$, that is $\mathrm{pr}_2 h'=\mathrm{id}_Y$, with the additional property that $h' i$ coincides with the similar unique section $Y_0\to Y'$, induced by $hi$. The inverse of the map $h\mapsto h'$ is given by $h'\mapsto \mathrm{pr}_1 h'$.
		
		By Lemma \ref{extension of spectrum} we have $Y'\simeq\mathrm{SSpec}(g^*(\mathcal{O}_X\oplus\Omega^1_{X/S}))=\mathrm{SSpec}(\mathcal{O}_Y\oplus g^*\Omega^1_{X/S})$ and similarly, $Y'_0\simeq \mathrm{SSpec}(\mathcal{O}_{Y_0}\oplus g_0^*\Omega^1_{X/S})$. The above sections are in one-to-one correspondence with the superalgebra sheaf morphisms
		$\mathcal{O}_Y\oplus g^*\Omega^1_{X/S}\to \mathcal{O}_Y$, those are equal to the identity map, being restricted on $\mathcal{O}_Y$. Additionally, they induce the canonical augmentation $\mathcal{O}_{Y_0}\oplus g_0^*\Omega^1_{X/S}\to \mathcal{O}_{Y_0}$. Thus the concluding arguments are trivial and can be copied verbatim from \cite[Proposition III.5.1]{sga 1}.
	\end{proof}	
	Let $X$ be a proper superscheme. Let $A$ be a superalgebra and $I$ be a square zero superideal in $A$. Consider $\phi\in \mathfrak{Aut}(\mathbb{X})(A/I)$. The composition of $\phi$ with the closed embedding $X\times\mathrm{SSpec}(A/I)\to X\times\mathrm{SSpec}(A)$ is denoted by $\phi_0$.   
	
	The morphism $\phi_0$ can be extended to a $\mathrm{SSpec}(A)$-endomorphism of $X\times\mathrm{SSpec}(A)$ if and only if $\mathcal{P}(\phi_0)(|X\times\mathrm{SSpec}(A)|)\neq\emptyset$ if and only if $\mathcal{P}(\phi_0)$ is a trivial $\mathcal{G}$-torsor, where $\mathcal{G}=\mathrm{Hom}_{\mathcal{O}_{X\times\mathrm{SSpec}(A/I)}}(\phi^*_0\Omega^1_{X\times\mathrm{SSpec}(A)/\mathrm{SSpec}(A)}, \mathcal{J})$ and $\mathcal{J}=\mathcal{O}_{X\times\mathrm{SSpec}(A)}I$. Note that $\mathcal{J}$ has a natural structure of $\mathcal{O}_{X\times\mathrm{SSpec}(A/I)}$-supermodule.
	
	Recall that $\mathrm{pr}_1$ denote the canonical projection $X\times\mathrm{SSpec}(A)\to X$.
	
	Assume that $A$ is finitely generated. Then the superscheme $X\times\mathrm{SSpec}(A)$ is Noetherian and separated (use \cite[Corollary 2.5]{maszub2}). Since the superscheme $X\times\mathrm{SSpec}(A)$ is of finite type over $\mathrm{SSpec}(A)$, the sheaf $\Omega^1_{X\times\mathrm{SSpec}(A)/\mathrm{SSpec}(A)}$ of $\mathcal{O}_{X\times\mathrm{SSpec}(A)}$-supermodules is coherent (cf. \cite[Section 3]{maszub4}). By the remark $(2)$ in the end of \cite[Section 3]{zub2}, $\phi_0^*\Omega^1_{X\times\mathrm{SSpec}(A)/\mathrm{SSpec}(A)}$ is a coherent sheaf of $\mathcal{O}_{X\times\mathrm{SSpec}(A/I)}$-supermodules.  In turn, Lemma \ref{dual of coherent}
	implies that $\mathcal{G}$ is a coherent sheaf of $\mathcal{O}_{X\times\mathrm{SSpec}(A)}$-supermodules as well.
	Finally, the morphism
	$\mathrm{pr}_1$ is obviously affine, hence Corollary \ref{affine trick} implies
	\[\mathrm{H}^1(X\times\mathrm{SSpec}(A), \mathcal{G})\simeq \mathrm{H}^1(X, (\mathrm{pr}_1)_*\mathcal{G}). \]
	
	By \cite[Lemma 21.4.3]{stack} if $\mathrm{H}^1(X\times\mathrm{SSpec}(A), \mathcal{G})=0$, then all $\mathcal{G}$-torsors are trivial. Therefore, the above remarks infer that if  
	$\mathrm{H}^1(X, (\mathrm{pr}_1)_*\mathcal{G})=0$, then $\phi_0$ can be extended to a $\mathrm{SSpec}(A)$-endomorphism of $X\times\mathrm{SSpec}(A)$. 
	
	For arbitrary open affine supersubscheme $U$ in $X$ we have
	$(\mathrm{pr}_1)_*\mathcal{G}|_{U}\simeq \widetilde{M}$, where 
	\[M=\mathrm{Hom}_{\mathcal{O}(U\times\mathrm{SSpec}(A))}(\phi_0^*\Omega^1_{X\times\mathrm{SSpec}(A)/\mathrm{SSpec}(A)}(U\times\mathrm{SSpec}(A)), \mathcal{O}(U)\otimes I)|_{\mathcal{O}(U)}.\]
	Note that $V=\phi(U\times \mathrm{SSpec}(A/I))$ is an affine supersubscheme of $X\times\mathrm{SSpec}(A/I)$. Since $|X\times\mathrm{SSpec}(A)|=|X\times\mathrm{SSpec}(A/I)|$, there is an open supersubscheme $W$ in $X\times\mathrm{SSpec}(A)$ such that $V=W\cap (X\times\mathrm{SSpec}(A/I))$ and $|V|=|W|$. In other words, $V$ is defined in $W$ by the square zero superideal sheaf $\mathcal{J}\cap \mathcal{O}_W$, as well as $V_{ev}$ is defined in $W_{ev}$ by the square zero superideal sheaf $(\mathcal{J}\cap \mathcal{O}_W + \mathcal{J}_W)/\mathcal{J}_W$. Since $V_{ev}$ is affine,  \cite[Proposition II.5.9]{hart} and \cite[I, section 2, Theorem 6.1]{gabdem} imply that $W_{ev}$ is affine, hence by \cite[Theorem 3.1]{zub2} the superscheme $W$ is affine as well.
	
	Observe that if $|\phi_0|(|U\times\mathrm{SSpec}(A/I)|)\subseteq T$ for an open subset $T\subseteq |X\times\mathrm{SSpec}(A)|$, then $|W|\subseteq T$. It infers that
	\[\phi_0^*\Omega^1_{X\times\mathrm{SSpec}(A)/\mathrm{SSpec}(A)}(U\times\mathrm{SSpec}(A))\simeq\] 
	\[\Omega^1_{X\times\mathrm{SSpec}(A)/\mathrm{SSpec}(A)}(W)\otimes_{\mathcal{O}(W)} (\mathcal{O}(U)\otimes A/I).\]   
	Next, $\mathcal{O}(W)/\mathcal{O}(W) I\simeq \mathcal{O}(V)$ is isomorphic to $\mathcal{O}(U)\otimes A/I$ via the superalgebra morphism $(\phi^{\sharp})^{-1}$, whence $\mathcal{O}(V)\simeq (\phi^{\sharp})^{-1}(\mathcal{O}(U)\otimes 1)\otimes A/I$. Let $C$ denote $(\phi^{\sharp})^{-1}(\mathcal{O}(U)\otimes 1)$.
	
	We have
	\[\Omega^1_{X\times\mathrm{SSpec}(A)/\mathrm{SSpec}(A)}(W)\otimes_{\mathcal{O}(W)} (\mathcal{O}(U)\otimes A/I)=\Omega^1_{\mathcal{O}(W)/A}\otimes_{\mathcal{O}(W)} (\mathcal{O}(U)\otimes A/I). \]
	By \cite[Lemma 3.6]{maszub4} the latter supermodule is isomorphic to 
	\[\Omega^1_{\mathcal{O}(V)/(A/I)}\otimes_{\mathcal{O}(V)} (\mathcal{O}(U)\otimes A/I)\simeq (\Omega^1_{C/\Bbbk}\otimes A/I)\otimes_{(C\otimes A/I)} (\mathcal{O}(U)\otimes A/I).\]
	Thus the $\mathcal{O}(U)$-supermodule $M$ is isomorphic to $\mathrm{Hom}_{\mathcal{O}(U)}(\Omega^1_{C/\Bbbk}, \mathcal{O}(U)\otimes I)$, where $\Omega^1_{C/\Bbbk}$ has a structure of $\mathcal{O}(U)$-supermodule via the isomorphism $(\phi^{\sharp})^{-1}$, so that \[M\simeq \mathrm{Hom}_{\mathcal{O}(U)}(\Omega^1_{C/\Bbbk}, \mathcal{O}(U)\otimes I)\simeq \mathrm{Hom}_{\mathcal{O}(U)}(\Omega^1_{\mathcal{O}(U)/\Bbbk}, \mathcal{O}(U))\otimes I .\]
	In other words, $M=(\mathrm{pr}_1)_*\mathcal{G}(U)\simeq \mathcal{T}_X(U)\otimes I$. Moreover, if we choose a finite open covering $\mathfrak{U}$ of $X$ by affine supersubschemes, then
	for arbitrary $p\geq 0$ there is $C^p(\mathfrak{U}, (\mathrm{pr}_1)_*\mathcal{G}\simeq C^p(\mathfrak{U}, \mathcal{T}_X)\otimes I$, hence
	\[\mathrm{H}^1(X, (\mathrm{pr}_1)_*\mathcal{G})\simeq \mathrm{\check{H}}^1(\mathfrak{U}, (\mathrm{pr}_1)_*\mathcal{G})\simeq \mathrm{\check{H}}^1(\mathfrak{U}, \mathcal{T}_X)\otimes I\simeq \mathrm{H}^1(X, \mathcal{T}_X)\otimes I.\] 
	\begin{theorem}\label{smoothness of Aut}
		If $\mathrm{H}^1(X, \mathcal{T}_X)=0$, then $\mathfrak{Aut}(\mathbb{X})(A)\to \mathfrak{Aut}(\mathbb{X})(A/I)$ is a surjective group morphism for arbitrary superalgebra $A$, and for any nil-superideal $I$ in $A$. In particular, $\mathfrak{Aut}(\mathbb{X})$ is a smooth locally algebraic group superscheme.		
	\end{theorem}
	\begin{proof}
		Since $A$ is a direct limit of its finitely generated supersubalgebras, the same arguments as in Lemma \ref{pro-representability} show that the general case can be reduced to the case when $A$ is finitely generated and $I$ is nilpotent. Moreover, the obvious induction on the nilpotency degree of $I$ allows us to suppose that $I$ is square zero. The above discussion implies that
		for any $\phi\in \mathfrak{Aut}(\mathbb{X})(A/I)$ the corresponding $\phi_0$ can be extended to a morphism $\psi\in\mathfrak{End}(\mathbb{X})(A)$. Similarly, one can find $\psi'\in \mathfrak{End}(\mathbb{X})(A)$, that corresponds to $\phi^{-1}$. The image $\psi\psi'$ (as well as the image of $\psi'\psi$) in $\mathfrak{End}(\mathbb{X})(A/I)$  is equal to
		$\mathrm{id}_{X\times\mathrm{SSpec}(A/I)}$, hence both $\psi\psi'$ and $\psi'\psi$ are invertible. Thus our statement follows. 
		\end{proof}
	

\begin{thebibliography}{99}
		
		\bibitem{fundalg} Barbara Fantechi ... [et al.], {\em Fundamental Algebraic Geometry : Grothendieck FGA explained}, Mathematical Surveys and Monographs, no. 123.
		
		\bibitem{zubbov} V.A.Bovdi, A.N.Zubkov, {\em Orbits of actions of group superschemes}, Journal of Pure and Applied Algebra, 227(2023), 107404.
		
		\bibitem{brp} U.Bruzzo, D.H.Ruiperez and A.Polishchuk, {\em Notes on fundamental algebraic supergeometry. Hilbert and Picard superschemes}, Advances in Mathematics, 415(2023), 
		108890. 

\bibitem{mbrion} Brion, Michel, {\em Some structure theorems for algebraic groups}, Proc. Sympos. Pure Math., 94,
American Mathematical Society, Providence, RI, 2017, 53--126.
		
		\bibitem{CCF} C. Carmeli, L. Caston, R. Fioresi, {\em Mathematical Foundations of Supersymmetry}, EMS Ser. Lect. Math., European Math. Soc., Zurich, 2011.
		
		\bibitem{gabdem} M. Demazure and P. Gabriel, {\em Algebraic Groups}, Vol. I, Algebraic Geometry. Generalities.
		Commutative Groups, North-Holland, Amsterdam (1970).
		
		\bibitem{EGA I} Grothendieck, A., {\em Elements de geometrie algebrique (rediges avec la collaboration de Jean Dieudonne). I. Le langage des schemas.} (French) Inst. Hautes Études Sci. Publ. Math. No. 4 (1960), p.5-228. 
		
		\bibitem{EGA  III} Grothendieck, A., {\em Elements de geometrie algebrique (rediges avec la collaboration de Jean Dieudonne). III. Etude cohomologique des faisceaux coherentes.} (French) Inst. Hautes Études Sci. Publ. Math. No. 11 (1961), p. 5--167.
		
		
		\bibitem{sga 1} A. Grothendieck, M. Raynaud, {\em Revetements etales et groupe fondamental (SGA1)}, Soc.
		Math. de France, Paris, 2003.
		
		\bibitem{gr} A. Grothendieck, {\em Technique de descente, I, II, III, IV.}  Seminaire Bourbaki, exposes 190 (1959),
		195 (1960), 212 (1961), 221 (1961).
		
		\bibitem{hart} R. Hartshorne., {\em Algebraic geometry}, Graduate Texts in Mathematics, No. 52. Springer-Verlag,
		New York-Heidelberg, 1977.
		
		\bibitem{hmt} M. Hoshi, A. Masuoka, Y. Takahashi, {\em Hopf-algebraic techniques applied to super Lie groups over complete fields}, J. Algebra 562 (2020) 28--93.
		
		\bibitem{jan} J.C.Jantzen, {\em Representations of algebraic groups}, Second edition. Mathematical Surveys and
		Monographs, 107. American Mathematical Society, Providence, RI, 2003.
		
		\bibitem{slang} S. Lang, {\em Algebra}, Graduate Texts in Mathematics, 211, third edition, Springer, 2002. 
		
		\bibitem{lev} Levelt, A.H.M. {\em Foncteurs exacts a gauche}, Invent. Math., 8(1969), 114-140.
		
		\bibitem{Man} Yuri I. Manin, {\em Gauge Field Theory and Complex Geometry}, Translated from the 1984 Russian origi-nal by N. Koblitz and J.R. King, With an appendix by Sergei Merkulov, second edition, Grundlehren der Mathematischen Wissenschaften (Fundamental Principlesof Mathematical Sciences), vol.289, Springer-Verlag, Berlin, 1997, xii+346 pp.
				
		\bibitem{kolzub} A. Zubkov, P. Kolesnikov, {\em On dimension theory of supermodules, super-rings and superschemes}, Commun. in Algebra, 2022, Vol. 50, No. 12, 5387-5409.
		
		\bibitem{mur}  J.P.Murre,  {\em On contravariant functors from the category of pre-schemes over a field into the category of abelian groups (with an application to the Picard functor)},  Inst. Hautes Études Sci. Publ. Math. No. 23 (1964), 5–43.
		
		\bibitem{mats} H.~Matsumura, {\em Commutative algebra}, 2nd ed., Math. Lec. Notes Series \textbf{56}, Benjamin/Cunning Publishing Company, 1980.
		
		\bibitem{mats-oort} Matsumura, Hideyuki; Oort, Frans, {\em Representability of group functors, and automorphisms of algebraic schemes}, Invent. Math. 4 (1967), 1–25. 
		
		
		\bibitem{masshib} A. Masuoka, T. Shibata, {\em On functor points of affine supergroups}, J. Algebra 503 (2018) 534--572.
		
		\bibitem{maszub1} A.~Masuoka and A.~N.~Zubkov, {\em Quotient sheaves of algebraic supergroups are superschemes}. J. Algebra, 348 (2011), 135--170.
		
		\bibitem{maszub2} A.~Masuoka and A.~N.~Zubkov, {\em Group superschemes}, Journal of Algebra, 605(2022), 89--145.
		
		\bibitem{maszub3} A.~Masuoka and A.~N.~Zubkov, {\em Solvability and nilpotency for algebraic supergroups}, Journal of Pure and Applied Algebra, 221(2017), no.2, 339--365.
		
		\bibitem{maszub4} A.~Masuoka and A.~N.~Zubkov, {\em On the notion of Krull super-dimension}, Journal of Pure and Applied Algebra, 224(2020), no.5, 106245, 30 pp.
		
		\bibitem{stack} Various authors, {\em Stacks project}, https://stacks.math.columbia.edu/browse.
		
		\bibitem{zub1} A.N. Zubkov, {\em Affine quotients of supergroups},
		Transform. Groups, 2009, 14(3), 713--745.
		
		\bibitem{zub2} A.N.Zubkov, {\em Some properties of Noetherian superschemes}, Algebra and Logic, Vol. 57 (2018),
		No. 2, 130--140. (Russian Original Vol. 57, No. 2, March-April, 2018).
		
	\end{thebibliography}
\end{document}